\documentclass{amsart}
\usepackage[dvips]{color}
\usepackage{graphicx}
\usepackage{epsfig}
\usepackage{tikz}
\usepackage{enumerate}
\usepackage{enumitem}
\usepackage{color}

\usepackage{hyperref}

\usepackage{latexsym}
\usepackage{amssymb}
\usepackage{amsmath}
\usepackage{amsthm}
\usepackage{amscd}
\usepackage{latexsym}
\usepackage{amssymb}
\usepackage{amsmath}
\usepackage{amsthm}
\usepackage{amscd}
\usepackage{marginnote}
\theoremstyle{plain}
\newtheorem{thm}{Theorem}[section]
\newtheorem{lem}[thm]{Lemma}
\newtheorem{defin}[thm]{Definition}
\newtheorem{definition}[thm]{Definition}
\newtheorem{prop}[thm]{Proposition}

\theoremstyle{definition}
\newtheorem*{definition*}{Definition}
\newtheorem{conjecture}{Conjecture}
\newtheorem*{notation*}{Notation}
\newtheorem{rem}{Remark}
\newtheorem{claim}[thm]{Claim}

\newtheorem{notation}[thm]{Notation}

\newtheorem{cor}[thm]{Corollary}

\newcommand{\friend}{friends}
\newcommand{\op}{{\rm OP}}
\newcommand{\Id}{{\rm Id}}
\newcommand{\red}{{\rm red}}
\newcommand{\tred}{{\widetilde{\rm red}}}
\newcommand{\J}{{J}}

\newcommand{\bx}{{\bf x}}
\newcommand{\by}{{\bf y}}
\newcommand{\bz}{{\bf z}}
\newcommand{\bzero}{{\bf 0}}
\newcommand{\one}{{\boldsymbol{1}}}

\newcommand{\CH}{{\mathcal H}}

\newcommand{\future}[1]{{}}

\newcommand{\N}{{\mathbb N}}
\newcommand{\PP}{{\mathbb P}}

\newcommand{\BR}{{\mathcal P}}
\newcommand{\GR}{{\mathcal Q}}

\begin{document}

\title{A prime system with many self-joinings}
\author{Jon Chaika}
\address{University of Utah, Salt Lake City, UT 84112}
\email{chaika@math.utah.edu}
\author{Bryna Kra}
\address{Northwestern University, Evanston, IL 60208}
\email{kra@math.northwestern.edu}

\begin{abstract}
We construct a rigid, rank 1, prime transformation that is not quasi-simple and  whose self-joinings form a Poulsen simplex. This seems to be the first example of a prime system whose self-joinings form a Poulsen simplex. 
 \end{abstract}

\thanks{The first author was partially supported by 
NSF grants DMS-135500 and DMS-452762, the Sloan foundation, Poincar\'e chair, and Warnock chair 
and the second author by NSF grant DMS-1800544. The authors thank M.~Lemanczyk for helpful comments and questions. The authors also thank the referees for careful reading, corrections of inaccuracies, and care in guiding us towards more helpful explanations.}
\maketitle

\section{Introduction}
A natural question is to find indecomposable structures, and we study this question in the setting of measurable dynamics.  
More precisely, we consider a
 measure preserving dynamical system $(Z,\mathcal{M},\mu,T)$, where $Z$ is a set 
  endowed with a $\sigma$-algebra $\mathcal{M}$, $\mu$ is a probability measure on the 
  measure space $(Z,\mathcal{M})$, and $T\colon Z\to Z$ is a measurable transformation that preserves the measure $\mu$.  Throughout this article, we assume that  $(Z,\mathcal{M},\mu)$ is  a 
  (non-atomic) Lebesgue space. A \emph{factor} of a 
measure preserving system $(Z, \mathcal{M}, \mu, T)$ is a measure preserving system $(Z', \mathcal M ', \mu', T')$ and a 
measurable map $\pi\colon Z\to Z'$ such that $\mu\circ\pi^{-1} = \mu'$ and $T'\circ\pi(x) = \pi\circ T(x)$ 
for $\mu$-almost all $x\in Z$. 
  In this setting, the indecomposable structures  are the \emph{prime} transformations, which are transformations with no non-trivial (measurable) factors. That is, any factor map on $(Z, \mathcal{M}, \mu,T)$ is either an isomorphism or a map to a one point system.  Historically, showing systems are prime has largely been accomplished by understanding the self-joinings of the system, that is,  the $T \times T$ invariant measures on $Z \times Z$ with marginals $\mu$ on each of the coordinates. 
  Our main result is that there exists a prime transformation with many self-joinings (the self-joinings form a Poulsen simplex) and the self-joinings  can be large (there is a self-joining that does not arise as a distal extension of the system): 
 
\begin{thm}
\label{theorem:main}
There exists a prime system $(Y, \mathcal B, \nu,T)$ that is rank 1, rigid, and has an ergodic self-joining $\eta$, which is not the product measure, 
such that $(Y\times Y, \mathcal B\times \mathcal B, \eta,T\times T)$ is not a distal extension of $(Y, \mathcal B, \nu,T)$.
 Moreover, the set of self-joinings of $Y$ is a Poulsen simplex. 
\end{thm}
To highlight the novelty of our construction,  we note that being not quasi-simple (or being not quasi-distal) is a residual property in the space of measure preserving transformations (endowed with the weak topology). 
This answers a question posed by Danilenko~\cite[Section 7, Question (iii)]{danilenko} who asked if the set of quasi-simple transformations and the set of \emph{distal-simple} transformations are both meager. It is also  a strengthening of a result of Ageev~\cite{ageev} who showed that being simple is meager.

\subsection{Context of the results} 
The first systematic family of prime systems was introduced by Rudolph~\cite{rudolph}, based on Ornstein's counterexample machinery,
 and has been studied extensively since; for example, see~\cite{   dJ, dJRS, king2, Join 3, weak lim chacon, ryzhikov bounded}.
A system $(Y, \nu, T)$ has \emph{$2$-fold minimal self-joinings} 
if all of the ergodic self-joinings are either $\nu \times \nu$ or are concentrated on the graph 
$\{(x,T^jx)\}$ 
for some integer $j\geq 0$.  
Defining the natural generalization for $k$-fold minimal self-joinings for all $k\geq 2$, 
Rudolph showed that 
any system having minimal self-joinings is prime. 

However, having minimal self-joinings is quite special, and so there was interest in more general criteria for 
obtaining prime systems.  
 In this direction, Veech showed that a \emph{2-simple} system is prime if it has no compact subgroups in its centralizer. Recall that a system is 2-simple (which Veech called \emph{property S}) if the only ergodic self-joinings arise from the product measure and measures 
carried on graphs of transformations in the centralizer of the system.
Simple systems have since been studied in a variety of contexts (see for example~\cite{veech, dJR, GW, ageev, danilenko, dJ, ddJ, GHR, T}).   
Veech's criterion gave rise to the first example of a rigid prime system, 
with the construction by del Junco and Rudolph~\cite{dJR} of  a specific  rigid, simple  system that  had no non-trivial compact subgroups in its centralizer.
 Glasner and Weiss~\cite{GW} constructed an example of a prime system that is not simple, by taking a simple system and considering the factor corresponding to a non-normal maximal compact subgroup, again using Veech's criteria 
 to show that the factor is prime since it arises from a maximal compact subgroup. 
In this example, as the subgroup is not normal, the factor itself is not simple, but the self-joinings of the factor of a simple system are always isometric extensions of the factor.

There are a few other known examples of prime systems.  For example, 
King~\cite[Section 2]{king} showed that the (proper) factors of rank 1 systems are rigid and so it follows that mildly mixing rank $1$ systems are prime.   Continuing in this vein, 
Thouvenot asked if mildly mixing rank 1 transformations have minimal self-joinings, and this difficult question remains open. 
Parreau and Roy~\cite{PR} gave 
a construction of prime systems for some  Poisson suspensions of (infinite measure preserving) prime systems, 
and it follows from results in~\cite{LPR} that the constructed systems are quasi-distal. 
In the same article, Parreau and Roy write ``it is yet unknown whether prime rank one maps are always factors of simple systems." 
Our construction resolves this by producing a prime rank 1 system that is not the factor of a simple system (it is not quasi-simple).

This short list of examples basically includes all known prime systems, and one motivation for this work is to give a new construction of prime systems not relying on a paucity of joinings (as in the minimal self-joinings, simple, or factor of simple systems) or soft restrictions on the prime factors (as in the mildly mixing rank 1 or Poisson suspension of prime infinite measure preserving systems with additional properties).

Turning to the second conclusion of Theorem~\ref{theorem:main}, 
we note that it is well-known that a residual set of measure preserving systems is rank 1 and rigid. 
King~\cite{king flat} showed that for a typical measure preserving transformation, 
its self-joinings form a Poulsen simplex (recall that a \emph{Poulsen simplex} is a simplex such that the extreme points are dense).  Putting this in context,  
Lindenstrauss, Olsen, and Sternfeld~\cite{LinOlsStern} proved that a Poulsen simplex is unique up to affine homeomorphism.  
   Ageev showed that the typical transformation is not prime~\cite{ageev finite extension} and is not simple~\cite{ageev}.

\subsection{A brief outline of the paper and a conjecture}
In Section 2, we introduce general concepts from ergodic theory. In Section 3, we define our system $(Y,\nu,T)$, 
as the first return map of an odometer to a compact set, and then we set up the basic notation used throughout and prove  first results on the mixing properties of the system. In Section 4, we show that our system is not quasi-distal and that its self joinings form a Poulsen simplex. In Section 6, we show that our system is prime, building heavily 
on ingredients developed in Section 5. As our arguments are technical and require some additional development, we defer conceptual descriptions of the proofs to Sections 4 and 5, after the preliminary tools have already been defined. 

Our methods for building self-joinings and building self-joinings that can not be distal extensions of the base system are fairly soft and general (if technical and involved). 

 Our proof that the transformation is prime is more combinatorial, 
 making heavy use of the specific construction.  This should not be surprising, because being prime is a 
 meager property~\cite{ageev finite extension} in the space of measure preserving transformations (with the weak topology). Nevertheless, an ideology of this work is that it may still be a fairly common property. In particular, we conjecture that in some families of measure preserving transformations  almost every system is prime. To be specific:

\begin{conjecture}
Almost every $3$-IET is prime.
\end{conjecture}
Although this may hold more generally for a $k$-IET, such a conjecture is out of reach at this point, 
but for a $3$-IET, some of these tools are already developed with 
the methods of~\cite{CE}.

A second conjecture, closer to the work of this paper, is stated in Section~\ref{sec:construction}, after 
we have developed some further background.

\section{Definitions and notation}
\subsection{Systems and joinings}
By a \emph{measure preserving system} $(X, \mathcal{B}, \mu, T)$, we mean that  $\mathcal{B}$ is the Borel $\sigma$-algebra for some compact metric topology on $X$,   $(X, \mathcal{B}, \mu)$ is a probability space, and $T\colon X\to X$ is a measurable, measure preserving map.  Throughout the paper, we generally omit the associated $\sigma$-algebra from the notation, 
assuming that any measure preserving system is endowed with the Borel $\sigma$-algebra.  
We say that the measure preserving system $(Y, \nu, S)$ is a \emph{factor} of $(X, \mu, T)$ if there exists a measurable map $\pi\colon X\to Y$ such that $\pi\circ T= S\circ \pi$ and $\mu\circ\pi^{-1} = \nu$. 

A \emph{joining} of the  ergodic
measure preserving systems $(X_i, \mu_i, T_i)$ for $i=1, 2$ is a 
$(T_1\times T_2)$-invariant measure $\alpha$ on $X_1\times X_2$  
such that $\alpha$ projects to $\mu_1$ on the first coordinate and to $\mu_2$ on the second coordinate.  
A \emph{self-joining} of a system is a joining of two copies of the same system.  
If $(X, \mu, T)$ is a measure preserving system,   $\J(n)$ denotes the \emph{off diagonal joining} on $\{(x,T^nx)\}$, 
meaning that $\J(n)$ is the measure on $X\times X$ such that for all $f \in C(X \times X)$ 
$$
\int f(x, y)\,d\J(n) = \int f(x, T^nx)\,d\mu.
$$

If $(X, \mu, T)$ is an ergodic measure preserving system, we say that the bounded linear operator 
$P\colon L^2(\mu)\to L^2(\mu)$ is a \emph{Markov operator} if it satisfies: 
\begin{enumerate}
\item \label{it:one1}
For all $f\in L^2(\mu)$ with $f\geq 0$, we have $Pf \geq 0$ and $P^*f\geq 0$.
\item $P\one_X = \one_X$ and $P^*\one_X = \one_X$, where $\one_A$ denotes the indicator function of the set $A$.
\item \label{it:three3}
$PU_T = U_TP$,  where  $U_T\colon L^2(\mu) \to L^2(\mu)$ by $U_T f=f\circ T$.
\end{enumerate}

Markov operators can be defined more generally for an operator mapping one measure preserving system to another, but our interest is when the operator arises as an integral of fibers of a factor and so we can take the map from a system to itself; see, for example, 
Glasner~\cite{glasner-book} for more on such operators.  More precisely, 
if $(X,\mu,T)$ has a factor $(Y,\nu,S)$ 
with factor map $\pi$, then by integrating over the fibers of the factor map, we obtain a bounded linear operator $P\colon L^2(\mu) \to L^2(\mu)$, satisfying Properties~\eqref{it:one1}--\eqref{it:three3}
and we call this the \emph{Markov operator defined by $\pi$}.  That is, by disintegration of measures there exist measures $\mu_y$ on $X$ 
such that $\mu=\int_Y\mu_y d\nu$ and $P(f)(x)=\int fd\mu_{\pi(x)}$.
Note that joinings also give rise to Markov operators. However, 
these do not formally enter the arguments and so we do not discuss these Markov operators.

\subsection{Rigid rank one by cylinders}
As above, we assume that each system is endowed with its Borel $\sigma$-algebra, but we omit it from the notation.  
\begin{definition}
\label{def:rigid-rank-one}
An invertible ergodic system $(Z, \lambda, R)$, where  $Z \subset [0,1]$ and $\lambda$ denotes normalized (probability) Lebesgue measure restricted to $Z$, is {\em rigid rank one by cylinders} if  
there exist a sequence of intervals $(I_i)_{i\in\N}$, which we call \emph{cylinders},  and a sequence of positive integers $(n_i)_{i\in\N}$ such that: 
\begin{enumerate}
\item
\label{it:one}
For $0 \leq i < n_k, $  the iterate $R^iI_k$ is a cylinder having the same measure as $I_k$. 
\item
\label{it:two}
The cylinders $R^iI_k$ and $R^jI_k$ are pairwise disjoint for all $k$ and for $0 \leq i < j  < n_k$.
\item 
\label{it:three}
The measure $\lambda(\bigcup_{i=0}^{n_k-1}R^iI_k)$ tends to $1$ as $k\to\infty$. 
\item 
\label{it:four}
The ratio   $\frac{\lambda(R^{n_k}I_k\Delta I_k)}{\lambda(I_k)}$ tends to $0$ as $k\to\infty$.  
\end{enumerate}
\end{definition} 
Note that rank 1 systems can be rigid without being rigid rank 1, 
but the rigidity of a rigid rank one system is not directly tied to the towers of the system.
Also, note that cylinders in this setting are intervals in $[0,1]$, but 
we refer to them as cylinders in analogy with the symbolic setting. 
By a \emph{symbolic system} $(X, T)$, we mean an infinite sequence space $X\subset \prod_{i=1}^\infty \mathcal{A}_i$, where each $\mathcal{A}_i$ is a finite alphabet,  and $T\colon X \to X$ is a measurable map.   We denote elements of the space as $\bx = (x_i)_{i\in\N}\in X$, with the convention that a bold face letter $\bx$ has its entries denoted as $x_i$.  
In a symbolic system $X$, 
a \emph{cylinder set $[w]$ determined by a word $w= w_1\ldots w_n$} 
is defined to be
$$
[w] = \{\bx\in X\colon x_i = w_i \text{ for all } 1 \leq i \leq n\}.
$$
We also consider cylinders defined only by some entries $a_{i_1}\in \mathcal{A}_{i_1},\ldots,a_{i_k} \in \mathcal{A}_{i_k}$ defining the cylinder
$$ \{\bx \in X\colon x_{i_j}=a_{i_j} \text{ for all }1\leq j\leq k\}$$ 
and we refer to the $i_j$ as \emph{defining indices} of the cylinder. 
The collection of cylinder sets forms a basis for the topology of $X$. 
When working with a symbolic system $(X,T)$, 
fixing initial entries corresponds to an interval in $[0,1]$, 
 meaning that a cylinder set corresponds to an interval. 

The first three conditions in the definition of rigid rank one by cylinders imply that $G$ is rank one, but in the general 
setting of a rank one transformation there is no requirement that the subsets $I_i$ are intervals.  
The fourth condition gives a sequence of times under which the transformation $R$ is rigid, meaning 
that along these times the iterates of $R$ approach the identity.   Indeed, Condition~\eqref{it:four} implies that $\frac{\lambda(R^{n_k}R^iI_k \cap R^i I_k)}{\lambda(I_k)}$ is close to 1 for all large $k$ and $0 \leq i <n_k$, and so using this 
with Conditions~\eqref{it:one} and~\eqref{it:three}, we have a rigidity sequence.

\subsection{Distal extensions}
We review the definitions of  (measurable) isometric and distal extensions, as introduced by Parry~\cite{parry}.  
These extensions were key in Furstenberg's proof~\cite{Furstenberg} of Szemer\'edi's Theorem (see~\cite{furstenberg-book} for further background), and the definition we use comes from Zimmer~\cite{zimmer1, zimmer}, who showed that a measurably distal system is equivalent to a (possibly transfinite) inverse limit of a tower of isometric extensions.  

If $G$ is a compact group, $H\subset G$ is a closed subgroup, and $(X, \mu, T)$ is a Borel probability system, then 
 a measurable map $\phi\colon X\to G$ is called a \emph{cocycle} and the \emph{extension of $G$ by $G/H$ given 
by the cocycle $\phi$} is defined to be the system 
$(X\times G/H, \mu\times m_{G/H}, T_\phi)$, where 
$T_\phi(x,\tilde{g}) = (Tx, \phi(x)\cdot \tilde{g})$ for $x\in X$ and $\tilde{g}\in G/H$ and $m_{G/H}$ is the Haar measure on $G/H$ (we use the convention that cosets in $G/H$ are denoted by $\tilde{\cdot}$).   Defining the topology of the group $G$ by a distance $d_G$ that is invariant under right translation, and of course continuous with respect to translation on either side, we have an induced distance $d_{G/H}$ on $G/H$ and we have that the restriction of $T_\phi$ to each fiber of the natural projection map $X\times G/H\to X$ is continuous.  
The system $(X\times G/H, \mu\times m_{G/H}, T_\phi)$ is an \emph{isometric extension of the system $(X, \mu, T)$}. 

If $(X, \mu, T)$ and $(Y, \nu, S)$ are ergodic systems, 
then $(X, \mu, T)$ is a \emph{distal extension} of $(Y, \nu, S)$ if 
it has a sequence of factors $X_\eta$ indexed by ordinals $\eta\leq\eta_0$ for some 
countable ordinal $\eta_0$ such that $X_0 = Y$, $X_{\eta_0} = X$,  $X_{\eta+1}$ is an isometric extension of $X_\eta$ 
for each $\eta$, and for each limit ordinal $\zeta\leq\eta_0$ the system $X_\zeta$ is an inverse limit of the systems $X_\eta$ with $\eta\leq\zeta$.

\begin{notation*}
We use $d$ to denote the metric in various settings, with a subscript indicating the space as needed.  Thus $d_G$ denotes 
the right  invariant metric 
  on the group $G$, $d_{G/H}$ denotes the induced distance on $G/H$. 
  \end{notation*}

\section{Construction of the system}
\label{sec:construction}
\subsection{Definition of the transformation $T$}
We begin by constructing an odometer.
Set  
\begin{equation}
\label{def:X}
X = \prod_i\{0,\dots,a_i-1\},
\end{equation}
where 
$$
a_i = \begin{cases}
8 & \text{ if } i \notin \{10^k\colon k\geq 2 \}\\
k & \text{ if } i=10^{k} \text{ for some }k\geq 2.
\end{cases}
$$ 
We write elements $\bx\in X$ as $\bx = (x_i)_{i\in\N}$. 
Let  $S$ denote the odometer on $X$, meaning that $S$ is addition by $(1, 0, 0, \ldots)$ with carrying to the right.  Thus 
\begin{equation}
\label{def:S}
S(\bx) = S(x_1, x_2, \ldots, x_k, x_{k+1}, \ldots) = (0, 0, \ldots, 0, x_k+1, x_{k+1}, \ldots), 
\end{equation}
where $k$ is the least entry such that $x_k < a_k-1$ and if there is no such $k$, then the odometer turns over and  outputs the point $\bzero = 
(0, 0, \ldots)$.

Set
\begin{equation}
\label{def:Zk}
Z_k=\{\bx\in X\colon  x_k=7 \text{ and }x_i=a_i-2 \text{ for all }i <k\}
\end{equation}
and
\begin{equation}
\label{def:Wk}
W_k=\{\bx\in X\colon x_i=a_i-2 \text{ for all }i<10^{2k} \text{ and }x_{10^{2k}}<a_{10^{2k}}/2\}.
\end{equation}
Define 
\begin{equation}\label{eq:def-Y}
Y = X \setminus  \bigl(\bigcup_{{\ell \notin \{10^k\colon k\geq 2\}}}Z_\ell \cup \bigcup_{k=1}^{\infty} W_k\bigr)
\end{equation}
and define $T\colon Y \to Y$ to be the first return map of $S$ to $Y$.   
Throughout this paper, $T$ refers to this map and 
$d_Y$ is any metric on $Y$ giving rise to the product topology, viewing it as a subspace of $X$.  When there is no 
confusion as to which metric is meant, we omit the subscript and just write $d$ for the metric on $Y$. 
As usual, we denote elements $\by\in Y$ as $\by = (y_i)_{i\in\N}$.  

Define  $D_k$ to be the cylinder sets with largest defining index $k$ in $X \setminus Y$. More explicitly, this means that: 
\begin{equation}
\label{def:Dk}
D_k = 
\begin{cases}
Z_k & \text{ if } k\notin \{10^k\colon k\geq 2 \} \\
W_\ell & \text{ if } k=10^{2\ell} \text{ for some } \ell\geq 1 \\ 
\emptyset & \text{ if } k=10^{2\ell +1} \text{ for some } \ell\geq 1. 
\end{cases}
\end{equation}

The following result is standard: 
\begin{lem} The odometer $S$ is uniquely ergodic with respect to a probability measure $\mu$, and 
thus the first return map $T$ is uniquely ergodic with respect to the measure 
$\nu=\mu(Y)^{-1} \cdot \mu|_{Y}$. 
\end{lem}
It follows immediately from the construction of the set $Y$ that its measure is strictly between $0$ and $1$, and so the maps $T$ and $S$ are not obviously isomorphic. In fact, they are not isomorphic,  as $T$ is weakly mixing (see Proposition~\ref{prop:wmix}), while $S$ has purely discrete spectrum. 

\begin{notation*}[for the systems we study throughout this article]
Throughout this article,  $X$ is the space defined by~\eqref{def:X}, $S$ is the odometer defined on $X$ as in~\eqref{def:S}, $\mu$ is the unique ergodic measure on this system, and  $(X, \mu, S)$ 
 is the odometer system thus defined. 
The space $Y$ is defined by~\eqref{eq:def-Y} and 
$(Y, \nu, T)$ 
is the associated uniquely ergodic system 
defined by the first return map.  
\end{notation*}

Both $(X,S)$ and $(Y, T)$ are measurable maps of compact metric spaces. The remainder of this paper is devoted to studying the properties of the system  $(Y, \nu, T)$.  

\subsection{An overview of the behavior in the system $(Y, \nu, T)$}
To give an idea of what types of behaviors built into the system  $(Y, \nu, T)$ give rise to it being both prime and having many self-joinings, we summarize the types of irregularities that are built into the system in the 
construction of the towers (see~\cite{KT} for the terminology) defining the system.
Namely, there are four distinct types of irregularities: \begin{enumerate}
\item The alphabet size for the odometer is typically $8$, but at stage $n = 10^k$, the alphabet has size $k$. 
This changing in the size of the indices is necessary to allow enough room for the constructions.  
\item  For any $n\neq 10^k$, before stacking the $n-1$ columns to obtain an $n$-tower, we delete
 a positive fraction (we fix this to be one eighth)
of the right most tower. This allows separation  of the indices in the set 
of indices for columns of atypical size, meaning those of the form $10^k$ for some $k\geq 1$.
\item  For $n= 10^{2k}$, we have $2k$ Rokhlin towers and we remove one level from each of the first $k$ of them and none from the other $k$. This is used in our construction of joininngs (in the language of~\cite{KT}, 
the joinings are built using that the system has
good linked approximation of type $(m,m+1)$).
\item  For $n =10^{2k+1}$, before stacking $n-1$ columns to obtain an $n$-tower, we make no change (meaning no deletion). This allows us to use results of Chaika and Eskin (Theorem~\ref{th:CE}) and King (Theorem~\ref{th:king}) to ensure that we obtain a system that is rigid rank $1$.
\end{enumerate}

\subsection{A further conjecture} 
Maintaining the notation of this section, we state a conjecture closely related to this subject: 
\begin{conjecture}
Let $a_1, a_2, \ldots \in \mathbb{N}$ with $a_i\geq  2$ for all $i \in\mathbb N$.  Let 
$T$ be the corresponding odometer viewed as a measure preserving map of $[0,1]$, 
meaning that if $x =\sum_{j=1}^{\infty}b_j \frac 1 {a_1\cdot\ldots\cdot a_j}$ with $b_i\in \{0,\ldots,a_i-1\}$, 
then  $Tx= \sum_{j=1}^{\infty}c_j \frac 1 {a_1\cdot\ldots\cdot a_j}$ where $c_k=b_k+1$ if $k=\min\{j\colon b_j<a_j-1\}$, $c_i=0$ for all $i<k$ and $c_i=b_i$ for all $i>k$. 
For almost every $x\in [0,1]$, the first return map of $T$ to $[0,x]$ is prime. 
\end{conjecture}

\subsection{Weak mixing of the transformation $T$}
Our first goal is to show that the transformation $T$ is weakly mixing, and we start with  a sufficient (but not necessary) condition for a transformation to be weakly mixing.
\begin{lem}\label{lem:wmix crit} 
Assume that $(Z_1, \lambda, T_1)$ is an ergodic measure preserving system with respect to the Lebesgue measure $\lambda$.
If there exist a constant $c>0$,  a sequence of integers $(n_i)_{i\in\N}$, and 
sequences of measurable sets $(A_i)_{i\in\N}$ and $(B_i)_{i\in\N}$ such that 
\begin{enumerate}
\item the measures $\lambda(A_i),\lambda(B_i)>c$ for all $i\in\N$, 
\item the limit $\underset{i \to \infty} \lim \int_{A_i}|T_1^{n_i}x-x|\,d\lambda(x)=0$, and
\item  the limit $\underset{i \to \infty} \lim \int_{B_i}|T_1^{n_i}x-T_1x|\,d\lambda(x)=0$, 
\end{enumerate}
then $T_1$ is weakly mixing.
\end{lem}
\begin{proof} 
Assume that $f$ is 
an eigenfunction of $T_1$ with eigenvalue $\gamma\neq 1$. By Lusin's Theorem, 
for every $\varepsilon>0$ there exists $\delta>0$ and a measurable set $U$ with $\lambda(U)>1-\varepsilon$ such 
that if $|x-y|<\delta$ and $x,y \in U$, then $|f(x)-f(y)|<\varepsilon$. 
Choose $\varepsilon<\min \{ \frac {1-|\gamma|} 9, \frac c 9\}$, where $c$ is the constant given in the statement. 
For all sufficiently large $i\in\N$, 
by hypothesis there exists a measurable set $A_i'$ with measure at least $\frac c 2$ 
and integer $n_i$ such that if $x \in A_i'$, 
then $|T_1^{n_i}x-x|<\varepsilon$. It follows that there exists $x \in A_i'\cap U$ and $T_1^{n_i}x\in U$ and so
$$|f(x)-f(T_1^{n_i}x)|=|(1-\gamma^{n_i})|\cdot |f(x)|=|1-\gamma^{n_i}|<\varepsilon.
$$
Similarly there exists $y\in B_i \cap U$ such that $|f(y)-T_1^{n_i+1}y|=|1-\gamma^{n_i+1}|<\varepsilon$. 
If these two inequalities hold simultaneously, this contradicts the choice of $\varepsilon$, 
and so $\gamma=1$. 
Since $T_1$ is ergodic, it follows that $f$ is constant almost everywhere
and so $T_1$ is weakly mixing.
\end{proof}

Set  
\begin{equation}
\label{eq:qi}
q_i=\prod_{j=1}^{i-1}a_j.
\end{equation} 
Then $S^{q_i}(\bx)$  fixes the first $i-1$ positions of $\bx$ and increments the entry in $x_{i}$ 
position by 1. All other entries remain the same unless the $i^{th}$ position was exactly $a_{i}-1$, 
in which case the carrying continues until this process terminates. 

Given $n\in\N$, we choose $c_i(n)$ such that 
\begin{equation}
\label{eq:ci}
n=\sum c_{i}(n)q_i \quad\text{ with } c_i(n) \in \bigl\{-\frac{a_{i+1}}{2}, \ldots, \frac{a_{i+1}}2\bigr\}.
\end{equation} 
Note that there is no unique choice of these coefficients, but 
we can make a canonical choice by using the greedy algorithm to define the coefficients $c_i$. That is, 
we choose $i$ and $c_i$ such that $|n-c_iq_i|$ is minimal out of all possible $i \in \mathbb{N}$ and $c_i\in \bigl\{-\frac{a_{i+1}}{2}, \ldots, \frac{a_{i+1}}2\bigr\}$, and then iteratively choose the next coefficient 
to be the maximal choice satisfying these conditions.
If there is a tie, that is if $|n-c_jq_j|=|n-c_{j'}q_{j'}|$ is minimal, we choose $i=\min\{j,j'\}$.  Once such a representation is fixed, our construction depends on this choice. 

We define two functions from $\mathbb{Z}$ to itself that allow us  to move between studying properties of the odometer $S$ and those of the first return map $T$:  
\begin{notation*} We introduce two functions to relate powers of $T$ and $S$. These are useful in arguments throughout the paper, most immediately in the proof that $T$ is weakly mixing (Proposition 3.4) below. 
For $\by\in Y$, define $\zeta_{\by}\colon \mathbb Z\to\mathbb Z$ to be the map taking the integer $n$ to 
the integer $m$ such that $S^m\by=T^n\by$. 

For $\by \in Y$, define $\xi_{{\by}}\colon \mathbb Z\to\mathbb Z$ 
to be the map taking the integer $n$ to the least integer $m$ such that there exists $\ell\geq n$ satisfying $T^m\by=S^\ell \by$.

Let $\bzero \in Y$ denote the point consisting of 
all $0$'s.   To keep track of the iterates of $S$ that fix the first $i$ positions, as determined by the $q_i$ defined in~\eqref{eq:qi} and the expansion of any integer in the base determined by the sequence $q_i$, as defined in~\eqref{eq:ci}, we  define 
\begin{equation}
\label{def:ri}
{r}_i=\xi_{\bzero}(q_i)
\end{equation}  
and define
\begin{equation}
\label{def:di}
{d_i}(n)=c_i(\zeta_{\bzero}(n)).  
\end{equation}
\end{notation*}
Thus the map $\zeta_{\by}$ maps an iterate of $T$ to an iterate of $S$ and the coefficients $c_i$ are 
changed into $d_i$, while the map $\xi_{\by}$ reverses this, taking an iterate of $S$ to an iterate of $T$.  
However they are not precisely inverses, as one can not regain all of the odometer $S$ from the first return $T$:  if $S^i(\bx) \notin Y$, then there is no corresponding $T$ time.

\begin{rem}
 Our arguments require understanding the dynamics of $T$ both at specified times and at arbitrary times. 
Starting with the proof of Proposition~\ref{prop:wmix}, we make use of the $r_j$ 
to choose powers of $T$ with desired dynamical properties, and the indices $i$ in the criterion for weak mixing given in  Lemma~\ref{lem:wmix crit} are chosen to be $r_j$ for some appropriately chosen $j$. 
These $r_j$ are then used to select powers of $T$ with desired dynamical properties 
throughout Section~\ref{sec:joinings}. The  $d_j$ (especially for the largest $j$ such that $d_j$ is non-zero) 
are useful for understanding the dynamics of $T$ at arbitrary times. To motivate this, informally the $d_j(n)$ give us a representation of $n$ in some base constructed to be compatible with the dynamics of $T$. This role is analogous 
to how the $c_i$ act like such a base for the odometer $S$; the construction of the $d_i$ 
depend on the $c_i$, and they play such a role for $T$, and this role explored and exploited in Sections~\ref{sec:coding} and~\ref{sec:end}. 
\end{rem}

An easy analysis of the return times for the odometer $S$ leads to (we omit the proof):
\begin{lem}\label{lem:same hit}  If $\mathcal{C}$ is a cylinder defined by positions $<i$, then the sum $\sum_{j=0}^{q_i-1}\one_{\mathcal{C}}(S^j\bx)$ does not depend on $\bx\in X$. 
Thus the sum $\sum_{j=0}^{q_i-1}\one_{Z_\ell}(S^j\bx)$ does not depend on $\bx\in X$ for any $\ell$ such that $\ell<i$ and similarly, $\sum_{j=0}^{q_i-1}\one_{W_\ell}(S^j\bx)$ does not depend on $\bx\in X$ for any $\ell$ such that $10^{2\ell}<i$.  
\end{lem}

We use this to show: 
\begin{prop}\label{prop:wmix}
The system $(Y, \nu, T)$  is weakly mixing. 
\end{prop}

\begin{proof}
Assume $i=10^k-k$ and 
set 
$$U_i= X \setminus\Bigl(\bigcup_{j=0}^{q_i} S^{-j}\bigl(\bigcup^\infty_{m  =i+1} Z_{ m} \cup \bigcup^\infty_{10^j>i}W_j\bigr)\Bigr)^c.
$$   
We claim that if $i\geq 10^8$, 
then $\mu(U_i)\geq 1-\frac 1 8-\sum_{ m\geq i+2}\frac 1 {8^{m-i+1}}$.  
Indeed, under this assumption $a_{ m}\geq 8$ and so $\mu(W_{ m})=\frac{ m}{q_{10^{2{m}}+1}}$ for ${ m} \geq i$ and $\mu(Z_{\ell})=\frac 1{q_{{\ell}+1}}$
 for all ${\ell}\geq i$.  Thus, $q_i\mu(Z_{\ell})\leq 8^{{i-\ell}-1}$ 
 for all ${\ell}\geq i$. 
By the assumption on $i$, it follows that $\sum_{\{j\colon 10^{2j}> m\}} q_i \frac{j}{q_j}<\frac 1 8.$

Set 
$$A_i=\{\bx \in U_i \cap Y\colon  x_i\leq 4 \}  
$$
and so $\nu(A_i)>\frac 1 2 -\frac 1 8$. 
For $x\in A_i$, we have $S^{j}\bx\notin D_\ell$  for any $0\leq j\leq q_i$ and $\ell\geq i$  (recall that the sets 
$D_k$ are defined in~\eqref{def:Dk} and $(S^j\bx)_i\neq a_i-2$). 
Thus by Lemma~\ref{lem:same hit} and the definition of $T$,  
and  $r_i\in \mathbb Z$ such that 
$T^{r_i}\bx=S^{q_i}\bx$, which by choice of $q_i$ is close to $\bx$ (note that $r_i$ is defined in~\eqref{def:ri}). 
Set 
\begin{multline*}B_i=\{\bx\in U_i\colon 
 x_i=7, x_{1}<5, x_{i-1}<6\}{\supset} \\
 \{\bx \in X\colon x_{i+1}< 5, x_i=7, x_{i-1}<6, x_1<5\}. 
\end{multline*} 
Then $\mu(B_i)>\frac 1 {64}$ and  $\nu(B_i\cap Y)>\frac 1 {128}$. 
For $\bx\in B_i$, we have 
 $S^{j}\bx\in D_i=Z_{10^k-k}$ for some $0<j<{q}_i$ and by definition $S^j\bx \notin D_\ell$ for all $\ell>i$ (because $(S^j\bx)_i\neq a_i-2$).  Thus Lemma~\ref{lem:same hit} implies  that $T^{r_i}\bx=S^{q_i+1}\bx$ (by our assumption that $x_1<5$,  we have $S^{q_i+1}\bx \in Y$).  
 Thus  the assumptions of 
Lemma~\ref{lem:wmix crit}  are verified for the measurable sets $A_i, B_i$,  and 
sequence of integers $n_i=r_i$ with $i  \in \{10^k-k\colon k\geq 8\}$.
\end{proof}

\subsection{$T$ is rigid rank one by cylinders}
\label{sec:rigid-rank-one}
We now show that the constructed system is rigid rank one by cylinders, using information on the odometer system $(X,S)$ to study the system $(Y,T)$. 
Recall that since the system $(X,S)$ is an odometer, fixing initial entries corresponds to an interval in $[0,1)$.  

\begin{lem}\label{lem:srank 1}
The system $(Y, \nu, T)$ is rigid rank one by cylinders. 
\end{lem} 
\begin{proof}
Let $I_k$ be the cylinder set determined by the word of all $0$'s up to $10^{2k+1}$ and with any value 
between $0$ and $2k+1-5 = 2k-4$ in the entry at $10^{2k+1}$. Let $n_k= r_{10^{2k+1}}$, as defined 
in~\eqref{def:ri}.  
If $\by\in\bigcup_{i=0}^{2n_k-1}T^i(I_k)$,  then $y_{10^{2k+1}}<2k+1-3$ and so $(S^i\by)_{10^{2k+1}}<2k+1-2$ for $0\leq i\leq n_k$. Thus,  $S^i\by \notin \bigcup_{\ell\geq 10^{2k+1}}D_\ell$  for $0\leq i\leq n_k$ and so 
$$
\bigcup_{i=0}^{2n_k-1}T^i(I_k) \cap \Bigl(\bigcup^\infty _{\ell=10^{2k+1}+1}Z_\ell\cup \bigcup^\infty_{j=k+1} W_j\Bigr)=\emptyset.
$$ 
Additionally, by Lemma~\ref{lem:same hit} we have that $\sum_{j=0}^{q_{10^{2k+1}}-1}\one_{\cup_{\ell < 10^{2k+1}}D_\ell}S^j\bx$ is constant on $X$. 
Therefore $\xi_\by(q_{10^{2k+1}})$ is constant (and equal to $n_k=r_{10^{2k+1}}$) on this set.

 For any $\bx\in I_k$, 
 we have that $(T^{n_k}(\bx))_i=  x_i $ for all $i\neq 10^{2k+1}$ and $(T^{n_k}(\bx))_{10^{2k+1}} =
  x_{10^{2k+1}}+1$. 
Thus 
$$\mu(T^{n_k}I_k\cap I_k)=(1-\frac 1 { 2k-4 })\mu(I_k), $$
 establishing condition~\eqref{it:four} (after passing from $\mu$ to $\nu$) of the definition of rigid rank one by cylinders.   
 For any $\bx\in I_k$ and $0<i<n_k$, 
 we have $T^i(\bx)_j\neq0$ for  some $j<10^{2k+1}$, 
 and so condition~\eqref{it:two} follows. 
 Since each $T^iI_k$ is either contained in or is disjoint from $Z_\ell$ and $W_s$ for $\ell<10^{2k+1}$ and $s<2k+1$, 
 and furthermore is disjoint from all other $Z_\ell$ and $W_s$, 
 we have that $T^iI_k$ is a cylinder set for all $0\leq i<q_k$, establishing condition~\eqref{it:one}. 
Finally condition~\eqref{it:three} follows since  $\bigcup_{i=0}^{n_k-1}T^iI_k$ 
contains all of ${Y}$ other than the cylinder sets  defined by having entries at least $2k-4$ in the position $10^{2k+1}$. \end{proof}

\section{Joinings}
\label{sec:joinings}

In this section we prove that our system is not quasi-distal and that the self-joinings of the system form a Poulsen simplex. We start by proving Theorem~\ref{thm:no compact}, a general 
criterion for a system to not be quasi-simple. As simple extensions arise via quite a general construction, it 
is natural that this argument becomes technical. In Sections~\ref{sec:not qs} and~\ref{sec:big join proof}, we show 
that our system $(Y,\nu,T)$ verifies the assumptions of Theorem~\ref{thm:no compact}. The key results 
used for doing this are Proposition~\ref{prop:prelimit omnibus} and Lemma~\ref{lem:close good}, 
and we include a paragraph after Proposition~\ref{prop:prelimit omnibus} for a description of its role. 
The motivating idea behind the proof of Proposition~\ref{prop:prelimit omnibus} comes from a modification 
of a construction of the first named author and Eskin~\cite[Section 3]{CE}, and 
in Section~\ref{sec:first r}, we verify that  our system $(Y,\nu,T)$ satisfies the assumptions of the 
 construction. 
 Lemma~\ref{lem:close good} is general. 
 The fact that our joinings form a Poulsen simplex is analogous to  the previously mentioned construction in~\cite{CE} 
and is established in Section~\ref{sec:Poulsen} using only the results from Section~\ref{sec:first r} (and in particular does not require Proposition~\ref{prop:prelimit omnibus}). Section~\ref{sec:residual} establishes that these properties are residual. 

\subsection{Isometric and distal extensions}
\label{sec:iso-and-distal}
Given systems $(Z_1,\zeta_1, T_1)$ and $(Z_2,\zeta_2, T_2)$, if $\eta$ is a measure on $Z_1\times Z_2$, we make a mild abuse 
of notation and let $\eta_x$ denote the measure on $Z_2$ that is defined for almost all $x\in Z_1$ by disintegrating the measure $\eta$ on the fiber $\{x\}\times Z_2$. 
 We want to have a condition to rule out that $(Z_1 \times Z_2, \eta, T_1\times T_2)$
 is measurably isomorphic to $T_{\phi}\colon  Z_1  \times G/H \to Z_1 \times G/H$ by $T_{\phi}(x,g)=(Tx,[\phi(x)]g)$. Note that the change in the second fiber of such a map is independent of $g$ (but may  depend on $x$). 
Theorem~\ref{thm:no compact} is the tool to do this, 
and we give a rough idea how the various the conditions 
in the hypotheses play different roles.  Condition~\eqref{cond:stay}  identifies what the 
change in the second fiber must be (note that it is allowed to depend on $x$) and condition~\eqref{cond:move} says that this can not be the change. 
Since our isomorphism is only a measurable map, conditions~\eqref{cond:big set},~\eqref{cond:eta sees} and~\eqref{cond:extra} 
are to allow us to be able to apply Lusin's Theorem. 

\begin{thm}\label{thm:no compact}
Assume $(Z_1,\zeta_1, T_1)$ and $(Z_2,\zeta_2, T_2)$ are ergodic, Borel probability systems such that 
$Z_1$ and $Z_2$ are compact metric spaces. Let 
$\eta$ be an ergodic joining of the systems $(Z_1,\zeta_1, T_1)$ and $(Z_2,\zeta_2, T_2)$, and let
$c>0$.  
Assume that there exists $\hat{\delta}>0$, a sequence of integers $(n_i)_{i\in\N}$ 
 tending to infinity,  a sequence of integers $(L_i)_{i\in\N}$ such 
 that $L_i>\hat{\delta} n_i$,  and measurable sets $A_i\subset Z_1$ satisfying 
 \begin{enumerate}
 \item
\label{cond:big set}
 $\zeta_1(A_i)>c$ for all $i\in\N$.
 \end{enumerate}
Further assume that for each $x \in A_i$, 
there exist sets $C_i(x), E_i(x) 
\subset Z_2$, and $j_x\in [-n_i,n_i]$  (all depending on $x$)  satisfying the following conditions: 
 \begin{enumerate}[resume]
 \item \label{cond:stay} $\underset{x\in A_i}{\sup}\,\underset{y \in C_i(x)}{\sup}\, \frac 1 {L_i}\sum_{\ell=0}^{L_i-1}d_{Z_2}(T_2^{\ell}T_2^{n_i}y,T_2^{\ell}T_2^{j_x}y)\to 0$.
 \item \label{cond:move} For all $y \in  E_i(x) $,  we have $$\frac 1 {L_i}|\{0\leq \ell\leq L_i-1\colon d_{Z_2}(T_2^\ell T_2^{n_i} y,T_2^{\ell}T_2^{j_x}y)>c\}|>c.$$ 
 \item \label{cond:eta sees} For all $x\in A_i$, $\eta_x(C_i(x)), \eta_x( E_i(x))> c$ . 
 \item \label{cond:extra} For any $c'>0$, there exists $i_0$  such  that for all  $i\geq i_0$ and 
 any  $x\in A_i$ if we have balls $B(p_\ell,c') \subset Z_2$  satisfying  $\eta_x(E_i(x) \cap \cup B(p_\ell,c')>c-c'$ then $\eta_x(C_i(x)\cap \cup B(p_\ell,2c'))>c-2c'$. 
 \end{enumerate}
Then $\eta$ is not a distal extension of $(Z_1,\zeta_1,T_1)$.
\end{thm}
 Note that this is a general result, holding for arbitrary measure preserving systems whose underlying spaces are compact metric spaces, and this result does not depend on the particular constructions we have for the systems $(Z_1,\zeta_1, T_1)$ and $(Z_2,\zeta_2, T_2)$.
We further note that in~\eqref{cond:stay}, we can not take $j_x = n_i$, as this would preclude 
Condition~\eqref{cond:move}.  
Note that since Condition~\eqref{cond:extra} holds for arbitrarily small choice of $c'$, 
this rules out the possibility that the joining is carried on a finite union of graphs.  Indeed, if $f_1,\ldots, f_r\colon Z_1\to Z_2$ are distinct functions satisfying $f_i(T_1z)=T_2(f_iz)$, 
then for all $\varepsilon>0$ there exists $\tilde{c}>$ such that for all but a set of $z\in Z_1$ of $\mu_1$-measure at most $\varepsilon$ we have $d_{Z_2}(f_i(z),f_j(z))>\tilde{c}$ for all $i \neq j$. 
Then we can not satisfy 
Condition~\eqref{cond:extra} with $\varepsilon$ small enough and $c'<\frac{\tilde{c}}3$.

The proof of Theorem~\ref{thm:no compact} proceeds by contradiction. We assume $(Z_1\times Z_2 , \eta,T_1 \times T_2)$ is an isometric extension of $(Z_1,\zeta_1,T_1)$, meaning that there exists a (measurable) isomorphism
$\Psi\colon  (Z_1\times Z_2 , \eta,T_1 \times T_2)\to (Z_1 \times G/H, \zeta_1 \times m_{G/H},T_\phi)$ that is the identity on the first coordinate, and use this to derive a contradiction. Since a distal system can 
be decomposed as a tower of isometric extensions, we conclude that it can not be a distal extension.  

Before turning to this proof, we start with some preliminaries and a lemma used to derive the contradiction. 
 
Let $\mathcal{K}$ be a compact continuity set for $\Psi$ with $\eta(\mathcal{K})>1-\frac {\hat{\delta}} {100}c^4$.  
Thus $\mathcal{K}$ is also a continuity 
set for $\pi_2\circ \Psi$, where $\pi_2\colon Z_1\times G/H \to G/H$ is the projection on the second coordinate.
Choose $\delta>0$ such that $d_{G/H}(g\tilde{h},g\tilde{h}')<\frac c 8$ whenever $d_{G/H}(\tilde{h},\tilde{h}')<\delta$ and $g \in G$. Choose $\frac c 8>\delta'>0$ such that 
$d_{G/H}(\pi_2\circ \Psi(x,y),\pi_2\circ \Psi(x'y'))<\delta$ whenever $(x,y),(x',y')\in \mathcal{K}$ and $d_{Z_1\times Z_2}\Big((x,y),(x',y')\Big)< \delta'$. 

\begin{lem}\label{lem:cond} Under the assumptions of 
Theorem~\ref{thm:no compact}, 
there  exist a pair of points $(x,y),(x,y')\in Z_1\times Z_2$ and  $b\in \mathbb{Z}$ such that 
\begin{enumerate}
\item $(x,y),(x,y'),(T_1^bx,T_2^by)$, and $(T_1^bx,T_2^by') \in \mathcal{K}$;
\item\label{conc:very close} $d_{Z_2}(y,y')<\delta'$;
\item \label{conc:far} $d_{Z_2}(T_2^by,T_2^by')> \frac c 2-\frac c 8>\frac c 3.$
\end{enumerate}
\end{lem}
\begin{proof} For all 
 $L\geq 1$, we have that $\eta(\{(x,y)\colon \sum_{i=0}^{L-1}\one_{\mathcal{K}}\Big((T_1^ix,T_2^iy)\Big)<L-\frac {L\hat{\delta}} {10}c\})<\frac {1} {10}c^3$. 
  Choosing $c'=\frac {\delta'} 8$ 
 as in Condition~\eqref{cond:extra} of Theorem~\ref{thm:no compact}, 
 for all sufficiently large $i$, we can pick 
 $x\in A_i$, $y \in C_i(x)$,  and $y'\in E_i(x)$ 
 satisfying $d_{Z_2}(y,y')<\delta'$ and the conditions
 \begin{align*}
& \sum_{i=0}^{L_i+n_i-1}\one_{\mathcal{K}}\Big((T_1^ix,T_2^iy)\Big)>L_i+n_i-\frac {L_i} {10}c, \\
& \sum_{i=0}^{L_i+n_i-1}\one_{\mathcal{K}}\Big((T_1^ix,T_2^iy')\Big)>(L_i+n_i)-\frac {L_i} {10}c. \end{align*}
By Conditions~\eqref{cond:stay} and~\eqref{cond:move} of Theorem~\ref{thm:no compact}, there exists $\ell$ such that the points 
$(T_1^{\ell+n_i} x,T_2^{\ell+n_i}y)$, $(T_1^{\ell+j_x} x,T_2^{\ell+j_x}y)$, $(T_1^{\ell+n_i} x,T_2^{\ell+n_i}y')$, and $(T_1^{\ell+j_x} x,T_2^{\ell+j_x}y')$ all lie in the set $\mathcal{K}$, while at the same time $d_{Z_2}(T_2^{\ell+n_i}y,T_2^{\ell+j_x}y)<{\delta'}$ and $d_{Z_2}(T_2^{\ell+n_i}y',T_2^{\ell+j_x}y')>c$. 
Thus we can take $b$ to be one of $\ell+j_x$ or $\ell+n_i$. 
Indeed, 
\begin{multline*}\max\Big\{d_{Z_2}(T^{\ell+n_i}y,T^{\ell+n_i}y'),\, d_{Z_2}(T^{\ell+j_{x}}y,T^{\ell+j_x}y')\Big\}\geq \\
d_{Z_2}(T^{\ell+n_i}y',T^{\ell+j_x}y')-d_{Z_2}(T^{\ell+n_i}y,T^{\ell+j_x}y).\quad\qedhere
\end{multline*}
\end{proof}

\begin{proof}[Proof of Theorem~\ref{thm:no compact}] We first show that $\eta$ is not an isometric extension. 
Observe that if $g=\phi(T_1^{b-1}x)\cdot\ldots \cdot\phi(x)$, then $\pi_2\circ\Psi(T_1^bx,T_2^by)=g\pi_2\circ \Psi(x,y)$ and 
$\pi_2\circ\Psi(T_1^bx,T_2^by')=g\pi_2\circ \Psi(x,y').$
Because all four of 
the points $(T_1^{\ell+n_i} x,T_2^{\ell+n_i}y)$, $(T_1^{\ell+j_x} x,T_2^{\ell+j_x}y)$, $(T_1^{\ell+n_i} x,T_2^{\ell+n_i}y')$, and $(T_1^{\ell+j_x} x,T_2^{\ell+j_x}y')$ lie in the set $\mathcal{K}$, Conclusion~\eqref{conc:very close} of Lemma~\ref{lem:cond} implies that 
$d_{Z_2}(T_2^by,T_2^by')<\frac c 4$, a contradiction of Conclusion~\eqref{conc:far} of Lemma~\ref{lem:cond}.

Now assume that $\eta$ is a distal extension of $(Z_1,\zeta_1,T_1)$. 
 By the structure theorem for 
distal flows of Furstenberg~\cite{Furstenberg} and Zimmer~\cite{zimmer}, the system $(Z_1\times Z_2,\eta,T_1\times T_2)$. 
is an inverse limit of systems, each of which is an isometric extension of the preceding one. 
Thus there is a factor of our distal extension, which is an isometric extension of $(Z_1,\zeta_1,T_1)$
, and which satisfies the assumptions of Theorem~\ref{thm:no compact} (with different $c$). Indeed, by the definition of inverse limits,  
 we can embed our distal extension into the product defining the inverse limit.  
  This contradicts the previous paragraph. 
\end{proof}

\subsection{A self-joining that is not quasi-simple} \label{sec:not qs} We apply Theorem~\ref{thm:no compact} to establish part of Theorem~\ref{theorem:main}:
 \begin{thm}\label{thm:big joining} 
There exists a non-trivial ergodic  self-joining of $(Y,\nu,T)$ that is not a distal extension of $(Y,\nu,T)$. 
 \end{thm}
 By non-trivial, we mean that 
  the self-joining is not $\nu \times \nu$.

Before turning to the proof, we start with  some preliminaries.
If $(Y, T)$ is a compact metric space, 
let $\mathcal{M}(Y\times Y)$ denote the set of Borel probability measures on $Y$ and let $d_{\mathcal{M}(Y \times Y)}(\cdot,\cdot)$ denote the 
\emph{Kantorovich-Rubenstein metric}, defined for Borel probability measures $\mu, \nu\in \mathcal{M}(Y\times Y)$ 
as $$d_{\mathcal{M}(Y \times Y)}(\mu,\nu):=\sup\Bigl\{\Bigl\vert \int fd\mu-\int f d\nu\Bigr\vert\colon f \text{ is 1-Lipschitz function on } Y \times Y\Bigr\}.$$
 This metric  endows the set of Borel probability measures  $\mathcal{M}(Y\times Y)$
 on $Y\times Y$  with   the weak*-topology.  
 Similarly, define $d_{\mathcal{M}(Y)}$ to be the Kantorovich-Rubenstein metric on the $\mathcal{M}(Y)$.

Recall that $\J(n)$ denotes the off diagonal joining on $\{(\bx,T^n\bx)\}$, 
meaning that $\J(n)$ is the measure on $X\times X$ defined by 
$$
\int f(\bx, \by)\,d\J(n) = \int f(\bx, T^n\bx)\,d\mu.
$$ 

Recall that if $\sigma$ is a self-joining of $(Y,\nu,T)$, we let $\sigma_\bx$ denote the disintegration of $\sigma$ given by projection to the first coordinate, thought of as a measure on $Y$. Note that this is only defined $\nu$-almost everywhere and is slightly different than the usual disintegration of measures:  it defines a measure on $Y$, rather than a measure on $Y \times Y$ that gives full measure to $\{\bx\}\times Y$. 

The main tool in establishing 
Theorem~\ref{thm:big joining} 
is the following proposition:
\begin{prop} \label{prop:prelimit omnibus}
  For any $\varepsilon>0$ and  $k_1,\ldots,k_r \in \mathbb{Z}$, 
there exist $\ell_1,\ldots,\ell_{2r}$, $ N, M,L \in \mathbb{Z}$, with $M\leq L$,  a set $A\subset Y$ with $\nu(A)>\frac 1 {99}$, 
 and for each $\bx\in A$ 
there exists $j_\bx\in [-M,M]$ such that 
\begin{enumerate}[label=(\alph*)]
\item\label{cond:close to bary}
 $\nu\Big(\big\{ 
\bx
\colon d_{\mathcal{M}(Y)}\big((\frac 1 {2r}\sum_{n=1}^{2r}\J(\ell_{n}))_{\bx},( \frac 1 r \sum_{n=1}^r \J(k_{n}))_{\bx}\big)>\varepsilon\big\}\Big)<\varepsilon$.
(Recall our  convention that the disintegration of measure on $Y \times Y$ by projection onto the first coordinate is a measure on $Y$.) 

\item \label{cond:fiber}
\begin{multline*}
\nu\bigl(\{\bx  \in A  \colon \text{there exist reorderings }i_1,\ldots,i_r \text{ of }1,\ldots,r \text{ and }\\ i_{1+r},\ldots,i_{2r} \text{ of }r+1,\ldots,2r 
\text{ such that for all }1\leq s\leq r, \\
d_Y\big(T^{k_s}\bx,T^{\ell_{i_s}}\bx\big)<\varepsilon \text{ and } d_Y\big(T^{k_s}\bx,T^{\ell_{i_{r+s}}}\bx\big)<\varepsilon\}\bigr)>\nu(A)- \varepsilon. 
\end{multline*}
\item\label{conc:close} 
$d_{\mathcal{M}(Y \times Y)}\big(\frac{1}{N}\sum_{i=1}^N\delta_{(T^i\bx,T^iT^{\ell_{n}}\bx)}, \frac 1 r \sum_{i=1}^r \J(k_i)\big)\}\big)<\varepsilon$ for all $n\leq 2r$ and $\bx \in A$.  
\item\label{conc:stay} $\frac 1 {L} \sum_{i=0}^{L-1}d_{Y}(T^{M+i}T^{\ell_{n}}\bx,T^{i+j_\bx}T^{\ell_{n}}\bx)<\varepsilon$ for all $\bx\in A$ and $n\leq r$.
\end{enumerate}
Moreover, if we assume that there exist $a,b \in \mathbb{N}$ and $c>0$ such that $d_{Y}(T^a\bx,T^b\bx)>4c+\varepsilon$ for a set $W$ of $\bx$ with   $\nu(W) = \frac 1 2 $  and 
\begin{equation}\label{eq:to intertwine}
|\{1\leq n\leq r\colon d_{Y}(T^{k_{n}}\bx,T^a\bx)<c\}|=|\{1\leq  n \leq r\colon d_{Y}(T^{k_{n}}\bx,T^b\bx)<c\}|=\frac r 2
\end{equation}
for all $\bx \in W$, then 

\begin{enumerate}[resume, label=(\alph*)]
\item\label{conc:move} $\frac 1 {L} |\{0\leq i\leq L-1\colon d_{Y}(T^{M+i}T^{\ell_d}\bx,T^iT^{\ell_d}T^{ j_\bx}\bx)>c$ for all $\bx\in A$ and 
$r<d\leq r+ 2 \lceil  \frac 1 {16}   r \rceil \}|>\frac 1 9$.
\end{enumerate}
\end{prop}

The proof of this proposition occupies the rest of this section, starting with finding the first half of the $\ell_i$ and then the second half.  
Before we turn to this, we comment on the role that this proposition plays. 
To prove Theorem~\ref{thm:big joining}, we iteratively apply this proposition, and at the $k-1^{\text{st}}$ application obtain a joining that is the barycenter of $2^k$ 
off diagonal joinings. We then take the weak*-limit of this sequence of (non-ergodic) joinings and obtain $\sigma$, an ergodic joining that satisfies the criterion of Theorem~\ref{thm:no compact}.
Using the proposition, we obtain a joining with the desired properties before passing to a limit. 
Before turning to the proof, we give some indication of the role of the various conditions in the statement.  
Conditions~\ref{cond:close to bary} and~\ref{cond:fiber}
are used to prove   that $\sigma$ is ergodic, and
Conditions~\ref{conc:stay} and~\ref{conc:move}  are used to show that $\sigma$ satisfies the assumptions of Theorem~\ref{thm:no compact}.  
More precisely, the sets $C_i(x)$ for $\sigma$  
are approximated by $\{T^{\ell_{n}}\bx\}_{n=1}^r$  in the sense that $\eta_x$ restricted to $C_i(x)$ is close to $\frac 1 {2r} \sum_{n=1}^r \delta_{T^{\ell_n}\bx}$. Similarly, 
  the sets  $\{T^{\ell_n}\bx\}_{n=r+1}^{r+2\lceil \frac r {16}\rceil}$ correspond to $E_i(x)$,  $M$ corresponds to $n_i$, and $L$ to $L_i$. 
 Conclusion~\ref{conc:move}  is the analog of~\eqref{cond:move} in Theorem~\ref{thm:no compact} and Conclusion~\ref{conc:stay} of~\eqref{cond:stay}. Condition~\eqref{cond:extra} in Theorem~\ref{thm:no compact} corresponds to observing that~\ref{cond:fiber} implies that for most $\bx$, for any $r< d \leq r+2\lceil \frac r {16} \rceil$ there exists $1\leq i_d \leq i_r$ such that $T^{\ell_d}\bx$ is close to $T^{\ell_{i_d}}\bx$. 
 The relation between the pre-limit versions of the properties and the desired properties for the limiting measure $\sigma$ is 
 addressed in Lemma~\ref{lem:close good}.

 \subsubsection{Finding $\ell_1$,\ldots,$\ell_r$}\label{sec:first r}

We now construct $\ell_1,\ldots,\ell_r$ satisfying the conclusions of Proposition~\ref{prop:prelimit omnibus}. 
\begin{lem}\label{lem:towers shadow} For all $\varepsilon>0$, there exists $k_0\in\N$ such that for all $k>k_0$ and $1 \leq \ell\leq k$:
\begin{enumerate}
\item 
If $(T^i\bx)_{10^{2k}}<k$ for all $0\leq i\leq \ell r_{10^{2k}}$, then 
$d(T^{\ell r _{10^{2k}}}\bx,\bx)<\varepsilon.$ 

\noindent
Similarly, if  $(T^{i}\bx)_{10^{2k}}<k$ for all $0\geq i\geq -\ell r_{10^{2k}}$, then 
$d(T^{\ell r _{10^{2k}}}\bx,\bx)<\varepsilon.$
\item 
\label{item:close-in-T}
If $k\leq (T^i\bx)_{10^{2k}}<2k-2$  for all $0\leq i\leq \ell (r_{10^{2k}}+1)$, then 
$d(T^{\ell (r _{10^{2k}}+1)}\bx,\bx)<\varepsilon$. 

\noindent
Similarly, if 
 $k\leq (T^i\bx)_{10^{2k}}<2k-2$  for all $0\geq i\geq -\ell(r_{10^{2k}}+1)$, then 
$d(T^{\ell (r _{10^{2k}}+1)}\bx,\bx)<\varepsilon$. 
\end{enumerate}
\end{lem}

 Recall from Section~\ref{sec:construction} that $d=d_Y$ is a metric giving the subspace topology for $Y$ coming from product topology on $X$. 
\begin{proof} 
We only include the proof of the first part of~\eqref{item:close-in-T}, as the proofs of all four statements are similar. 
Thus we need to show that under the assumptions, $(T^{\ell (r _{10^{2k}}+1)}\bx)_i=x_i$ for all  
$i< 10^{2k}$.  
This statement immediately follows once we show that  \begin{equation}
\label{eq:close-T-and-S}
T^{\ell (r _{10^{2k}}+1)}\bx=S^{\ell q_{10^{2k}}}\bx.
\end{equation}
To prove~\eqref{eq:close-T-and-S}, note that by assumption,  
$(T^i\bx)_{10^{2k}}<2k-2$ for all $0\leq i\leq \ell (r_{10^{2k}}+1)$, 
and so 
$S^i\bx \notin \bigcup_{j>10^{2k}} D_j$ for all $0\leq i\leq \ell q_{10^{2k}}$.  (Recall that $D_j$ are defined in~\eqref{def:Dk}.) 
Similarly, by the assumption that 
$k\leq (T^i\bx)_{10^{2k}}$ for all $0\leq i\leq \ell (r_{10^{2k}}+1)$, 
we have that $S^i\bx \notin D_{ 10^{2k}}$ for all $0\leq i\leq \ell q_{10^{2k}}$. 
Thus for any such $\bx$, $\zeta_{\bx}(q_{10^{2k}})=\zeta_{\bf{0}}(q_{10^{2 \ell}})+1=r_{10^{2k}}+1$. Indeed, there exists $0< i < r_{10^{2k}}$ such that $S^i({\bf{0}})\in D_{10^{2k}}$. Iterating this process 
for $\bx$, we obtain that $S^{q_{10^{2k}}}\bx=T^{r_{10^{2k}}+1}\bx,\ldots,
S^{\ell q_{10^{2k}}}\bx = T^{\ell (r_{10^{2k}}+1)}\bx$, thus proving the claim.  (Note that by assumption, $S^{i}(T^{j (r_{10^{2k}}+1)}{\bf{x}}) \notin D_{10^{2k}}$ for any $0\leq j\leq \ell$ and $0\leq i\leq q_{10^{2k}}$.) 
If $k_0$ is large enough (depending on the metric, $d$, and $\varepsilon$), the lemma follows. 
\end{proof}

\begin{lem}
\label{lemma:bound-cylinders} Let $u,v\in \mathbb{Z}$ and 
 $n=u+vr_{10^{2k}}$. If 
 \begin{equation}\label{eq:x range 1} x_{10^{2k}} \in [|u|+|v|+1,k-|u|-|v|-1],
 \end{equation} 
 then $(T^n\bx)_{j}=(T^u\bx)_j \text{ for all }j\neq 10^{2k}$.
Similarly, if  
\begin{equation}\label{eq:x range 2} x_{10^{2k}} \in [|u|+|v|+k+1,2k-|u|-|v|-3],
\end{equation}
 then $(T^n\bx)_{j}=(T^{u-v}\bx)_j$ for all $j\neq 10^{2k}$. 
\end{lem} 
\begin{proof}  These results follow from Lemma~\ref{lem:towers shadow}, and again we only prove the first part as the others are analogous. If 
${x_{10^{2k}} \in [|u|+|v|+1,k-|u|-|v|-1]}$, then we apply the first part of Lemma~\ref{lem:towers shadow} 
with $\ell=v$ to $T^u\bx$. 
\end{proof}

\begin{cor}\label{cor:most conditions} For all $\varepsilon>0$ and $b,b' \in \mathbb{Z}$, there exists $k_0$ such that for all $\ell>k_0$ there exists $p_\ell \in \mathbb{Z}$, 
disjoint sets $A_\ell,\, B_\ell$, and a cylinder
 $J_\ell$ satisfying
\begin{enumerate}
\item\label{conc:most tracking 1} $T^{p_\ell}(\bx)_j=(T^b\bx)_j$ for all $j \neq 10^{2\ell}$ and $\bx \in A_\ell$.
\item\label{conc:most tracking 2} $T^{p_\ell}(\bx)_j=(T^{b'}\bx)_j$ for all $j \neq 10^{2\ell}$ and $\bx \in B_\ell$.
\item $\nu(A_\ell), \, \nu(B_\ell)>\frac 1 2 -\varepsilon$.
\item \label{conc:A description} $A_\ell=\bigcup_{i=0}^{\big(\ell-2(|b|-|b'|-1)\big)r_{10^{2\ell}}} T^iJ_{\ell}$
\item\label{conc:disjoint J} $T^iJ_{\ell} \cap J_{\ell}=\emptyset$ for all $0 \leq i\leq 2(\ell-|b|-|b'|-1)r_{10^{2k}}$. 
\end{enumerate}
\end{cor}
\begin{proof}
We apply Lemma~\ref{lemma:bound-cylinders}  to $n=b+(b-b')r_{10^{2 \ell}}$, and as in the lemma, we write $n=u+vr_{10^{2k}}$. (That is, $u=b$ and $v=b-b'$.) 
Since $b-(b-b')=b'$, by choosing 
$$J_{\ell}=\{\bx\colon x_j=0 \text{ for all }j< 10^{2\ell} \text{ and } x_{10^{2\ell}}=|u|+|v|+1\},$$ 
 the corollary follows with $p_\ell=b+(b-b')r_{10^{2 \ell}}$. Indeed, $A_\ell$ satisfies~\eqref{eq:x range 1} and 
 $$B_\ell= \bigcup_{i=(|u|+|v|+\ell)r_{10^{2\ell}}}^{(2\ell-|u|-|v|)r_{10^{2\ell}}}T^i J_\ell$$ satisfies~\eqref{eq:x range 2}. 
 Clearly the measures of each of these sets converge to $\frac 1 2 $ as $\ell$ goes to infinity (for fixed $u,v$). 
\end{proof}

\begin{lem}\label{lem:empirical}For any $\varepsilon>0$ and $b \in \mathbb{Z}$,  there exists $N\in\N$ such that for all  $n\geq N$, 
$$\nu(\{\bx \in Y\colon d_{\mathcal{M}(Y \times Y)}\big(\frac{1}{n}\sum_{i=1}^n\delta_{(T^i\bx,T^iT^{b}\bx)}, \J(b)\big)>\varepsilon\})<\varepsilon.$$
\end{lem}

\begin{proof} This follows immediately from the Ergodic Theorem and the compactness of the space of 1-Lipschitz functions with bounded integral. 
\end{proof}

We now combine these results with the strategy  developed in~\cite{CE} 
to build  off diagonal 
joinings that are weak-*close to the barycenter of other off diagonal joinings. 
We begin by summarizing 
the results of~\cite{CE},  
where the input is a sequence of numbers and sets with certain properties.

We assume that $c>0$,  $J_{j}$ is a sequence of  
cylinders,
 $m_j$ is a sequence of natural numbers, $b_j^{(1)},\ldots,b_j^{(d)}$ are sequences of integers, and 
$\hat{A}_j$, $\hat{B}_j$ and $U_j$ are sequences of sets, and $\varepsilon_j>0$ satisfy the following properties: 
\begin{enumerate}
\item\label{itt:one}
For all $j$, $\hat{A}_j=\bigcup_{\ell=1}^{m_j}T^\ell J_{j} \setminus U_j$. 
\item  For all $j$,  $\hat{B}_j=Y \setminus (A_j \cup U_j)$. 
\item  For all $j$, $\nu(\hat{A}_j),\nu(\hat{B}_j)>c$. 
\item\label{cond:minimal return} The minimal return time to $J_{j}$ is at least $\frac 3 2 m_j$. 
\item For all $j$,  $\nu(U_j)<\varepsilon_j$.  
\item For all $j$, $m_j \sum_{\ell = j+1}^\infty \nu(J_\ell)<\varepsilon_\ell$. 
\item For all $j$, $\varepsilon_{j+1}\leq  \varepsilon_j$ and $\sum_j \varepsilon_j<\infty$.
\item \label{cond:track}  For any $\bx \in \hat{A}_j$, we have 
$d(T^{b_{j}^{(p)}}\bx,T^{b_{j-1}^{(p-1)}}\bx)<\varepsilon_j$  
and 
for any $\bx \in \hat{B}_j$, we have $d(T^{b_{j}^{(p)}}\bx,T^{b_{j-1}^{(p)}}\bx)<
\varepsilon_j$. 
(Note that $b_{j-1}^{(p-1)}$ is interpreted to be $b_{i-1}^{(d)}$ if $p=1$.) 
\item\label{cond:kr close}$d_{\mathcal{M}(Y \times Y)}\bigl( \frac 1 L \sum_{i=1}^{L}(T \times
T)^i (\J(b_{i}^{(p)})_\bx),\J(b_{i}^{(p)})\bigr)<\varepsilon_i$ for all $\bx \in Y$, 
all $L\geq \frac{m_{k_{i+1}}}9$, and any $p \in \{1,\ldots,d\}$. 
\end{enumerate}

\begin{thm}[{Chaika-Eskin~\cite[Proposition 3.1 and (the proof of) Corollary 3.3]{CE}}]\label{th:CE}
Assuming sequences of numbers and sets satisfying~\eqref{itt:one}-~\eqref{cond:kr close}, 
there
exist $ \rho<1$, $C'>0$ (depending only on $c$ and $d$) such that 
$$d_{\mathcal{M}(Y \times Y)}\big(\J(b_{k}^{(p)}),\frac 1
r\sum^r_{p=1} \J({b_i}^{(p)})\big)\leq C'\sum_{q=i}^k\varepsilon_q 
+C'\rho^{k-i},$$   
whenever $k\geq i$ and $p\in \{1,\dots,r\}$. Moreover, if $\bx \notin  \bigcup_{q=i}^kU_q$, 
 there is a reordering (which is allowed to depend on $\bx$) $p_1,\ldots,p_d$ with $d(T^{b^{(j)}_i}\bx,T^{\hat{b}_{k}^{(p_j)}}\bx)<\sum_{q=1}^k\varepsilon_q$ for all $1\leq j\leq d$. 
\end{thm}
\begin{rem}
The last statement of this theorem is not in the statement of Corollary 3.3, but follows by  iterating~\eqref{cond:track}. 
 The condition in~\eqref{cond:track} is a slightly simpler condition than that in~\cite{CE}, 
where the conditional measure of an off diagonal joining on a fiber is used instead of the distance between points,  
but the condition in~\cite{CE} follows immediately by using the definition of the Kantorovich-Rubinstein metric.
\end{rem}
\begin{rem}We iteratively apply the result of Theorem~\ref{th:CE} for different (decreasing) choices of 
$\varepsilon_i$ and (increasing) $d$, with each choice satisfying all of the properties~\eqref{itt:one}--\eqref{cond:kr close}.  Corollary~\ref{cor:most conditions} 
is designed to ensure that conditions (i)-(ix) in the hypotheses of Theorem~\ref{th:CE} hold. Indeed conclusions~\eqref{conc:most tracking 1} and~\eqref{conc:most tracking 2} of Corollary~\ref{cor:most conditions} provide condition~\eqref{cond:track}. 
Conclusion~\eqref{conc:disjoint J}  of Corollary~\ref{cor:most conditions} provides condition~\eqref{cond:minimal return}, 
while conclusion~\eqref{conc:A description} provides condition~\eqref{itt:one}. Lemma~\ref{lem:empirical} provides condition~\eqref{cond:kr close}.
\end{rem}
\begin{cor} \label{cor:getting first r} For any $\varepsilon>0$ and integers $b_1,\ldots,b_d$, there exist integers $\hat{b}_1,\ldots,\hat{b}_d$ such that 
\begin{equation}\label{eq:getting first r}
d_{\mathcal{M}(Y\times Y)}\bigl(\J(\hat{b}_\ell),\frac 1 d \sum_{j=1}^d\J(b_j)\bigr)<\varepsilon
\end{equation}
for all $\ell\in\{1,\ldots,d\}$.  Moreover, we may assume that there is a set  $\hat{W}$ of measure $1-\varepsilon$ 
such  that for every $\bx\in \hat{W}$, 
 there is a reordering $p_1,\ldots,p_d$ with $d(T^{b_j}\bx,T^{\hat{b}_{p_j}}\bx)<\varepsilon$ for all $1\leq j\leq d$.
\end{cor}

We note  that the reordering in the second part of this statement depends on the particular $\bx$. 
\begin{proof} 
By Corollary~\ref{cor:most conditions} and Lemma~\ref{lem:empirical} , the result holds for $d=2$.  
Indeed, given $b_1,b_2$, and $\varepsilon'$, Lemma~\ref{lem:empirical}  provides $L_0$ such  that 
$$\nu(\{x\colon  d_{\mathcal{M}(Y \times Y)}
\left(\frac 1 L \sum_{i=1}^{L}\delta_{(T\times T)^i(\bx,T^{b_j}\bx)},\J(b_j)\right) 
>\varepsilon'\})<\varepsilon'$$ for $j\in \{1,2\}$ and $L\geq L_0$. Given this $L_0$, we apply Corollary~\ref{cor:most conditions} twice (for sufficiently large $\ell$ depending on $L_0$) to obtain sets $A$, $B$ and $p,p'\in \mathbb{Z}$ for $(b,b')=(b_1,b_2)$ and $(b,b')=(b_2,b_1)$ respectively as in the statement of the Corollary and such that $\nu(\cap_{i=0}^{L_0}A)$ and $\nu(\cap_{i=0}^{L_0}B)$ are at least $\frac 1 2 - \varepsilon'$. 
Iterating this provides conditions (i)-(ix). In particular, in the next application we have $p,p'$ instead of $b_1,b_2$.

Moreover, we claim that we can simultaneously apply these results to $d$ different pairs $b_1,b'_1,\ldots,b_d,b'_d$ 
 (the resulting common sets become $\hat{A}$ and $\hat{B}$).   To see this, choose 
$$\hat{A}=\{\bx\colon x_{10^{2\ell}}\in [\max\{|b_i|\}+\max\{|b_i-b_i'|\}+1,\ell-(\max\{|b_i|\}+\max\{|b_i-b_i'|\})-1]\}$$
and 
$$\hat{B}=\{\bx\colon x_{10^{2\ell}}\in [\ell+ \max\{|b_i|\}+\max\{|b_i-b_i'|\}+1,2\ell-(\max\{|b_i|\}+\max\{|b_i-b_i'|\})-3]\}.$$
We apply this argument for the $d$ pairs $b_1^{(1)},b_1^{(2)};\ldots;b_1^{(d)},b_1^{(1)}$  to produce measures $b_2^{(1)},\ldots,b_2^{(d)}$. 
Note that on $\hat{A}$, there is a reordering of $1,\ldots,d$, call it $p_1,\ldots,p_d$, such that $d(T^{b_j}\bx,T^{\hat{b}_{p_j}}\bx)<\varepsilon$ for all $1\leq j\leq d$ and $\bx \in \hat{A}$. (In fact, this is the reordering $p_j=j$.) 
There is a similar reordering on $\hat{B}$ (this is the reordering $p_j=j-1$ for $j\neq 1$ and $p_1=d$).  
 Inductively, given $b_j^{(1)},\ldots,b_j^{(d)}$ we apply this to the corresponding pairs 
$b_j^{(1)},b_j^{(2)};\ldots;b_j^{(d)},b_j^{(1)}$. 
Let $\hat{A}_1,\ldots,\hat{A}_j$ and $\hat{B}_1,\ldots,\hat{B}_j$ denote the corresponding sets, as above. 
 By Theorem~\ref{th:CE}, there exists 
 $j\in\mathbb{N}$ such that $b_j^{(i)}$ satisfy~\eqref{eq:getting first r} for all $i=1, \ldots, d$.  That is, 
$$ d_{\mathcal{M}(Y\times Y)}\bigl({b}_j^{(i)}),\frac 1 d \sum_{\ell=1}^d\J(b_\ell)\bigr)<\varepsilon.$$

Define 
\begin{equation}
\label{eq:def-hatW}
\hat{W} = \bigcap_{i=1}^j(\hat{A}_i\cup\hat{B}_i)
\end{equation}
to be the intersection of the sets obtained at each step, and this satisfies the desired conclusion.
\end{proof}

We now combine these results to find $\ell_1, \ldots, \ell_r$.  
Given $\varepsilon > 0$, applying Corollary~\ref{cor:getting first r} to $k_1,\ldots, k_r$, we obtain $\ell_1,\ldots,\ell_r$ such that 
$$\nu\bigl(\{\bx\colon d_{\mathcal{M}(Y)}(\frac 1 r \sum_{i=1}^rJ(\ell_i)_\bx,\frac 1 r \sum_{i=1}^rJ(k_i)_\bx\bigr) > \varepsilon/2 \})< \varepsilon/2$$ (Condition~\ref{cond:close to bary} for the first $r$)  and 
also satisfy the reordering condition on $\ell_1,\ldots,\ell_r$ in~\ref{cond:fiber}, where for each $\bx\in \hat{W}$,  the reordering is given by
$p'_k=k-|\{1\leq i\leq j\colon \bx \in \hat{B}_i\}|$ and this difference is taken modulo $d$.  (Note that 
as we have not yet introduced $\ell_{r+1},\ldots,\ell_{2r}$, we have not yet fully established~\ref{cond:close to bary} or~\ref{cond:fiber}.) 
Towards obtaining Conclusion~\ref{conc:close}, for each $\ell_i$, choose $N_i$ such that for all $L\geq N_i$ we have 
  $$d_{\mathcal{M}(Y\times Y)}\bigl(\frac 1 L\sum_{j=0}^{L-1}\delta_{(T^j\times T^j)(\bx,T^{\ell_i}\bx)},\J(\ell_i)\bigr)<\varepsilon.$$

\subsubsection{Finding $\ell_{r+1},\ldots,\ell_{2r}$}
We start first by finding $\ell_{r+1},\ldots,\ell_{2\lceil \frac 1 {16} r \rceil}$.

\begin{lem}\label{lem:pairing} 
Under the assumptions of Proposition~\ref{prop:prelimit omnibus} including the additional assumption, there exists $J\subset \{1,\ldots,r\}$ with $|J|= 2 \lceil  \frac {1} {2\cdot 8}  r \rceil$  and an order 2 bijection  $\phi\colon J \to J$ such that 
\begin{equation}\label{eq:shift}
\nu(\{\bx\colon d(T^{\ell_i}\bx,T^{\ell_{\phi(i)}}\bx)>c\})>\frac 1 8
\end{equation}
for all $i \in J$. 
\end{lem}
\begin{proof} First we claim that for each $i \leq r$, we have that 
\begin{equation}\label{eq:pairing goal}\Big|\Big\{j \leq r\colon \nu\big(\{\bx\colon d(T^{\ell_i}\bx,T^{\ell_{j}}\bx)>c\}\big)>\frac 1 8\Big\}\Big|\geq \frac r 8.
\end{equation}
To justify~\eqref{eq:pairing goal}, we limit our consideration to $W\cap \hat{W}$, where  $W$ is as in the statement of Proposition~\ref{prop:prelimit omnibus} and $\hat{W}$ is defined as in~\eqref{eq:def-hatW} as given in the proof of Corollary~\ref{cor:getting first r}
and note that $\nu(W\cap \hat{W})\geq \frac 1 2 -\varepsilon>\frac {49} {100}$. 
If~\eqref{eq:pairing goal} does not hold, 
$$
\int_{W\cap \hat{W}}\frac 1 r |\{j\colon d(T^{\ell_j}\bx,T^{\ell_i}\bx)
\leq c\}| \,d\nu 
\geq (\frac {49}{100}-\frac 1 8) \frac 7 8>\frac 1 2 \nu(W\cap \hat{W}).$$
It follows that there exists $\bx \in (W \cap \hat{W})$ such that 
$$|\{j:d(T^{\ell_j}\bx,T^{\ell_i}\bx)<c\}|>\frac r2.$$ 
Since $\bx \in \hat{W}$, it follows that 
$$|\{j\colon d(T^{k_j}\bx,T^{\ell_i}\bx)<c+\varepsilon\}|>\frac r2.$$
But since $\bx \in W \cap \hat{W}$, we have that $d(T^{\ell_i}\bx, T^a\bx)$ or $d(T^{\ell_i}\bx,T^b\bx)$ is less than $c+\varepsilon$, all of these $T^{k_j}\bx$ are at least $2c$ away from whichever of $T^a\bx$ or $T^b\bx$ that $T^{\ell_i}\bx$ is not close to. This contradicts the fact that $\bx \in W$.

Given~\eqref{eq:pairing goal}, we can obtain our set of $J$, because until $|J|=\lceil\frac 1 8 r\rceil$, we can always inductively pick any $i \notin J$ and find $j \notin J$ satisfying~\eqref{eq:shift} and add them both into $J$, letting $\phi(i)=j$ and $\phi(j)=i$. 
Thus we can obtain 
a set $J$ whose cardinality is the smallest even number that is at least $\frac 1 8 r$.
\end{proof}

\begin{lem}\label{lem:shift happens}
Assume there exist $a,b \in \mathbb{Z}$ and $c>0$ 
 such that 
$$\nu(\{\bx\colon d(T^a\bx,T^b\bx)>c\})>\frac 1 8.$$ 
Let $$G_k=\{\bx \in Y\colon x_{10^{2k}}\in[\frac 1 3 k+|a-b|,\frac 1 2 k-2-|a-b|]\text{  and } x_{10^{2k}-1}=3\}$$  and set
$d_k=a+(a-b)r_{10^{2k}}$. Then for 
every $\varepsilon>0$, there exists $k_0$ such for all $k\geq k_0$ and $\bx \in G_k$, 
 there exists $j_\bx \in [-k,k]$  satisfying 
\begin{enumerate}
\item \label{conc:follow a} $d(T^{ \ell+a+ j_\bx }\bx,T^{\ell+ a+\frac k 2 r_{10^{2k}}}\bx)<\varepsilon$
\item\label{conc:follow b} $d(T^{d_k+\ell+j_\bx+\frac k2r_{10^{2k}}}\bx,T^{b+\ell }\bx)<\varepsilon$
\end{enumerate}
 for all $\ell\in [-r_{10^{2k}-1},r_{10^{2k}-1}]$.
Moreover, for all but a set of such $\bx$ with measure at most $\varepsilon$, 
we have
\begin{equation}\label{eq:shift happens}
\Big|\{\ell\in [-r_{10^{2k}-1},r_{10^{2k}-1}]\colon  d(T^{a+\ell}\bx,T^{b+\ell}\bx)>c\}\Big|> \frac 1 9 2r_{10^{2k}-1}.
\end{equation}
\end{lem} 

\begin{proof}
We apply the proof of Corollary~\ref{cor:most conditions} with $n=d_k$ to obtain the first  2 conditions. 
More precisely, 
by construction $G_k$ is a subset of $A_k \cup B_k$ where $A_k$ and $B_k$ given by the proof of Corollary~\ref{cor:most conditions}.   
Choose $j_{\bx}=\sum_{i=0}^{\zeta_{\bx}(\frac k 2 r_{10^{2k}})-1}\one _{D_{10^{2k}}}(S^i\bx)-\frac k 2$. (Recall that $D_j$ is defined in~\eqref{def:Dk}.)
Then since $x_{10^{2k}-1}=3$, 
it follows that $j_{\bx}=j_{T^{\ell'}\bx}$ for all such $\bx$ and $\ell'\in [-\frac 3 2 r_{10^{2k}-1},\frac 3 2 r_{10^{2k}-1}]$. 
 Indeed, because $\by \in D_{10^{2k}}$ implies $y_{10^{2k}-1}=6$, for all $\bx$ with $ x_{10^{2k}-1}=3 $ and 
$\ell'\in [-\frac 3 2 r_{10^{2k}-1},\frac 3 2 r_{10^{2k}-1}]$ we have 
$$\sum_{i=0}^{\zeta_{\bx}(\frac k 2 r_{10^{2k}})-1}\one_{D_{10^{2k}}}(S^i\bx)- \sum_{i=0}^{\zeta_{T^{\ell'}\bx}(\frac k 2 r_{10^{2k}})-1}\one_{D_{10^{2k}}}(S^i\bx)=0.$$
 
Choosing $k_0$ such that $|a|,\,|b|<r_{10^{2k_0}-1}$,  the first 2 conditions 
hold. 
For the final condition, let $V= \{\bx\colon d(T^{a}\bx,T^{b}\bx)>c\}$. 

 By the (mean)  ergodic theorem, there exists $N\in\N$ such that 
$$\nu\bigl(\{\bx\colon \frac 1 M\sum_{i=0}^{M-1}\one_{V}(T^i\bx)>(1-\varepsilon)\nu(V)\}\bigr)>1-\varepsilon$$
for all $M\geq N$. Choosing $r_{10^{2k}-1}> N$ we have~\eqref{eq:shift happens}.  
\end{proof}

We now use this to define $\ell_j$ for $j\in \{r+1,\ldots, r+ 2\lceil  \frac r {16} \rceil \}$.
 Choose $J$ as in Lemma~\ref{lem:pairing} and enumerate 
the elements of the set $J$ 
 as $a_1,\ldots, a_{2\lceil\frac  r {16}   \rceil}$. 
Let  $\ell_{r+j}$ be given by Lemma~\ref{lem:shift happens} applied 
with $a=a_j$ and $b=\ell_{\phi(a_j)}$ 
and where $k$ is chosen larger than the $k_0$ needed for the $2\lceil  \frac r {16}  \rceil$ 
 applications of Lemma~\ref{lem:shift happens}, as well as sufficiently large such 
that $\ell_i< \frac 1 2 r_{10^{2k}-1}$ 
for all $1\leq i \leq r$.  For each such $j$,  we define $\ell_{r+j}$ to be  the corresponding $d_k$.

We now define $\ell_{r+ 2\lceil  \frac {r} {16}  \rceil+1},\ldots,\ell_{2r}$.  Define $\hat{V}_{a}$ to be the set 
of $\bx$ 
such that $d(T^{\ell_j}\bx,T^{\ell_j+r_{10^{2a+1}}}\bx)<\varepsilon$ for all $j \in J^c$. Observe that for all large enough $a$, $\nu(\hat{V}_a)>1-\varepsilon$. 
Let $i_1,...,i_{r-2\lceil \frac r {16}\rceil}$ be an enumeration of $J$ and define $\ell_{j+r+2\lceil \frac r {16}\rceil}=\ell_{i_j}+r_{10^{2a+1}}$. So $\{\ell_j\}_{j=r+2\lceil \frac r {16}\rceil+1 }^{2r}=\{\ell_{j'}+r_{10^{2a+1}}\}_{j' \in J^c}$. 

We use this to prove the proposition: 
\begin{proof}[Proof of Proposition~\ref{prop:prelimit omnibus}] 
Choose $M=\frac k 2 r_{10^{2k}}$, $L=r_{10^{2k}+1}$, and  set 
$$
B = \big[\frac 1 3 k+2\max_{j=1,\ldots, 2\lceil \frac {r}{16} \rceil}\{|\ell_{a_j}|\},\frac 1 2 k-2-2\max_{j=1,\ldots, 2\lceil \frac r {16} \rceil}\{|\ell_{a_j}|\}\big]$$
and 
$$A=\big\{\bx\colon x_{10^{2k}-1}=3 \text{ and } x_{10^{2k}}\in B \big\} \cap \hat{W} \cap \hat{V}.
$$
By Conclusion~\eqref{conc:follow a}  of  Lemma~\ref{lem:shift happens} and our choice that $|\ell_j|<r_{10^{2k}-1}$, we have Conclusion~\ref{conc:stay} with  
$j_\bx$ as in Lemma~\ref{lem:shift happens}. 
  By Corollary~\ref{cor:getting first r} (see also the last paragraph of Section~\ref{sec:first r}) we have Conclusion~\ref{cond:close to bary} and Conclusion~\ref{cond:fiber} for $\{\ell_i\}_{i=1}^r$.
Since each $\ell_{a_1},\ldots, \ell_{a_{2\lceil \frac {r} {16} \rceil}}$ appears as both $j$ and $\phi(j)$ 
in our construction of $\{\ell_i\}_{i=r+1}^{r+2\lceil \frac r{16} \rceil}$, 
 it follows outside a set of $\bx$ with small measure, 
 for each such $\bx$ 
 there exist, $p_1,\ldots, p_{2\lceil \frac {r}{16} \rceil}$, a reordering of $\ell_i$ for $r<i \leq r+2 \lceil \frac {r}{16} \rceil$  such that $d_Y(T^{\ell_{a_i}}\bx, T^{\ell_{p_i}}\bx)<\varepsilon$. 
For the remaining $r<j\leq 2r$, the off diagonal joining $\ell_j$ is built to be $\varepsilon$ close to the corresponding $\ell_{i}$.
 Thus Conclusion~\ref{cond:fiber} follows  for $\{\ell_i\}_{i=r}^{2r}$.  We have~\ref{conc:close} for $N=\max\{N_i\}_{i=1}^r$ (see the end of Section~\ref{sec:first r}). Finally, by~\eqref{eq:shift happens} and~\eqref{conc:follow b} of Lemma~\ref{lem:shift happens}, we have~\ref{conc:move}. Indeed, by our choice of $A$ we can apply~\eqref{conc:follow b} and by~\eqref{eq:shift happens} this gives the desired distance bound of $T^{M+d_k+\ell}\bx$ from $T^{d_k+\ell+j_\bx}\bx$. 
\end{proof}

\subsection{Proof of Theorem~\ref{thm:big joining}} \label{sec:big join proof}
Let $\mathcal{J}_{\nu}$ denote the self-joinings of $(Y,T,\nu)$.
  Recall that $\sigma_x$ denotes a measure on $Y$, and not on $\{x\}\times Y$.

\begin{lem}\label{lem:close good} 
Let $\varepsilon> 0$,  $k_1,\ldots, ,k_r \in \mathbb{Z}$, $\ell_1,\ldots,\ell_{2r}, L, N, M \in\mathbb Z$, the set $A$, and $j_\bx\in[-M,M]$ be as in Proposition~\ref{prop:prelimit omnibus}. 

There exists $\frac 1 {20}>\delta>0$ such that if for some $\sigma \in \mathcal{J}_\nu$ we have 
\begin{equation}\label{eq:pointwise close}
\nu\Big(\big\{\bx\colon  d_{\mathcal{M}(Y)}\big(\sigma_\bx, (\frac{1}{2r}\sum_{n=1}^{2r}\J(\ell_{n}))_\bx\big)>\delta\big\}\Big)<\delta,
\end{equation} 
then there exists $\tilde{A} \subset A\subset Y$ and $\nu(\tilde{A})> \frac 1 {999}$ such that
\begin{enumerate}
\item\label{item:close good i} $\sigma(\{(\bx,\by) \in \tilde{A} \times Y\colon d_{\mathcal{M}(Y \times Y)}(\frac 1 N \sum_{i=1}^N\delta_{(T^i\bx,T^i\by)}, \frac 1 r \sum_{i=1}^r \J(k_i))<2\varepsilon\})>\frac 9 {10} \nu(\tilde{A}) $.
\item \label{item:close good b}
 for all $\bx \in \tilde{A}$, there exists $C_\bx$ with 
$\sigma_\bx(C_\bx)
>\frac 1 {99999}$ and 
$$\frac 1 L \sum_{i=0}^{L-1}d(T^{M+i}\by,T^{j_\bx+i}\by)<2\varepsilon$$ for all $\by \in C_\bx$.

\end{enumerate}
Moreover, under the additional assumption 
that there exist $a,b \in \mathbb{N}$ and $c>0$ such that $d(T^a\bx,T^b\bx)>3c$ for a set of $\bx$ with measure  $\frac 1 2 $, 
then  there exists $E_\bx$ with $\sigma_\bx( E_\bx) > \frac{1}{99999}$ satisfying 
\begin{enumerate}[resume]
\item\label{item:close good iii} $\frac 1 L|\{0\leq i<L\colon  d(T^{M+i}\by,
 T^{j_\bx+i}\by )>\frac c 2\}|>\frac c 2 $ for all $\bx \in \tilde{A}$ and $\by \in  E_\bx$.
\end{enumerate}
\end{lem}
\begin{proof} 
Choose a compact set $G$ with $\nu(G)>1-\frac \varepsilon {10000r}$ such 
that $T^i|_G$ is (uniformly) continuous for all $|i|\leq \max\{N,L,M\}$.  Let $\hat{G}=G\cap   \bigcap_{n=1}^{2r}T^{-\ell_n}G$.  There exists $\delta>0$ such that if $\bx \in G$ and $d(\bx,\by)<\delta$, 
then 
\begin{equation} \label{eq:little move} d(T^i\bx,T^i\by)<\min\{\varepsilon,10^{-7}\} \quad\text{ for all }|i|\leq  \max\{N,L,M\}.
\end{equation}
Thus we can choose $A_1 = A \cap  \hat{G}$. 
If $\bx \in A_1$, $\by,\by' \in \hat{G}$,  $d(\by,\by')<\delta$, and  
$$d_{\mathcal{M}(Y \times Y)}(\frac 1 N \sum_{i=1}^N\delta_{(T^i\bx,T^i\by)}, \frac 1 r \sum_{i=1}^r \J(k_i))<\varepsilon,$$ then by~\eqref{eq:little move} the definition of $d_{\mathcal{M}(Y\times Y)}$ we have 
$$ d_{\mathcal{M}(Y \times Y)}(\frac 1 N \sum_{i=1}^N\delta_{(T^i\bx,T^i\by')}, \frac 1 r \sum_{i=1}^r \J(k_i))<2\varepsilon.$$
Thus  Condition~\eqref{item:close good i} follows from Condition~\ref{conc:close} of Proposition~\ref{prop:prelimit omnibus}  (as well as~\eqref{eq:pointwise close} and the measure bound on $G$). 
Setting 
$$C_{\bx}= \hat{G} \cap  \bigcup_{n=1}^rB(T^{\ell_{n}}\bx,\delta),$$   
then~\eqref{item:close good b} (without the measure bound) follows from~\ref{conc:stay}. 
Setting 
$$E_{\bx} 
=\hat{G}\cap  \bigcup_{n=r+1}^{r+2\lceil \frac r {16}  \rceil} B(T^{\ell_{n}}\bx,\delta),$$ 
then~\eqref{item:close good iii}  (without the measure bound)  follows from~\ref{conc:move}. Let $\tilde{A}$ be the subset of $A_1$ such that~\eqref{eq:pointwise close} holds, $\sigma_{\bx}(A_1)>\frac {19} {20}>\frac 9 {10}+\delta$ and so that $\sigma_\bx(C_\bx),\,\sigma_{\bx}(E_{\bx})>\frac 1 {99999}$. 
\end{proof}

Before completing the proof of Theorem~\ref{thm:big joining} we note the following.  
If $(X,T,\mu)$ is an  Borel probability system and $X$ is a compact metric space, 
then $\mu$ is ergodic if and only if  there exists a sequence $N_i \to \infty$ such that for every $f \in C(X)$, 
\begin{equation}\label{eq:generic}
\underset{i \to \infty}{\lim} \frac 1 {N_i} \sum_{j=0}^{N_i-1}f(T^jx)-\int f d\mu=0
\end{equation}
for $\mu$-almost every $x \in X$.

\begin{proof}[Proof of Theorem~\ref{thm:big joining}]
 We now produce a joining as close (with respect to $d_{\mathcal{M}(Y\times Y)}$) as desired to $\frac 1 2 (J(0)+J(1))$, which can thus be assumed to be different from the product joining.
Let  $k_1^{(1)}=0$ and $k_2^{(1)}=1$.  We apply Proposition~\ref{prop:prelimit omnibus} with $\varepsilon<\frac 1 {10^{-3}}<\frac 1 {10}\,  \underset{\bx \in Y}{\min}\, d(T\bx,\bx)$ to obtain $\ell_1,\ldots,\ell_4$, which we denote as $k^{(2)}_1,\ldots,k_4^{(2)}$. 
We also obtain $N_1,L_1,M_1$ and we then apply Lemma~\ref{lem:close good} to obtain $\delta_1$.
Applying Proposition~\ref{prop:prelimit omnibus} with $\varepsilon_2 =\min\{\frac 1 {2^2},  \frac{\delta_1}2\}$ and $k^{(2)}_1,\ldots,k_4^{(2)}$ and we  also obtain $N_2,L_2,M_2$. 
We repeat the application of Lemma~\ref{lem:close good} to obtain $\delta_2$, 
which without loss of generality we can 
assume 
is less than $\frac{\delta_1}2$. 
We repeat this procedure, inductively obtaining $k_1^{(r)},\ldots,k_{2^r}^{(r)}$, which by 
Part~\ref{cond:close to bary} of Proposition~\ref{prop:prelimit omnibus} satisfies 
$$\nu\Big(\big\{\bx\colon d_{\mathcal{M}(Y)}\big(\frac{1}{2^r}
\sum_{i=1}^{2^r}\J(k_i^{(r)})_\bx, \frac{1}{2^n}
\sum_{i=1}^{2^{n}}\J(k_i^{(n)})_\bx\big)>\delta_{n}(1-\frac 1 {2^{r-n}})\big\}\Big)<\delta_{n}(1-\frac 1 {2^{r-n}})$$
 for all $n<r$.  
 Applying Lemma~\ref{lem:close good} to obtain $\delta_r$, 
 which again we take to be bounded by $\frac{\delta_{n}}{2^{r-n}}$ for all $n<r$, 
 we can repeat the application of Proposition~\ref{prop:prelimit omnibus}, 
 but with  $\varepsilon_{r+1}=\min\{\frac 1 {2^{r+1}},\frac 1 2\delta_r\}$. 
 
We pass to the weak*-limit of $\frac 1 {2^r}\sum_{i=1}^{2^r} \J(k_r^{(i)})$, which we denote $\sigma$. By construction,  we have that 
 \begin{equation}\label{eq:sigma close}
 \nu\Big(\big\{\bx\colon d_{\mathcal{M}(Y)}\big(\sigma_\bx, \frac{1}{2^{r}}
 \sum_{i=1}^{2^{r}}\J(k_i^{(r)})_x\big)>\delta_{r} )\big\}\Big)\leq  \delta_{r}
 \end{equation}
 for all $r\in\N$. 
 From this, it follows via Lemma~\ref{lem:close good}~\eqref{item:close good i} and Proposition~\ref{prop:prelimit omnibus}~\ref{conc:close},  
 \begin{multline*}
 \sigma\Bigl\{(\bx,\by) \in Y \times Y\colon d_{\mathcal{M}(Y \times Y)}(\frac 1 {N_r}\sum_{i=1}^{N_r}\delta_{( T^i\bx,T^i\by)}, \frac 1 {2^{r-1}} \sum_{i=1}^{2^{r-1}} \J(k^{(r-1)}_i))>\varepsilon_r\Bigr\}\\ <2\varepsilon_r\leq\frac 1 {2^{r-1}}.
 \end{multline*}
 Thus 
 $\underset{r \to \infty}{\lim} \frac 1 {N_r}\sum_{i=1}^{N_r}\delta_{(T^i\bx,T^i\by)}$ is the weak*-limit of $\frac 1 {2^r} \sum_{i=1}^{2^r} \J(k^{(r)}_i)$, which is $\sigma$, and this holds for $\sigma$-almost every $(\bx,\by)\in Y\times Y$. 
 By the criterion given in~\eqref{eq:generic}, 
 it follows that $\sigma$ is ergodic.

 Thus to complete the proof of the theorem, it suffices to show that the assumptions for Theorem~\ref{thm:no compact} are satisfied.   By ~\eqref{eq:sigma close} we can apply Lemma~\ref{lem:close good} with $\sigma$ and any 
 $\frac{1}{2^r}\sum_{i=1}^{2^{r}} J(k_i^{(r)})$ that we have produced. That is, it suffices to show that for each $c'>0$, for all large enough $r$, we have that $\frac{1}{2^r}\sum_{i=1}^{2^{r}} J(k_i^{(r)})$ satisfies the assumptions of Theorem~\ref{thm:no compact} with $L_i=L_r$, $n_i=M_r$,  as this then verifies the assumptions of Theorem~\ref{thm:no compact} for the weak* limit and with the same parameters $L_r,M_r$. 
 This gives us a sequence of sets $\tilde{A}_m$ such that  for every $\bx \in \tilde{A}_m$ we have sets $C^{(m)}_\bx,  E_\bx^{(m)}$  such that $\sigma_\bx(C_\bx^{(m)}),\sigma_\bx( E_\bx^{(m)})
 >\frac 1 {99999}$ and 
 \begin{enumerate}
\item  $\frac 1 L \sum_{i=0}^{L-1}d(T^{M+i}\by,T^{j_\bx+i}\by)<2\varepsilon$ 
 for all $\bx \in \tilde{A}_m$ and  $\by \in C_\bx$ giving Theorem~\ref{thm:no compact}~\eqref{cond:stay} as $m\to \infty$ and we can choose $\varepsilon \to 0$. 
 (This uses Proposition~\ref{prop:prelimit omnibus}, part~\ref{conc:stay} and Lemma~\ref{lem:close good}, part~\eqref{item:close good b}.)  
\item $\frac 1 L|\{0\leq i<L\colon d(T^{M+i}\by,T^{j_\bx+i}\by)>\frac c 2\}|>\frac c 2 $ for all $\bx \in \tilde{A}_m$ and $\by \in  E_\bx$ giving Theorem~\ref{thm:no compact}~\eqref{cond:move}. 
(This uses Proposition~\ref{prop:prelimit omnibus}, part~\ref{conc:move} and Lemma~\ref{lem:close good}, part~\eqref{item:close good iii}.) 

\item The assumption that $\sigma_\bx(C_\bx),\sigma_\bx(E_\bx)>\frac 1 {99999}$ giving Theorem~\ref{thm:no compact} ~\eqref{cond:eta sees}.
\item  Proposition~\ref{prop:prelimit omnibus}, part~\ref{cond:fiber} applied to $r<d\leq r+ 2 \lceil  \frac r {16}    \rceil $ 
combined with~\eqref{eq:pointwise close} imply Theorem~\ref{thm:no compact}~\eqref{cond:extra}.
\end{enumerate}

Note that strictly speaking, $M$, $L$, $C_\bx$, $E_\bx$, and $j_{\bx}$ depend on $m$, but we omit the dependency for the sake of readability.  The choices of $C_{\bx}$ and $E_{\bx}$ are given by Lemma~\ref{lem:close good} and we have that $M$ and $L$ are $M_m$ and $L_m$ introduced earlier in the proof (which are required input for using Lemma~\ref{lem:close good}). 
Thus we have proven the assumptions needed to apply Theorem~\ref{thm:no compact}. 
\end{proof}

\subsection{Poulsen simplex}\label{sec:Poulsen}

We have assembled the tools to prove the last part of Theorem~\ref{theorem:main}, showing that the set of self-joinings of the 
constructed system form a {\em Poulsen simplex},   meaning that they form a simplex such that the extreme points are dense.

\begin{prop}
\label{prop:paulsen}
The self-joinings of  the system $(Y,\nu,T)$ form a Poulsen simplex. 
\end{prop}

In the proof, we make use of a result of King:
\begin{thm}[{King~\cite[EJCL Theorem]{king flat}}]
\label{th:king}
If $\eta$ is a self-joining of $(Y,\nu,T)$ and $T$ 
is rigid rank 1, then there exist real numbers $\alpha_i^{(k)}>0$ such that $\sum_i \alpha_i^{(k)}\J(i)$ converges
 in the weak*-topology to $\eta$.
\end{thm}
Note that in our setting, we can utlize this result, as rigid rank 1 transformations have {\em flat stacks}. 
In fact, King establishes that the ergodic self-joinings of transformations with flat stack lie in the weak closures of off diagonal joinings.  A different proof of this result is given in~\cite[Corollary 2.3]{CE} (see also~\cite[Corollary 0.3]{note}).

\begin{proof}[Proof of Proposition~\ref{prop:paulsen}]
By King's Theorem, it suffices to show that for any integers $n_1, \ldots, n_k$ and positive rationals $\beta_1,\ldots, \beta_k$ such that $\sum \beta_i=1$, there exists $m$ such that {$d_{\mathcal{M}(Y\times Y)}(\J(m),\sum \beta_i \J(n_i))<\varepsilon$.}
Without loss of generality, we may assume that all of the rationals have a common denominator, 
writing $\beta_i=\frac{m_i}r$ where all $m_i$ are positive integers. 
By Corollary~\ref{cor:getting first r}, applied to $n_1,\ldots,n_k$ where each $n_i$ appears $m_i$ times, there exists $m$ such that $d_{\mathcal{M}(Y\times Y)}(\J(m),\frac 1 r\sum_i \sum_{\ell=1}^{m_i} \J(a_\ell^{(i)}))<\frac \varepsilon 2$. Thus 
$d_{\mathcal{M}(Y\times Y)}(\J(m), \sum \beta_i \J(n_i))<\varepsilon$. 
\end{proof}

\subsection{These properties are residual}  \label{sec:residual}
\begin{thm} A residual set of measure preserving transformations are not quasi-simple.
\end{thm}

If $(h_j)_{j\in\N}$ is a sequence of positive integers, 
we say a system $(X,T,\mu)$ \emph{admits special linked approximation of type $(h_j,h_j+1)$}
 if for each $j\in\N$, there exist sets $A_j, C_j\subset X$ satisfying the following five conditions: 
\begin{enumerate}
\item $\underset{j \to \infty}{\lim}\, \mu(\bigcup_{i=0}^{h_j-1}T^iA_j)=\frac 1 2 =\underset{j \to \infty}{\lim}\, \mu(\bigcup_{i=0}^{h_j}T^iC_j)$; 
\item The sets $A_j,\ldots,T^{h_j-1}A_j,C_j, \ldots,T^{h_j}C_j$ are pairwise  
disjoint; 
\item $\underset{j \to \infty}{\lim}\, \frac{\mu(T^{h_j}A_j \cap A_j)}{\mu(A_j)}= 1  =\underset{j \to \infty}{\lim}\, \frac{\mu(T^{h_j+1}C_j\cap C_j)}{\mu(C_j)}$; 
\item \label{cond:special} 
Defining $$\mathcal{R}_A^{(j)}=\bigsqcup_{i=0}^{h_j-1}T^i A_j \text{ and } \mathcal{R}_C^{(j)}=\bigsqcup_{i=0}^{h_j}T^i C_j, $$ 
there exist measurable sets $J_j \subset A_j$  and $a,b \in \mathbb{N}$ such that $J_j,\ldots, T^{a+b-1}J_j$ are all pairwise disjoint, $T^iJ_j \subset \mathcal{R}_A^{(j)}$ for all $0\leq i\leq a-1$, and $T^iJ_j \subset  \mathcal{R}_C^{(j)}$ for all $a\leq i\leq a+b-1$ and $\underset{j \to \infty}{\lim}\, \mu(\bigcup_{i=0}^{a+b-1}T^iJ_j)=1$; \item For all $\varepsilon>0$, there exist  measurable sets $B^{(j)}_0,\ldots,B^{(j)}_{h_j-1}$ and $\hat{B}^{(j)}_0,\ldots,\hat{B}^{(j)}_{h_j}\in X$ of diameter at most $\varepsilon$ such that 
$$\underset{j \to \infty}{\lim}\, \sum_{i=0}^{h_j-1}\mu(T^iA_j\setminus B^{(j)}_i)=0=\underset{j \to \infty}{\lim}\, \sum_{i=0}^{h_j}\mu(T^iC_j\setminus \hat{B}^{(j)}_i).$$ 
\end{enumerate}

Condition~\eqref{cond:special} distinguishes
this from usual linked approximation,  
and is needed to carry out the arguments of Section~\ref{sec:not qs}. 
This property is a residual property in the space of measure preserving transformations. 
Indeed, it is conjugacy invariant, and nonempty. Halmos~\cite[Theorem 1]{hal mix} showed that the conjugacy class of any aperiodic measure preserving transformation is dense. Our conditions (i)-(v) are the intersection of a countable number of open conditions and so the property holds on a $G_\delta$ set.
Thus it is a dense $G_\delta$,  that is residual,  property. 

We say a system $(X,T,\mu)$ is {\em rigid rank 1} if there exist numbers $n_j$ and sets $I_j$ such that \begin{enumerate}
\item $\underset{j \to \infty}{\lim}\, \mu(\bigcup_{i=0}^{n_j-1}T^iI_j)=1$;
\item The sets $I_j,\ldots,T^{n_j-1}I_j$ are  pairwise disjoint;
\item $\underset{j \to \infty}{\lim}\, \frac{\mu(T^{n_j}I_j \cap I_j)}{\mu(I_j)}= 1$;
\item For all $\varepsilon>0$, there exist measurable sets $B^{(j)}_0,\ldots, B^{(j)}_{n_j-1}\in X$ of diameter at most $\varepsilon$ such that  
$$\underset{j \to \infty}{\lim}\, \sum_{i=0}^{n_j-1}\mu(T^iI_j\setminus B^{(j)}_i)=0.$$
\end{enumerate}
Note that this property is stronger than being both rigid and rank 1. 
Similarly to  the property of admitting a special linked approximation, 
rigid rank 1 is also a residual property in the space of measure preserving transformations.

 Any transformation that both admits a special 
 linked approximation of type $(h_j,h_{j+1})$ and is rigid rank 1 has a self-joining that is not a distal extension of 
$(X,T,\mu)$. 
Indeed, these transformations have the following property: 
for any pair of integers $a,b\in\N$ and $\varepsilon>0$, 
there exists $m\in\N$ and a pair of sets $C,D$ with measure at least $\frac 1 2-\varepsilon $ so that 
$$\mu(\{x\in C\colon d(T^ax,T^mx)>\varepsilon)\}<\varepsilon \text { and } \mu(\{x\in D\colon d(T^bx,T^mx)>\varepsilon)\}<\varepsilon.$$
Using this property, rank 1 rigidity, and the ergodicity of $\mu$, 
our construction of the joining that is not a distal extension of $(X,T,\mu)$ proceeds  similarly to Sections~\ref{sec:not qs} and~\ref{sec:big join proof}. More precisely, for 
sufficiently large $j$, we can choose $C=\bigcup_{i=0}^{h_j-1}T^iA_j$, $D= \bigcup_{i=0}^{h_j}T^iB_j$, and $m= a+(a-b)h_j$. The inductive construction of $\ell_1,\ldots,\ell_r$ proceeds verbatim. Similarly for $\ell_{r+1},\ldots,\ell_{r+2\lceil \frac r {16} \rceil}$ is almost verbatim (the described set in Lemma~\ref{lem:shift happens} is less explicit) and 
the construction of $\ell_{r+2\lceil \frac r {16}\rceil+1},\ldots,\ell_{2r}$ is verbatim (making use of the property that our transformation is rigid). 
\begin{rem}Analogously, Proposition~\ref{prop:paulsen} can be generalized for any rigid rank 1 transformation that admits special linked approximation of type $(h_j,h_j+1)$. 
Using this, it follows that there is a residual set of measure preserving transformations 
such that their self-joinings form a Poulsen simplex.
\end{rem}

\section{Coding and results}\label{sec:coding}
Sections~\ref{sec:coding} and~\ref{sec:end} are interrelated and technical, 
and these contain the arguments that rule out an arbitrary, non-trivial factor. 
We do this by studying the Markov operators. 
As our system $(Y,\nu,T)$ has many self-joinings, it also has many Markov operators on $L^2(\nu)$. 
The crux of our argument is that none of these Markov operators can be projections, 
other than those corresponding to trivial factors.  
A challenge, which offers some justification for the technical nature of our proof, is that our arguments need to take into account that there are two projections, with qualitatively different behaviors from each other, arising  from the two different types of trivial factors: integration against $\nu$, which arises via the factor map to the one point system, 
and the identity map,  which arises via a factor map that is an isomorphism.
These arguments are carried out in Section~\ref{sec:end}, which unfortunately is difficult to summarize at this point, as it rules out non-trivial Markov operators that are not projections by treating three possible cases. 
The rough idea of Section~\ref{sec:end} is that if $T$ has a non-trivial factor $P$ with Markov operator $F$, 
then there exists a measurable set $A$ such that $\langle F\one_A,F\one_{T^{-1}A}\rangle =0$ 
 and (the contrapositive to) Lemma~\ref{lem:easy friend} shows that this Markov operator corresponds to a factor to the one point system (and so the factor was trivial).  
The idea of  verifying the negation of Equation~\eqref{eq:small friends}  in Lemma~\ref{lem:easy friend}
 is that if $F'$ is the Markov operator corresponding to a non-trivial self-joining of $T$, 
then for any measurable set  $A$ of positive measure, there exists some iterate $M$ of the operator such that 
\begin{equation}\label{eq:idea}
\langle F'^M\one_A,F'^M\one_{T^{-1}A}\rangle>0.
\end{equation} To do this, we relate $F'$ to $\sum \alpha_i U_{T^i}$ (Theorem~\ref{thm:CE factor} and Corollary~\ref{cor:CE factor}) and define a notion called $i$-friends adapted to this property in Section~\ref{sec:friends}, 
showing that there is some small iterate $i$ of the transformation $T$ such that $\langle F'^M\one_A,F'^M\one_{T^{-i}A}\rangle>0$. 
To study this quantity, we relate $\sum \alpha_i U_{T^i}$ to $\sum_{i=-r_{N}}^{r_N}\beta_iU_{T^i}$ (this is the idea of Section~\ref{sec:recoding})  for fixed $N$ depending on $A$. 
We show that there exists some $M$, which depends on the choice of $N$,  such that  
$$\langle\sum_{i=-r_{N}}^{r_N}\beta_iU_{T^i}^M\one_A,(\sum_{i=-r_{N}}^{r_N}\beta_iU_{T^i})^M\one_{T^{-1}A}\rangle>0.$$
This argument covers the first case of Section~\ref{sec:end}. 
Now, though, we can (and do!) choose $\alpha_i$ such that 
$(\sum \alpha_i U_{T^i})^M$ is close to $F'^M$,  we can not conclude that
 $\sum_{i=-r_{N}}^{r_N}\beta_iU_{T^i}^M$ is close to $\sum \alpha_i U_{T^i}$; 
 in particular, the closeness of  $\sum_{i=-r_{N}}^{r_N}\beta_iU_{T^i}$ to $\sum \alpha_i U_{T^i}$ depends on $N$, but $M$ also depends on $N$.  However, we can show that if these two quantities are not close then~\eqref{eq:idea} still holds  (this corresponds to cases 2 and 3 of Section~\ref{sec:end}). 
 Section~\ref{sec:coding} sets up the machinery for Section~\ref{sec:end} and is perhaps even more opaque, though it is motivated by explanations within that section. As we are only concerned with factors, our results are all stated for Markov operators corresponding to factors. (There are two simplifications in the above description: In reality, we can not just consider $\langle F'^M\one_A,F'^M\one_{T^{-1}A}\rangle,$ 
 and instead we must consider 
 $\langle F'^M\one_A,F'^M\one_{T^{-i}A}\rangle$ for $i\in \{1,2,3\}$. Additionally we approximate  $\sum \alpha_i U_{T^i}$ by $\sum \beta_i\one_{B_i}U_{T^i}$ where $B_i\subset Y$ are cylinders.)  
 
Before starting, we also give a short overview of Section~\ref{sec:coding}. In Section~\ref{sec:friends} we introduce a key definition \emph{i-friends} 
 and Lemma~\ref{lem:easy friend} whose contrapositive is used to show our system is prime. The application of this lemma 
  makes use of an elaborate inductive definition (see Definition~\ref{def:reduce}), 
characterizing the relation between the transformation $T$ (or some of its small powers) and large powers of $T$.  
Informally, we call this reducing or the reduction of the power, as it gives us a procedure by which to replace higher powers of $T$ by lower ones.
We study this procedure in Lemmas~\ref{lem:reduction close} and~\ref{lem:cylinders containing bad} and throughout Section~\ref{sec:getting friends}, showing how it is connected to the notion of $n$-friends. 
This leads to a criterion for our process to be prime, developed in Section~\ref{sec:restrict factors}.  
Namely, using Propositions~\ref{prop:no friends} and~\ref{prop:norig}, we show that if 
$T$ has a non-trivial factor, then the inductive procedure given in Definition~\ref{def:reduce} only can produce small errors. We provide additional motivation throughout this section.
 
\subsection{The mechanism for showing $(Y, \nu,T)$ is prime}
Throughout this section, we continue to assume that $(X, \mu, S)$ and $(Y, \nu, T)$ are the systems defined in Section~\ref{sec:construction}, maintaining all of the notation introduced in that section. 
The proof that $(Y,\nu,T)$ is prime is based on showing that a factor map is either an isomorphism or a map to the one point system. 
The first step is relating factor maps to linear combinations of powers of $T$ which holds for any rigid rank $1$ transformation: 
\begin{thm}\label{thm:CE factor}(\cite[Theorem 2.2]{CE}) 
If $P$ is a factor map of $T$  and $F$ is the corresponding Markov operator, 
then $F$ is the limit in the strong operator topology 
of linear combinations of powers of $U_{T}$ with non-negative coefficients. 
\end{thm}
This theorem is stated in~\cite{CE} for any Markov operator corresponding to any self-joining of any rigid rank $1$ transformation.
\begin{cor}\label{cor:CE factor}If $P$ is a factor map of $T$  and $F$ is the corresponding Markov operator, then there exists a sequence of convex combinations $\sum_{i \in\mathbb{Z}}\alpha_i^{(k)}$ satisfying 
$\sum_{i \in\mathbb{Z}}\alpha_i^{(k)} =1$ and 
 such that $ \sum_{i\in \mathbb{Z}}\alpha_i^{(k)}U_{T^i} \rightarrow F$ in the strong operator topology as $k\to\infty$. 
\end{cor}
\begin{proof}
The existence of the sequence of $\alpha_i^{(k)}$ without the 
extra hypothesis that $\sum \alpha_i^{(k)}=1$ for each $k$ follows from Theorem~\ref{thm:CE factor}. For this last assumption, observe that 
$F\one_Y=\one_Y$, and so $\underset{k \to \infty}{\lim}\,\sum_i\alpha_i^{(k)}U_{T^i}\one_Y=\one_Y$  and we may assume the (non-negative) coefficients add up to 1. Indeed, because $\alpha_i^{(k)}$ are all positive, $\|\sum_i\alpha_i^{(k)}U_{T^i}\one_Y\|=\sum_i\alpha_i^{(k)}$ and we see that we may assume the $\sum_{i}\alpha_i^{(k)}$ is a convex combination.
\end{proof}

\subsection{A condition for a factor to be the one point system}\label{sec:friends}
Recall that $Z_k$ and $W_k$ are defined in~\eqref{def:Zk} and~\eqref{def:Wk}.  
Given $n \in \mathbb{N}$, 
we say that $\bx,\by \in Y$ are \emph{n-\friend} if 
$$\sum_{j=0}^n\one_{Z_k}(S^j\bx)=\sum_{j=0}^n\one_{Z_k}(S^j\by)$$ 
for all but one $k\in\mathbb{N}$,  
$$\Bigl|\sum_{j=0}^n\one_{Z_{k}}(S^j\bx)-\sum_{j=0}^n\one_{Z_{k}}(S^j\by)\Bigr|=1$$ for exactly one $k\in\mathbb{N}$, and 
$$\sum_{j=0}^n\one_{W_\ell}(S^j\bx)=\sum_{j=0}^n\one_{W_\ell}(S^j \by)$$   
 for all $\ell\in\mathbb{N}$.

\begin{lem}\label{lem:friends to offset}
If $\bx$ and $\by$ are $\zeta_\bx(n)$-\friend, then $0<|\zeta_\by(n)-\zeta_\bx(n)|\leq 3$.
\end{lem}
\begin{proof}
Since $Y = X \setminus  \bigl(\bigcup_{\ell \notin \{10^k\colon k\geq 2\}}Z_\ell \cup \bigcup_{k=1}^{\infty} W_k\bigr)$ and
since $\bx$ and $\by$ are $n$-\friend, it follows that
$$\Big|\sum_{j=0}^{\zeta_{\bx}(n)}\one_Y(S^j\bx)-\sum_{j=0}^{\zeta_{\bx}(n)}\one_Y(S^j\by)\Big|=1.
$$
Assume $\sum_{j=0}^{\zeta_{\bx}(n)}\one_Y(S^j\by)=\sum_{j=0}^{\zeta_{\bx}(n)}\one_Y(S^j\bx)-1$, and so $\zeta_{\by}(n)=\zeta_{\bx}(n)+m$ where $m$ is the least integer such that  
$$\sum_{j=1}^m \one_Y(S^jS^{\zeta_{\bx}(n)}\by)=1.$$ 
To prove the statement, we are left with showing that $m\leq 3$. If $\bz \in X$, $\ell\in \mathbb{Z}$, and 
$S^\ell\bz,S^{\ell+1}\bz \notin Y$, then one of the two iterates lies in $Z_1$ (the only $D_\ell$ with 1st index not $6$) and the other lies in $\bigcup_{\ell=2}^\infty D_\ell$, and so $(S^{\ell+2}\bz)_1\notin \{6,7\}$ which means it lies in  $Y$. 
 It follows that $m\leq 3$, completing the proof. 
\end{proof}
We record part of the proof for future reference:
\begin{cor}\label{cor:friends to offset}For every $\bx\in Y$, $n \in \mathbb{Z}$ we have $|\zeta_\bx(n)|\leq 3|n|$. 
\end{cor}

\begin{notation}\label{notation:friends}
We introduce notation (namely $\CH_{n,\varepsilon}$) that is crucial for establishing that $T$ is prime, and is used extensively starting in Section~\ref{sec:getting friends}. To motivate its meaning, sets with $k$-friends play a key role in the proof, 
being used in Lemma~\ref{lem:easy friend} to establish a criterion that rules factors not being to the one point system. To do this, we invoke Lusin's Theorem and use that many pairs of friends share their initial entries. We keep this in mind and make the definition precise. 

Given $j \in \mathbb{Z}$, we say that $2$ disjoint measurable sets $\mathcal{A}_j, \mathcal{B}_j \subset Y$ of equal measure and a measure preserving map 
$G_j\colon \mathcal{A}_j \to \mathcal{B}_j$ are an \emph{ $(n, \varepsilon)$-triple for $j$} if $\nu(\mathcal{A}_j)=\nu(\mathcal{B}_j)>\varepsilon$, $\bx$ and $G_j(\bx)$ are $\zeta_{\bx}(j)$-friends, and 
$\bx_k=G(\bx)_k$ for all $k\leq n$. (Note this terminology  is local and is only used in this definition.) 
We define: 
\begin{equation}
\label{def:Cneps}
\CH_{n,\varepsilon}=\{j\colon  \text{ there exists  an }(n,\varepsilon) \text{-triple for }j.\}
\end{equation} 
\end{notation}

The next lemma is not used until Section~\ref{sec:restrict factors}, but as we aim to prove numbers are in $ \CH_{n,\varepsilon}$ in Section~\ref{sec:getting friends}, and we set 
up useful definitions to do this in Section~\ref{sec:recoding},  it is placed here for motivation.
In the next lemma we approximate a non-explicit measurable set by cylinders.

\begin{lem} \label{lem:easy friend} 
Assume that $(Y, \nu, T)$  has a non-trivial factor $(Z, \rho, R)$ with associated factor map $P\colon Y\to Z$.  Let 
$F\colon L^2(\nu)\to L^2(\nu)$ be  the Markov operator defined by 
 $P$ and 
further assume that $F$ is the limit (as $k\to\infty$) of $\sum \alpha_i^{(k)}U_{T^i}$, in the strong operator topology where $\alpha_i^{(k)}\geq 0$ for all $i,k$ and $\sum_i \alpha_i^{(k)}=1$ for all $k$. Then  for all small enough  ${\varepsilon}>0$   there exists 
 $N_0 = N_0(F, {\varepsilon})$ 
such that for all $N\geq N_0$ and sufficiently large $m$, 
\begin{equation}\label{eq:small friends}\sum_{j\in \CH_{N,\hat\varepsilon}}\alpha_j^{(m)}<{\varepsilon}.
\end{equation}
\end{lem}
Note that the proof only uses that the factor is not to a 1 point system, but is phrased this way for consistency with the results in Section~\ref{sec:restrict factors}.
\begin{proof}
Since $T$ is weakly mixing, $R$ is aperiodic and by  Rokhlin's Lemma, 
for any $\delta>0$,  there exists $V \subset Z$ such that $\rho(V)>\frac 1 4 -\delta$ 
 and such that  $V,RV,R^2V,R^3V$ are pairwise disjoint.
 Define $g\colon Z\to \mathbb{C}$ by setting 
 $$g=\one_V+\sqrt{-1} \one_{RV}-\one_{R^2V}-\sqrt{-1}\one_{R^3V}
 $$ and define $f\colon Y \to \mathbb{C}$ by $f=g\circ P$ to be the pullback of $g$ to $Y$. 
Choose $\tilde{f}$, taking values in $\{\sqrt{-1}^j\}_{j=0}^3$,  to be a 
finite  linear combination of characteristic functions of cylinder sets such 
 that $\nu(\{\bx\colon\tilde{f}(\bx)\neq f(\bx)\})<\delta$, and let $k$ be the largest defining index out of all of these cylinder sets. 
 
We claim that if $N>{k+1}$ and $n \in \CH_{N,\hat{\varepsilon}}$, then
$$\nu(\{\bx\colon|U_{T^n}\tilde{f}(\bx) -\tilde{f}(\bx)|>\frac 1 {\sqrt{2}}\})>\hat\varepsilon-11\delta.$$

To prove the claim, assume that 
$G_n\colon\mathcal{A}_n\to \mathcal{B}_n$ is the measure preserving 
bijection given  in the definition of $\CH_{n,\varepsilon}$ 
and define 
$$G(\bx)=\begin{cases} G_n(\bx)& \text{ if }\bx\in \mathcal{A}_n\\G_n^{-1}(\bx)& \text{ if }\bx \in \mathcal{B}_n \\ \bx & \text{ otherwise.}\end{cases}$$ 
 We restrict our attention to the set of points $\by$ of measure at least $\varepsilon-10\delta$ that satisfy the following properties: 
\begin{enumerate}
\item the points lie in $\mathcal{A}_n$ 
\item the points satisfy $\tilde{f}(\by)=f(\by)$ and $\tilde{f}(G(\by))=f(G(\by))$. 
\item the points are such that $P({T}^n\by)$ and $P(G({T}^n\by))$ lie in $V\cup RV\cup R^2V\cup R^3V$. 
\end{enumerate}
Then for any such point $\by$, we have that $\tilde{f}(\by)=\tilde{f}(G(\by))$ 
(because $y_i=G(y)_i$ for every $i$ in the defining indices of the cylinders defining $\tilde{f}$) and furthermore for some $1\leq m\leq 3$ (which may depend on $\by$)  we have
 $$\tilde{f}(T^n\by)=f(T^n\by)=\sqrt{-1}^m f(T^nG(\by))=\sqrt{-1}^m \tilde{f}(T^nG(\by))$$ (the second equality follows from Lemma~\ref{lem:friends to offset}). 
 Thus either $\tilde{f}(\by) \neq \tilde{f}(T^n \by)$ or $\tilde{f}(G(\by))\neq \tilde{f}(T^nG(\by))$. 
 Since $\tilde{f}$ takes values in $\{\sqrt{-1}^j\}_{j=0}^3$, if $\tilde{f}(\bx) \neq \tilde{f}(\bz)$ then
 $|\tilde{f}(\bx)-\tilde{f}(\bz)|\geq {\sqrt{2}}$ and the claim follows.

 By construction $\|F(\tilde{f})-\tilde{f}\|_2<4\delta$. However, if 
 $$\|U_{T^{n_j}}\tilde{f}-\tilde{f}\|_2>c,  \text{ for some } c> 0 \text{ and }\ \gamma_j\geq 0 \text{ satisying } 
 \sum \gamma_j\leq 1, $$
 then by taking a convex combination it follows that 
\begin{equation}
 \label{eq:convex far}
 \|\sum \gamma_j U_{T^{n_j}}\tilde{f}-\tilde{f}\|_2>Cc^2
 \end{equation}
for some constant $C> 0$. 
Indeed, using that $\|U_T\tilde{f}\|_2=\|\tilde{f}\|_2=1$ we see that $\langle \sum \gamma_j U_{T^{n_j}}\tilde{f},\tilde{f}\rangle\leq \langle f,f\rangle -c^2$. Because $\langle f,g \rangle=\|f\|_2 \cdot \|g\|_2 \cos(\langle(f,g))$, either $\|\sum \gamma_j U_{T^{n_j}}\tilde{f} \|_2 \leq \|f\|_2-\frac {c^2} 2$ or 
 $\cos(\langle \tilde{f}, \sum \gamma U_{T^{n_j}} \tilde{f}  )\geq \frac {c^2}2$ and in either case~\eqref{eq:convex far} follows. 
 Similarly if 
 $$\sum_{\{j\colon\|U_{T^{n_j}}\tilde{f}-\tilde{f}\|_2>c\}} \gamma_j>\varepsilon,
$$ then $\|\sum \gamma_j U_{T^{n_j}}\tilde{f}-\tilde{f}\|_2>C\varepsilon^2 c^2.$
Since $\delta$ is arbitrary, the lemma follows.
\end{proof}

\subsection{Recoding of time scales}\label{sec:recoding}
This section is devoted to relating  
$T$ iterated a large number of times to 
 $T$ iterated a smaller number of times, 
 or perhaps several smaller powers with accompanying subsets of $Y$. 
 
 This procedure, which we call reducing or the reductions, is 
carried out via Definition~\ref{def:reduce}, which also contains a parameter for testing how good this relation is. The defect of it is related to the notion of $n$-friends in the next section. To carry out the reduction, the next definition is a mechanism for computing the ``order of magnitude" of the relevant power of $T$. Note that this order of magnitude is used to identify cylinders where friends are contained. 

\begin{notation*}
Let 
\begin{equation}
\label{def:sigma}
\sigma_n=\max\{i\colon d_i(n)\neq 0\}, 
\end{equation}
where $d_i$ is defined as in~\eqref{def:di}.

Set \begin{equation}
\label{def:E} 
E=\{10^k\colon  k\geq 2 \}.
\end{equation}
\end{notation*}

Now if $\sigma_m \notin E$ it is relatively easy to see that $m \in \mathcal{H}_{\sigma_m-1,\varepsilon}$ (see Lemma~\ref{lem:no E friends}, whose proof uses $\sigma_m$ to identify the explicit cylinders which can be chosen to be the domain and codomain of $G_m$).  
If $\sigma_m \in E$, we seek to obtain
$m'$ where $T^m$ is ``close" to $T^{m'}$ and  $\sigma_{m'}<\sigma_m$. In this way, if $\sigma_m \in E$ and $\sigma_{m'}\notin E$ we can still show that $m \in \mathcal{H}_{\sigma_{m'}-1,\varepsilon}$: first using Lemma~\ref{lem:no E friends}, this time applied to $m'$, and second using the closeness of $T^m$ to $T^{m'}$, which is made precise in Lemma~\ref{lem:reduction close},  to show these same cylinders 
contain the domain and codomain of $G_j$; see Lemma~\ref{lem:reduce no E friends}. 
(While this specific example motivates reducing, we note that this is not the only way reductions are used,  and in particular it is used in Cases 1 and 2 in Section~\ref{sec:end}.)  We now consider two motivating examples: 
$T^{r_{10^{2k+1}}+2}\approx T^2$, because off of a small measure set $(T^{r_{10^{2k+1}}+2}\bx)_j =(T^2\bx)_j$ for all $j<10^{2k+1}$. There is a more complicated situation,  $T^{r_{10^{2k}}+2}$ is roughly $T^2$ on $\{\by \in Y\colon  \by_{10^{2k}}<k\}$ (off of a set of small measure) and 
$T^{r_{10^{2k}}+2}$ is roughly $T$ on $\{\by \in Y\colon  \by_{10^{2k}}\geq k\}$ (off of a set of small measure).  Note that $\sigma(r_{10^{2k+1}}+2)=10^{2k+1}$ and similarly for the other powers.  We make this recoding precise below by triples which keep track of the new powers in the first coordinate, the set where this approximation is relevant in the second coordinate and the measure of the set where this approximation fails in the third coordinate. Note, the third coordinate can also be related to friends (Lemma~\ref{lem:rigid friends}). 
The next definition defines an inductive procedure, and the relevant initial conditions are 
deferred until Definition~\ref{def:frak-sets}.

\begin{defin}\label{def:reduce}
Fix $N\in\N$ and $\varepsilon > 0$.  
For $r \geq 1$ and a set of triples $\mathfrak{H}_{r}(N,\varepsilon) \subset  \mathbb{Z}\times 
\mathcal{B}(Y)\times [0,1]$, we inductively  define 
the set of triples $\mathfrak{H}_{r+1}(N,\varepsilon)$ as follows: if $(j, A,\rho) \in \mathfrak{H}_{r}(N,\varepsilon)$  and at least one of the following conditions holds
\begin{enumerate}
\item $\sigma_{j}\notin E$
\item $j=0$
\item $\sigma_j\leq N$
\item $\rho>\varepsilon$, 
\end{enumerate}
then $(j, A,\rho) \in \mathfrak{H}_{r+1}(N,\varepsilon)$. 
Otherwise we modify the triple, depending on the value of $\sigma_j$: 
\begin{enumerate}
\item If  $\sigma_j\in\{10^{2\ell+1}\colon\ell \geq 1\}$,  then 
$$(j-d_{\sigma_j}(j){r}_j, A,\rho+\frac{|d_{\sigma_j}(j)|}{a_{\sigma_j}})\in \mathfrak{H}_{r+1}(N,\varepsilon).
$$
 \item 
   If $\sigma_j\in \{10^{2\ell}\colon\ell \geq 1\}$, then both
$$(j-d_{\sigma_j}(j){r}_j,A\cap \bigcup_{\ell<\frac{a_{\sigma_j}}2}\mathcal{C}_{\sigma_j}(\ell), \rho +\frac{|d_{\sigma_j}(j)|}{a_{\sigma_j}})\in \mathfrak{H}_{r+1}(N,\varepsilon)$$
 and 
 $$(j-d_{\sigma_j}(j)r_j+d_{\sigma_j}(j),
A\cap \bigcup_{\ell\geq \frac{a_{\sigma_j}}2}\mathcal{C}_{\sigma_j}(\ell), \rho +\frac{|d_{\sigma_j}(j)|}{a_{\sigma_j}})\in \mathfrak{H}_{r+1}(N,\varepsilon).$$ 
\end{enumerate}
\end{defin}

\begin{defin}
\label{def:frak-sets}
Fix $N\in\N$ and $\varepsilon > 0$.  

Define $\mathfrak{F}(N,\varepsilon)$ to be the collection of triples $\mathfrak{H}_{r}(N,\varepsilon)$ that stabilize with respect to $r$, 
meaning that 
$$
\mathfrak{F}(N,\varepsilon) = \mathfrak{H}_{r}(N,\varepsilon)  \text{ when } \mathfrak{H}_{r}(N,\varepsilon)=\mathfrak{H}_{r+1}(N,\varepsilon).$$ 
Define $\mathcal{F}(N,\varepsilon)$ to keep track of the measure of the sets in $\mathfrak{F}(N,\varepsilon)$, meaning that $$\mathcal{F}(N,\varepsilon)=\{(n,\nu(A),\rho)\colon (n,A,\rho) \in \mathfrak{F}(N,\varepsilon)\}.$$ 
If $\mathfrak{H}_0(N,\varepsilon)=(i,[0,1],0)$ for some $i\in \N$, we define
$\mathcal{F}_i(N,\varepsilon)$ 
to be the set $\mathcal{F}(N,\varepsilon)$. 

Similarly, define $\mathfrak{H}_{r,i}(N,\varepsilon)$ to be  $\mathfrak{H}_r(N,\varepsilon)$ if $\mathfrak{H}_0(N,\varepsilon)=\{(i,[0,1],0)\}$. 
We similarly define $\mathfrak{F}_i(N,\varepsilon)$ to be $\mathfrak{F}(N,\varepsilon)$  when 
$\mathfrak{H}_0(N,\varepsilon)=\{(i,[0,1],0)\}$. 
\end{defin}

Note that $\sigma_n$ is defined in~\eqref{def:sigma},  $a_i$ are defined in~\eqref{def:X}, and $r_i$ are defined in~\eqref{def:ri}.  We 
state a lemma that motivates the sets given in Definition~\eqref{def:frak-sets}. In particular, it shows how these definitions relate $T$ to a large power to $T$ to a smaller power, or possibly two smaller powers with relevant sets.

\begin{lem} \label{lem:reduction close}
Given $n\in\N$, let $\mathcal{C}$ be a cylinder defined by positions in $E$ that are greater than $\sigma_n$.  
\begin{enumerate}
\item Assume $\sigma_n$ is an odd power of $10$.  Setting $\tilde{n}=n-d_{\sigma_n}(n)r_{\sigma_n}$, we have 
$$\nu(\{x\in \mathcal{C}\colon  (T^{\tilde{n}}\bx)_i\neq T^n(\bx)_i \text{ for some }i< \sigma_{n}\})<4\nu(\mathcal{C}) \frac{|d_{\sigma_n}(n)|}{a_{\sigma_n}}.
$$

\item  Assume $\sigma_n$ is an even power of $10$.  Setting $n'=n-d_{\sigma_n}(n)(r_{\sigma_n}-1)$ and defining 
$A_1=\{\bx\colon x_{\sigma_n}\geq \frac{a_{\sigma_n}}2\}$, 
 we have
 $$\nu(\{\bx\in A_1\colon  (T^{n'}\bx)_i\neq T^n(\bx)_i \text{ for some }i< \sigma_{n'}\})<4\nu(\mathcal{C} \cap A_1) \frac{|d_{\sigma_n}(n)|}{a_{\sigma_n}}.$$
Furthermore, setting 
 $n''=n-d_{\sigma_n}(n)r_{\sigma_n}$, and defining 
 $A_2=\{\bx\colon x_{\sigma_n}<\frac{a_{\sigma_n}}2\}$, we have 
  $$\nu(\{\bx\in A_2\colon  (T^{n''}\bx)_i\neq T^n(\bx)_i \text{ for some }i< \sigma_{n''}\})<4\nu(\mathcal{C} \cap A_2) \frac{|d_{\sigma_n}(n)|}{a_{\sigma_n}}.$$
  \end{enumerate}
 \end{lem} 

 \begin{proof}  
For convenience, in this proof we assume $d_{\sigma_n}>0$ (the case $d_{\sigma_n}<0$ is similar).  
 Recall that $D_j$ is defined in~\eqref{def:Dk}. 
 Observe that  if $T^n(\bx)=S^{d_{\sigma_n}(n)q_{\sigma_n}}(T^{\tilde{n}}\bx)$, 
 then
  $T^n(\bx)_j=T^{\tilde{n}}(\bx)_j$ for all $j< \sigma_n$, and (by Lemma~\ref{lem:same hit}) this holds if 
\begin{equation}
\label{eq:estimate-D}
\sum_{i=0}^{d_{\sigma_n}(n)q_{\sigma_n}-1} \one_{\bigcup_{j=\sigma_n}^{\infty}D_j}(S^i\bx)=0.
\end{equation}  
 First we consider the case of $\sigma_n=10^j$ for $j$ odd. 

Since $\sigma_n \in \{10^{2k+1}\colon k\geq 1\}$, we have that 
$D_{\sigma_n}=\emptyset$ and so~\eqref{eq:estimate-D} 
fails
for a set of $\bx$ of $\mu$ measure at most 
\begin{equation}\label{eq:unchanged}
\frac{d_{\sigma_n}(n)}{a_{\sigma_n}}q_{\sigma_n}\mu(\bigcup_{j=\sigma_n+1}^\infty D_j)\leq \frac{d_{\sigma_n}(n)}{a_{\sigma_n}}.
\end{equation} 
Furthermore,
$$ \nu (\{\bx \in Y\colon \sum_{i=0}^{d_{\sigma_n}(n)q_{\sigma_n}-1} \one_{\bigcup_{j=\sigma_n}^{\infty}D_j}(S^i\bx)\neq 0\})\leq   3d_{\sigma_n}(n)q_n\mu( Z_{\sigma_{n+1}})=3 \cdot \frac 1 8 \frac{d_{\sigma_n}(n)}{a_{\sigma_n}}.$$
(Restricting to $\bx\in Y$ and converting from $\mu$ to $\nu$ changes this by a factor of less than 3.)

The next two cases are similar, but a bit more 
complicated as $D_{\sigma_n}$ is not empty, but 
is equal to $W_\ell$ for some $\ell$. 
If $\bx \in A_1$, then the conclusion holds if 
$$\sum_{i=0}^{d_{\sigma_n}(n)q_{\sigma_n}-1}\one_{\bigcup_{j=\sigma_n}^\infty D_j }(S^{i}T^{n'}\bx)=0.$$
Indeed, if  $\sum_{i=0}^{q_{\sigma_n}-1}\one_{\bigcup_{j=\sigma_n}^\infty D_j }(S^{i}\by)=0$, 
then  this follows from Lemma~\ref{lem:same hit} and the fact that ${\sum_{i=0}^{q_{\sigma_n}-1}\one_{\bigcup_{j=\sigma_n}^\infty D_j}(S^i\bzero)=1}$. So, 
$$T^{r_{\sigma_n}}(\by)_j=\begin{cases}(T^{-1}\by)_j& \text{ for }j \neq \sigma_n\\ (T^{-1}\by)_j+1 & \text{ for }j<\sigma_n.\end{cases}$$
Thus this case follows analogously to~\eqref{eq:unchanged} above after estimating 
$$\mu(\{\by\in A_1\colon S^i\by \in A_1\setminus (\cup_{j=\sigma_n}^\infty D_j )=A_1\setminus (\cup_{j=\sigma_n+1}^\infty D_j ) \text{ for all }i\leq d_{\sigma_n}(n)q_{\sigma_n}\}).$$ 
This is at most $\frac{d_{\sigma_n}(n)}{a_{\sigma_n}}$.

The third case is similar: 
if $\bx \in A_2$, then the conclusion holds if 
$$\sum_{i=0}^{d_{\sigma_n}(n)q_{\sigma_n}-1} \one_{\bigcup_{j=\sigma_n+1}^\infty D_j }(S^i\bx)=0
$$ and 
$$\sum_{i=0}^{d_{\sigma_n}(n)q_{\sigma_n}-1} \one_{D_{\sigma_n}}(S^i\bx)=d_{\sigma_n}(n),$$
where $10^{2\ell}=\sigma_n$. The remainder of the proof is analogous to the first case. 
 \end{proof}
Motivated by the sets in Lemma~\ref{lem:reduction close}, we make 
a few more definitions. 
If $(n,A,\rho) \in  \mathfrak{H}_{r, i }(N,\varepsilon)$,  let  
 $$P_r(n,A,\rho)=\{\bx\in A\colon (T^n\bx)_j\neq (T^i\bx)_j \text{ for some }j\leq\sigma_n \}$$ and 
 $$Q_r(n,A,\rho)= \{\bx\in A\colon (T^n\bx)_j= (T^i\bx)_j \text{ for all }j\leq \sigma_n\}.$$

 Define 
 \begin{equation} 
\label{eq:bad} 
\BR_r= \underset{(n,A,\rho)\in  \mathfrak{H}_{r}(N,\varepsilon)}{\bigcup} P_r(n,A,\rho)
\end{equation}
and 
\begin{equation}
\label{eq:good}
\GR_r=\underset{(n,A,\rho)\in \mathfrak{H}_r(N,\varepsilon)}{\bigcup}Q_r(n,A,\rho).
\end{equation}

\begin{lem}\label{lem:cylinders containing bad} 
Assume $\sigma_i\in E$ and let $A$ be a cylinder with all defining indices at least $\sigma_i$.  
Let $\mathfrak{H}_0(N,\varepsilon)=\{(i,[0,1],0)\}$. There exist cylinders $C_1,\ldots,C_\ell$ defined in positions greater than or equal to $\sigma_i$ such that the following hold: 
\begin{enumerate}
\item $A \cap \BR_{1} \subset \bigcup_{j=1}^\ell C_j$.  
\item $99 \nu(A \cap \BR_{1})>\nu(\bigcup_{j=1}^\ell C_j).$ 
\end{enumerate} 
\end{lem}
\begin{proof} 
 We treat $i$ with $\sigma_i \in \{10^{2k+1}\}$. 
Consider the set of $\by \in Y$ such that~\eqref{eq:estimate-D} fails.  We cover this set
by cylinders and show that $\nu(\BR)$ is proportional to the union of these cylinders. 
 The set $D_\ell$  requires that $x_j=a_{j-2}$ for all $j <\ell$, and so $S^{-d_{\sigma_i}(i)q_{\sigma_i}}(\bigcup_{\ell=\sigma_i+1}^\infty D_\ell)$ is contained in at most $d_{\sigma_i}(\sigma_i)+1$ cylinders defined 
 by the position $\sigma_i$. Furthermore,  
 $$\mathcal{P}\supset \{\by\in Y\colon \sum_{j=0}^{d_{\sigma_i}(i)q_{\sigma_i}-1}\one_{D_{\sigma_i+1}}(S^j\by)=1 \text{ and } \sum_{j=0}^{n_{\by}-1}\one_{D_{\sigma_i+1}}(S^j\by)=0\}$$ where $n_{\by}$ is the first coordinate of $(n_{\by},B,\rho)\in \mathfrak{H}_1(N,\varepsilon)$ and $\by \in B$.   This set has measure at least 
 $(\frac 1 {8}- \frac 1 {10^3})\frac{d_{\sigma_i}(i)}{a_{\sigma_i}}$. 
 
 The argument for $i$ with $\sigma_i \in \{10^{2k}\}$ is similar, but slightly complicated analogously to the proof of Lemma~\ref{lem:reduction close}, because $D_{\sigma_i}=W_{\frac 1 2 a_{\sigma_i}}$.
\end{proof}

\subsection{Obtaining friends}\label{sec:getting friends}
This section illustrates how the imperfections in the reduction process and the termination of the reduction process are related to the presence of $n$-friends. Indeed, we show the parameter $\rho$ in Definition~\ref{def:reduce} is proportional to the measure of a set of points that have friends and if $\sigma_n\notin E$ then $n \in \CH_{\sigma_n-1,\frac 1 {99}}$.

The proofs are technical and so we outline the strategy and complications. The idea is similar to the proof of Proposition~\ref{prop:wmix}.  As in that proof, we identify a sequence of particular sets (in the proof of Proposition~\ref{prop:wmix} this is $Z_{10^k-k}$) and produce $n$-friends by choosing pairs of points where one 
lands in the set and the other does not. 
(In the proof of Proposition 3.4, the points in $B_i$ hit this $D_j$ and the points in $A_i$ do not.)  
There are two complications in the proofs of this section that do not arise in the proof of Proposition \ref{prop:wmix}. 
%
%
%
%
The first issue is that in Proposition~\ref{prop:wmix}, we can choose the iterates, but 
Theorem~\ref{thm:CE factor} does not have this freedom because we can not pick which coefficients $\alpha_i^{(k)}$ in Corollary \ref{cor:CE factor} are non-negative. 
This forces us to analyze various cases, depending on whether $\sigma_i\notin E$ (Lemma~\ref{lem:no E friends}) or $\sigma_i \in E$ (Lemma~\ref{lem:rigid friends}), where $E$ is the set defined in~\eqref{def:E}. 
Furthermore, when $\sigma_i\notin E$, there are further cases to consider, 
depending whether either of $\sigma_{i+1}$ or $\sigma_{i-1}$ are in  $E$ (see the proof of Lemma~\ref{lem:no E friends}). Additionally, we need to use a more ``pointwise approach." Rather than having two sets, $A_i$, $B_i$ such that $\xi_{\bz}(i)$ is constant on each set and $\xi_{\by}(i)-\xi_{\bx}(i)=1$ for all $(\by,\bx)\in B_i\times A_i$, as we did in Proposition \ref{prop:wmix}, we define a set $A$ and a map $G$ such that $\bx$ and $G(\bx)$ are $\zeta_{\bx}(m)$-friends for all $\bx \in A$. In particular, we do not claim $\zeta_{\bz}(m)$ is well behaved as $\bz$ varies in $A$.  (Recall that $\xi$ and $\zeta$ are morally ``inverses" of each other and while $\xi$ was more convenient for the proof of Proposition~\ref{prop:wmix}, $\zeta$ is more convenient here and in the remainder of the proof.) 


The second issue is that we have to take care that our arguments 
work with the recoding procedure introduced in Section~\ref{sec:recoding}.  
This is carried out in Lemmas~\ref{lem:repeat rigid friends} and~\ref{lem:reduce no E friends}, which are versions of Lemma \ref{lem:no E friends} and \ref{lem:rigid friends} adapted to the recoding procedure. 
These complications are already reflected in Lemmas~\ref{lem:no E friends} and~\ref{lem:rigid friends}, 
as it no longer suffices to produce cylinders where a  definite proportion of their points  that can be paired to be $i$-friends, but rather  we require entire cylinders that can be paired in this manner. 

\begin{lem}\label{lem:no E friends}  
If $\sigma_m\notin E$, 
then $m\in \CH_{\sigma_m-1,\frac 1 {99}}$. 
Furthermore,  if $G\colon\mathcal{A}_m\to \mathcal{B}_m$ is the
measure preserving bijection 
associated to $\CH_{\sigma_m-1,\frac 1 {99}}$ as defined in~\eqref{def:Cneps}, then 
 $\mathcal{A}_m$ and $G(\mathcal{A}_m)$ can be chosen to be 
  a
 union of cylinders whose defining indices are a subset of 
 $\sigma_m-1,\, \sigma_m$, and $\sigma_m+1$.  
\end{lem}

\begin{proof}
Assume $\sigma_m\notin E$ and set $k=\sigma_m$.  Recall that $d_k = d_k(m)$ is defined in~\eqref{def:di}.  
Assume that $d_k\in \{1,2,3,4\}$ (the case that $d_k \in \{-1,-2,-3,-4\}$ is analogous). 
 Set $x_k=0$, $x_{k-1}=5$, and $x_{k+1}=4$ for whichever of $k-1$ and $k+1$ do not lie in $E$.  
 Whenever
  $k-1$ or $k+1$ lies in $E$, 
  we  stipulate that 
   $x_{k-1}$ or $x_{k+1} \in (\frac{a_{k-1}}2,a_{k-1}-3)$. 
  Set  $y_k=7$ and $x_j=y_j$ for all other $j$. 
We claim that if $\bf{x},\bf{y}\in Y$ are as above, then they are $\zeta_{\bx} = \zeta_{\bx}(m)$-\friend.
We first check that 
 $$\sum_{j=0}^{\zeta_\bx}\one_{Z_\ell}(S^j\by)=\sum_{j=0}^{\zeta_\bx}\one_{Z_\ell}(S^j\bx)$$ 
 for all $\ell< k$. 
 To see this, note that the inclusion $S^j\bz \in Z_\ell$
 depends only on $z_1,\ldots,z_\ell$ and we have that 
 $x_j=y_j$ for all $j<k$. Likewise if $10^{2\ell}<k$, 
 then $\sum_{j=0}^{\zeta_\bx}\one_{W_\ell}(S^j\by)=\sum_{j=0}^{\zeta_\bx}\one_{W_\ell}(S^j\bx)$. 
 Also note that  because 
  $S^j(\bx)_{k+1}, S^j(\by)_{k+1}\neq a_{k+1}-2$ for all $j\leq \zeta_\bx$,
  we have
 $$\sum_{j=0}^{\zeta_\bx}\one_{Z_\ell}(S^j\by)=\sum_{j=0}^{\zeta_\bx}\one_{Z_\ell}(S^j\bx)=0$$ 
 for all $\ell> k+1$  and 
 $$\sum_{j=0}^{\zeta_\bx}\one_{W_\ell}(S^j\by)=\sum_{j=0}^{\zeta_\bx}\one_{W_\ell}(S^j\bx)=0$$
for all  $10^{2\ell}>k+1$. 
  Now, if $k+1\notin E$, then since $(S^j\bx)_{k+1}, (S^j\by)_{k+1}\neq a_{k+1}-1$ 
 we have 
 $$\sum_{j=0}^{\zeta_\bx}\one_{Z_{k+1}}(S^j\by)=\sum_{j=0}^{\zeta_\bx}\one_{Z_{k+1}}(S^j\bx)=0.$$
If $k+1\in E$, then since 
 $y_{k+1} = x_{k+1} >\frac{a_{k+1}}2$,  
we have 
 $$\sum_{j=0}^{\zeta_\bx}\one_{W_{\ell}}(S^j\by)=\sum_{j=0}^{\zeta_\bx}\one_{W_{\ell}}(S^j\bx)=0,$$
where $10^{\ell}=k+1$.

Lastly, since $\zeta_\bx>\frac 5 8 a_kq_{k-1}$, 
we have that by the condition on the digits $k$ and $k-1$ of $\by$, 
$$\sum_{j=0}^{\zeta_\bx}\one_{Z_k}(S^j\by)=1.$$  
But since $\zeta_\bx<5 q_k$, using that $x_k=1$ we have that  
$(S^j\bx)_k<7$ for all $0\leq j\leq \zeta_\bx$ and so $\sum_{j=1}^{\zeta_\bx}\one_{Z_k}(S^j\bx)=0$. 
  This proves the claim that $\bx$ and $\by$ are $\zeta_{\bx} = \zeta_{\bx}(m)$-\friend\   and $G$ is the 
  bijection taking $\bx$ to $\by$. (That is, changing the $k^{\text{th}}$ entry from 0 to 7.)
  \end{proof}
  
\begin{lem}\label{lem:rigid friends}
If $\sigma_m \in E$, then there exist cylinder sets $K_1,\ldots,K_r$  defined on 
the entries $\sigma_m+1$, $\sigma_m$, and $\sigma_m-1$ 
such that 
\begin{equation}\label{eq:rigid bound}
\nu(\bigcup_{j=1}^r K_j) > \frac 1 2\cdot \frac 1 3\cdot \frac{|d_{\sigma_m}(m)|}{a_{\sigma_m}}\cdot \frac 1 {64}
\end{equation} 
and there exists a measure preserving map  
$G\colon\bigcup_{j=1}^r K_j \to Y \setminus \bigcup_{j=1}^r  K_j$  defined by changing the $\sigma_{m}+1$ entry   such that 
 if $\bx \in \bigcup_{j=1}^r  K_j$, 
then $\bx$ and $G(\bx)$ are $\zeta_{\bx}(m)$-\friend. Moreover, $K_1,\ldots K_r,G(K_1),\ldots G(K_r)$ are disjoint cylinders. 
\end{lem}

\begin{proof}
Let $10^k=\sigma_m$ and assume that $d_{10^k}(m)>0$ (the case that 
$d_{10^k}(m)<0$ is similar). Let  $x_{10^{k}+1}=0$ and $y_{10^k+1}=7$. Let 
$$y_{10^k}\in \bigl\{k-2,k-3,\ldots,k-1-\min\{{d_{10^k}(m)},\frac 1 3 a_k\}\bigr\}$$ and set $y_{10^k-1}=5$. 
Furthermore, set $x_\ell=y_\ell$ for all $\ell \neq 10^k+1$. 
It is straightforward that $\sum^{\zeta_\bx(m)}_{j=0}\one_{Z_{10^{k}+1}}(S^j\bx)=0$ and $\sum^{\zeta_\bx(m)}_{j=0}\one_{Z_{10^{k}+1}}(S^j{\by})=1$. 

 We claim that  
$\one_V(S^j\bx)=\one_V(S^j\by)$ for all $|j|\leq \zeta_\bx(m)$, 
where $V$ is either $Z_\ell$ for $\ell \neq 10^k+1$ or $V$ is 
any $W_\ell$.  To see this,
 for $Z_\ell$ with $\ell<10^k$ and $W_\ell$ with $\ell \leq k$, this 
holds since $\by$ and $\bx$ agree in the relevant entries. 
Furthermore, 
$(S^j\bx)_{10^{k}+1}$ and $(S^j\by)_{10^k+1}$ are not 6 in this range and so we never land in $Z_\ell$ for $\ell>10^k+1$ or in $W_\ell$ for $\ell>k$, 
proving the claim.  
 Assuming $K_j$ is a cylinder set as in the statement, define $G(\bx)$ to change the $10^{k}+1$ entry from $0$ to $7$ (leaving all the other entries unchanged).  Thus $G$ satisfies all of the announced properties. 
 
 Now 
 \begin{multline*}
 \mu(\{\bx\colon x_{10^k}\in\bigl\{k-2,k-3,\ldots,k-1-\min\{{d_{10^k}(m)},\frac 1 3 a_k\}\bigr\}, \\
  x_{10^k+1}=0, \, \text{ and }x_{10^k-1}=5\})
 \geq \frac 1 3 \cdot \frac{|d_{\sigma_m}(m)|}{k} \cdot\frac 1 {64}.
 \end{multline*}
  Considering the set of such $\bx\in Y$ so that $\by\in Y$ as well with $x_i=y_i$ for all $i\neq \sigma_{m}+1=10^k+1$ and $y_{10^k+1}=7$ and (trivially) converting to $\nu$ establishes~\eqref{eq:rigid bound}. 
\end{proof}

\begin{lem}\label{lem:repeat rigid friends}
Assume $r\geq1$, $\varepsilon<\frac 1 {8 \cdot 99}$,  $\mathfrak{H}_0(N,\varepsilon)=\{(i,[0,1],0)\}$, $(n,A,\rho)\in \mathfrak{H}_{r}(N,\varepsilon)$  and $\sigma_n\in E$. 
We can choose $B_1,\ldots,B_\ell \subset A$ to be cylinders whose defining indices are at least $\sigma_{n-1}$ such that $\nu(\bigcup_{j=1}^{\ell} B_j)> \frac 1 {6 \cdot 64} \cdot  \frac{|d_{\sigma_n}(n)|}{a_{\sigma_n}}$ and such that there exists $\tilde{B}\subset \bigcup_{j=1}^{\ell} B_j$ with 
\begin{equation}\label{eq:most together}\nu(\tilde{B})\geq (1-4\cdot 99\rho)\nu(\bigcup_{j=1}^{\ell} B_j), 
\end{equation}
$G\colon\tilde{B} \to Y$ a measure preserving injection, defined by changing the $\sigma_n+1$ position, and thus $G(\tilde{B})\subset A$, such that $\bx$ and $G(\bx)$ are $\zeta_{\bx}(i)$-\friend \, and $x_{j}=G(\bx)_{j}$ for all ${j}<\sigma_n$.
\end{lem}
\begin{proof}
Let $\hat{B}_1,\ldots,\hat{B}_{\ell}$ be the cylinders and $\hat{G}$ be the function given by Lemma~\ref{lem:rigid friends}  
applied with $m=n$. 
Set $B_i=\hat{B}_i\cap A$ and let $\tilde{B}$ be the set of points in $\bigcup_{j=1}^{\ell} B_j \cap \GR_r\cap G^{-1}(\GR_r)$. 
Let $(n',A,\rho')\in \mathfrak{H}_{r-1}(N,\varepsilon)$ be the predecessor of $(n,A,\rho)$.
We claim that because $\sigma_{n'}>\sigma_n+1$, if $\bx,\by \in \GR_r $  then for all $k$. 
  \begin{equation}\label{eq:n same}
  \sum_{j=0}^{\zeta_{\bx}(n)}\one_{Z_k}(S^j\bx)-\sum_{j=0}^{\zeta_{\bx}(n)}\one_{Z_{k}}(S^j\by)=\sum_{j=0}^{\zeta_{\bx}(i)}\one_{Z_k}(S^j\bx)-\sum_{j=0}^{\zeta_{\bx}(i)}\one_{Z_{k}}(S^j\by).
  \end{equation} 
  We first 
consider the case of $k< \sigma_{n'}$. The sums on the left hand side of~\eqref{eq:n same} are either $\lfloor \frac n {q_k}\rfloor$ or $\lceil \frac n {q_k} \rceil$, 
while 
on the right hand side they are either 
$\lfloor \frac i {q_k}\rfloor$ or $\lceil \frac i {q_k} \rceil$; 
the choice of $\lfloor\cdot\rfloor$ or $\lfloor\cdot\rfloor+1=\lceil\cdot\rceil$
depends on comparing $x_j$ and $(S^n\bx)_j$ for 
the left sums on each side,  and similarly $y_j$ and $(S^n\by)_j$ for the right sums on each side, for $j\leq k$. By our assumption that $\bx, \, \by\in \GR_r$, we have that $(S^n\bx)_j$ is the same as $(S^i\bx)_j$ for all $j\leq k< n'$ and so whether the first sum on the left hand side is the floor function or one more is the same for the first sum on the right hand side. The case of $\by$ is identical. 

Next consider the case of $k>\sigma_{n+1}$ (since $\sigma_n+1<\sigma_{n'}$ this covers $k\geq \sigma_{n'}$). 
We have that
$\sum_{j=0}^{\zeta_{\bx}(n)}\one_{Z_k}(S^j\bx)=\lfloor \frac{\zeta_\bx(n)}{q_k}\rfloor$ and $\sum_{j=0}^{\zeta_{\bx}(n)}\one_{Z_{k}}(S^j\by)=\lfloor \frac{\zeta_{\by}(n)}{q_k}\rfloor$ 
by the argument in 
Lemma~\ref{lem:rigid friends}.  (This argument only depends on the  cylinders with defining entries in  positions $\sigma_{n}-1$,  $\sigma_n$, and $\sigma_{n}+1$ that define the cylinders in the proof of Lemma~\ref{lem:rigid friends} and these entries are the same for $\hat{B}_i$.) 
For the right hand side, as above $(S^n\bx)_j=(S^i\by)_j$ for $j\in \{n,n+1\}$, 
so whether we take the floor or ceiling in the summands on the right hand side depends on $j>n+1$. These are the same for $\bx$ and $\by$ by construction, 
giving~\eqref{eq:n same}. So in the left hand side both summands take the floor and for the right hand side they either both take the floor or both take the ceiling, establishing~\eqref{eq:n same}. 

A similar computation yields
  \begin{equation}\label{eq:n same2}
  \sum_{j=0}^{\zeta_{\bx}(n)}\one_{W_\ell}(S^j\bx)-\sum_{j=0}^{\zeta_{\bx}(n)}\one_{W_{\ell}}(S^j\by)=\sum_{j=0}^{\zeta_{\bx}(i)}\one_{W_\ell}(S^j\bx)-\sum_{j=0}^{\zeta_{\bx}(i)}\one_{W_{\ell}}(S^j\by).
  \end{equation} 
To complete the lemma we 
are left with establishing~\eqref{eq:most together}. To check this, 
we claim that it suffices to show that $\GR_r$ 
 can be chosen to be unions of cylinders defined by entries with positions at least $10^{\log_{10}(\sigma_n)+1}$.   
This follows from the following:

\noindent
\textbf{Claim.} For all $\delta>0$, there exists $k\in\N$ such that if $C_1,C_2$ are cylinders with the smallest entry defining $C_2$ at least $k$ larger than the largest entry defining $C_1$, 
then 
   \begin{equation}\label{eq:indep}
   \frac{\nu(C_1\cap C_2)}{\nu(C_1)\nu(C_2)}\in [1-\delta,1+\delta].
   \end{equation}
  
\noindent
\textit{Proof of Claim:} To check that the claim holds, let $L$ be the smallest entry defining $C_2$. 
Let $U_1,\ldots.,U_m$ be the cylinders given by proscribing the first $L-1$ terms that intersect $Y$.   All but one of these cylinders are also a cylinders in $X$, and so they all have the same $\nu$ measure.  If  $U_i$ is the one cylinder set in $Y$ that is not also a cylinder set in $X$, then  $U_i$ has smaller $\nu$ measure than the other $m-1$ cylinders.  Assume $U_1,\ldots,U_{m'}$ are those cylinders that are contained in $C_1$. If $i \in \{1,\ldots,m'\}$, then $\nu(C_1\cap C_2)<\nu(C_1)\nu(C_2)$, but it is at least 
   $\frac{m'-1}{m'}\nu(C_1)\nu(C_2)$. Similarly, if $i \notin \{1,\ldots,m'\}$, then 
   $\nu(C_1\cap C_2)<\nu(C_1)\nu(C_2)$, but it is at most $\frac{m}{m-1}\nu(C_1)\nu(C_2)$.
 Since $\tilde{B}$ is a union of the sets $B_i$ that pairwise satisfy this 
 condition, the claim follows. \hfill $\qed$ 
 
 Finally we check that the sufficient condition, meaning that $\GR_r$ can be chosen to be unions of cylinders defined by entries with positions at least $10^{\log_{10}(\sigma_n)+1}$,  holds.  Namely, 
   by iterating  Lemma~\ref{lem:cylinders containing bad} and  using the assumption that $\varepsilon<\frac 1 {8\cdot 99}$, we obtain the complement of the cylinders that cover $\mathcal{P}$  and so the $\nu$ measure of these cylinders contained in $\mathcal{Q}_r$ is at least $\frac 1 2 $;
   indeed, by Lemmas~\ref{lem:reduction close} and~\ref{lem:cylinders containing bad} the $\nu$ measure of the cylinders that cover $\mathcal{P}$ are at most $4 \cdot 99\rho$. 
    \end{proof}

\begin{lem} \label{lem:reduce no E friends}
Assume  $r\geq 1$,  $\varepsilon<\frac 1 {8\cdot 99}$,  $\mathfrak{H}_0(N,\varepsilon)=\{(i,[0,1],0)\}$,   $(n,A,\rho)\in \mathfrak{H}_r(N,\varepsilon)$,
 and $\sigma_n \notin E$. Let $(n',A',\rho') \in \mathfrak{H}_{r-1}(N,\varepsilon)$ so that $A\subset A'$. 
 Then there exist cylinders $B_1,\ldots,B_\ell \subset A'$ 
 defined by positions whose entries are at least $\sigma_n-2$, 
 such that there exists
 $\tilde{B}\subset \bigcup_{j=1}^\ell  B_j$ with $\nu(\tilde{B})> \frac 1 {999}\nu(A) \geq \frac 1 2\cdot \frac 1 {999}\nu(A') $
   and a map $G\colon\tilde{B}\to Y$
   such that  
  $\bx$ and $G(\bx)$ are $\zeta_{\bx}(i)$-\friend \, for all $\bx \in \bigcup_{j=1}^\ell B_j$. Moreover, $x_j=G(\bx)_j$ for all $j<\sigma_n$. Thus $G(\tilde{B})\subset A'$ as well. 
 \end{lem}
 \begin{proof} 
We first prove the statement under the assumption that $\sigma_n\neq \sigma_{n'}-1$.  In this case, 
 let $\hat{B}_1,\ldots,\hat{B}_\ell$ be the cylinders and let $\hat{G}\colon \hat{A}\to\hat{B} $ be the map given by Lemma~\ref{lem:no E friends} for $m=n$. 
 Let $B_j=\hat{B}_j \cap A$ and 
 $\tilde{B}=\bigcup_{j=1}^\ell  B_j \cap Q_r \cap \hat{G}^{-1}Q_r$. 

 Repeating the proof used to derive~\eqref{eq:indep}, we obtain cylinders defined in entries at most $\sigma_{n-2}$, with the entry before the last place defining the cylinders in $B_j$ (and also smaller than the cylinders defining $Q_r$ and $A$). 
As in that proof, one of these cylinders differs from the cylinder with the same defining entries in $X$. On all the other cylinders, $B_j$ intersects $A \cap Q_r$ as expected and the lemma follows using an argument 
analogous to the proof of~\eqref{eq:indep}.  

 Now we treat the remaining case, $\sigma_{n}=\sigma_{n'}-1$ and the largest entry of the cylinders defining $B_j$ overlaps with the smallest entry of the cylinders defining $A$. In this case, 
we consider $A'$ whose defining entries are all larger than $\sigma_n$ (they are at least $10^{k+1}$ where the smallest entry defining $A$ is $10^k$). Then $(n',A',\rho')$ has two descendants in $\CH_r(N,\varepsilon)$, $(n,A,\rho)$ and $(\tilde{n},\tilde{A},\tilde{\rho})$. One of these is $A \cap \{\bx\colon x_{10^k}\geq \frac k 2\}$ and so by the definition of $B_j$ in Lemma~\ref{lem:no E friends} has nonempty intersection with the cylinders $\hat{B}_1,\ldots, \hat{B}_\ell$. 
The proof then follows as above, via the same arguments used to conclude the proof of Lemma~\ref{lem:repeat rigid friends}.\end{proof}

\subsection{Restricting factors}\label{sec:restrict factors} In this section, 
we develop our main criteria to rule out factors, Propositions~\ref{prop:no friends} and~\ref{prop:norig} and Corollary~\ref{cor:to quote}. Morally, Proposition~\ref{prop:no friends} and Corollary~\ref{cor:to quote} 
rely on the assumption that the factor map is not to the 1 point system, while Proposition~\ref{prop:norig} relies on the assumption that the factor map is not an isomorphism, though for technical reasons it is helpful not to disambiguate these situations (in particular the proof of Proposition~\ref{prop:norig} uses Proposition~\ref{prop:no friends}). 

For approximating, we make use of a metric giving rise to the strong operator topology on $L^2(\nu)$. 
While any such metric suffices for our purposes, it is convenient 
to choose one that simplifies the computations:
\begin{notation*}  Let $\mathcal B = \mathcal B(L^2(\nu))$ denote the set of continuous linear operators on $L^2(\nu)$ and 
let $\{f_i\}_{i=1}^{\infty}$ be an orthonormal basis for $L^2(\nu)$ such 
that $\|f_i\|_{\infty}=1$ for all $i$.  
Set $D\colon \mathcal{B}\times \mathcal{B}
\to [0,\infty)$ 
to be the metric defined by 
\begin{equation}
\label{eq:def-D}
D(U,V)=\sum_{i=1}^\infty 2^{-i}\|Uf_i-Vf_i\|_2.
\end{equation}
\end{notation*}
Note that restricting $D$ to the set (the choice of $10$ is any arbitrary positive real)
$$\{(U,V)\in \mathcal{B}\times\mathcal{B} \colon \|U\|_{\op}+\|V\|_{\op}\leq 10\}
$$
endows this set 
with the strong operator topology.

Recall that the notation $\mathcal{F}_i(N,\varepsilon)$ is introduced in Definition~\ref{def:frak-sets}.
\begin{prop} \label{prop:no friends} 
Assume that $(Y, \nu,T)$  has a non-trivial factor with associated factor map $P\colon Y\to Z$ and let  
$F\colon L^2(\nu)\to L^2(\nu)$ be the Markov operator defined by
 $P$ and 
$\underset{k \to \infty}{\lim}\, D(\sum_i \alpha_i^{(k)} U_{T^i} ,F)=0$ where $\alpha_i^{(k)}\geq0$ for all $i,k$ and $\sum_i \alpha_i^{(k)}=1$ for all $k$.  
Then for all $\varepsilon>0$, 
there exists $N_0$ such that for all $N>N_0$ and all large enough $k$, 
\begin{equation}\label{eq:reduce}
 \sum_{i}\alpha_i^{(k)} \sum_{\{(n,\beta,\rho)\in \mathcal{F}_{i}(N,\varepsilon)\colon \sigma_n>N\}}\beta<\varepsilon. 
\end{equation}
\end{prop}
We record an immediate corollary for later use:
\begin{cor}\label{cor:to quote} Assume that $(Y, \nu,T)$  has a non-trivial factor with associated factor map $P\colon Y\to Z$ and let  
$F\colon L^2(\nu)\to L^2(\nu)$ be the
 Markov operator defined by 
$P$, then for all small enough $\varepsilon>0$ there exists $N_0$ and $\delta>0$ such that if $a_i\geq 0,$ $\sum a_i=1$ and 
$$D(\sum_i a_i U_{T^i},F)<\delta$$ 
then for all $N\geq N_0$
$$
 \sum_{i}a_i \sum_{\{(n,\beta,\rho)\in \mathcal{F}_{i}(N,\varepsilon)\colon \sigma_n>N\}}\beta<\varepsilon. 
$$
\end{cor}

\begin{proof}[Proof of Proposition~\ref{prop:no friends}]
Let $\varepsilon>0$ be small enough such that Lemma~\ref{lem:easy friend} is satisfied. We proceed by contradiction, and show that if 
$$
 \sum_{i  }\alpha_i^{(k)} \sum_{\{(n,\beta,\rho)\in \mathcal{F}_{i}(N, \varepsilon )  \colon \sigma_n>N\}}\beta>10c, 
$$ 
then 
\begin{equation}\label{eq:shifts}
\sum_{i\in \CH_{N,\frac {c \varepsilon }{9999}}}\alpha_i^{(k)}>c.
\end{equation}
Then by taking $N$ sufficiently large, we obtain a contradiction via Lemma~\ref{lem:easy friend}.

  Let  $J\subset \mathbb{N}$ be the  set of indices $i$  such 
  that $$\sum_{\{(n,\beta,\rho)\in \mathcal{F}_i(N, \varepsilon ) \colon \sigma_n>N\}}\beta>c.
  $$ 
Since 
$\sum_{i\notin J}\alpha_i^{(k)}\sum_{\{(n,\beta,\rho)\in \mathcal{F}_i(N, \varepsilon ) \colon  \sigma_n>N\}}\beta\leq c$, 
it follows that  $\sum_{i\in J}\alpha_i^{(k)}>c$.

\begin{claim}  
\label{claim:i-in-J}
For any $i \in J$, we have $i\in \CH_{N,\frac {c \varepsilon }{9999}}$. 
\end{claim}
To check this, the triples defined in Definition~\ref{def:reduce} give the two possible reasons for $(n,A,\rho) \in \mathfrak{F}({N, \varepsilon})$ with $\sigma_n>N$.
The first is that $\sigma_n \notin E$, in which case Lemmas~\ref{lem:no E friends} and~\ref{lem:reduce no E friends} give a set of points, contained in $A$, which have $i$-\friend\, of  measure at least $\frac 1 {999}\nu(A)$ 
 and a map $G$ defined on these symbols, identifying the point with its friend so that $G(\bx)_j=x_j$ for all $j< \sigma_n$. 
The second is that $\rho> \varepsilon $, in which case Lemmas~\ref{lem:rigid friends} and~\ref{lem:repeat rigid friends} similarly give cylinders with measure at least $\frac 1 {6 \cdot 64}  \varepsilon \cdot (1-4\cdot 99 \varepsilon )\nu(A)$.  So if $i$ is such that $\sum_{\{(n,\beta,\rho)\in \mathcal{F}_i(N, \varepsilon )\colon \sigma_n>N\}}\beta>c$, then since the $\beta$ are the measure of the sets  $A$ mentioned in the previous 2 sentences, we have 
$i \in \CH_{N,\frac{c \varepsilon }{9999}}$.  
\end{proof}

The next two lemmas (Lemmas~\ref{lem:reduction close 2}  and~\ref{lem:B}) are only used in the proof of 
Proposition~\ref{prop:norig}.

\begin{lem}\label{lem:reduction close 2} 
Assume that $(Y, \nu,T)$  has a non-trivial factor with associated factor map $P\colon Y\to Z$ and assume that 
$F\colon L^2(\nu)\to L^2(\nu)$ is  the Markov operator defined by 
 $P$.  
Then for all $c, \varepsilon >0$, there exists $N_0$ such that for all $N>N_0$ we have that 
\begin{equation}\label{eq:most not in B}
\sum_{i \in B_{c,\varepsilon}} \alpha_i^{(k)}<\sqrt{\varepsilon}, 
\end{equation}
where $$B_{c,\varepsilon}=\Big\{i\colon D\big(U_{T^i}, \sum_{(n,A,\beta) \in \mathfrak{F}_i(N,\varepsilon)}\one_A \cdot U_{T^n}\big)>10\sqrt{\varepsilon}+c\Big\}.$$ 
\end{lem}
This notation $B_{c,\varepsilon}$ is local and only used until the end of this section. 
\begin{proof} We claim that for all $c>0$ and small enough $\varepsilon>0$, there exists $N$ such that if $S_1,\,S_2$ are measure preserving transformations such that $\nu(\{x\colon( S_1\bx)_j=(S_2 \bx)_j \text{ for all }j<N\})>1-\varepsilon$ then 
$D(U_{S_1},U_{S_2})<c+3\sqrt{\varepsilon}$. 

To prove the claim, given $N\in\N$, set 
$$A(N) = \{\bx\in X\colon  (S_1\bx)_j=(S_2\bx)_j \text{ for } j<N\}.  $$
By Lusin's Theorem and uniform integrability, 
for any $f \in L^2(\nu)$,
there exists $N\in\N$ such that if $A = A(N) $, then
$\|(f\circ S_1) \one_A-(f\circ S_2)\one_A\|_2<c$. 
 As in the definition of $D$,  let $\{f_i\}_{i=1}^\infty$ be an orthonormal basis of 
 $L^2(\nu)$ with $\Vert{f_i}\Vert_\infty = 1$ for all $i\in\N$.
Given $\frac 1 4 >\varepsilon > 0$, choose $k$ such that $2^{-k}<\varepsilon$ and pick $N$ sufficiently large 
such that the associated set $A = A(N)$ ensures that 
$\|(f_i\circ S_1) \one_A-(f_i\circ S_2)\one_A\|_2<c$ for $i=1, \ldots, k$. 
Then $\nu(A) > 1-\varepsilon$ for some $\varepsilon> 0$
 and so (the definition of the metric $D$ is given in~\eqref{eq:def-D}) 
\begin{multline*}
D(U_{S_1},U_{S_2})\leq c+\sum_{i=1}^\infty 2^{-i}\bigl(\int_{\one_{A^c}}|f_i\circ S_1-f_i\circ f_2|^2d\nu\bigr)^{1/2}{+\sum_{i=k+1}^{\infty}\|f_i\circ S_1-f_i\circ S_2\|_2}\\
\leq c+2\sqrt{\varepsilon}{+2\varepsilon\leq c+3\sqrt{\varepsilon}}, 
\end{multline*}
proving the claim.  

We now complete the proof by contraposition.
Take $N_{1} = N$ where $N$ is sufficiently large such that the above 
claim holds. 
Then for any $N > N_{1}$,
if $$D(U_{T^i},\sum_{\{(n,A,\rho) \in \mathfrak{F}_i(N,\varepsilon)\}}
\one_A 
\cdot U_T^n)>c+10\sqrt{\varepsilon},$$
then the claim implies that 
$$\nu(\{\bx\colon  (T^i\bx)_j \neq (T^n\bx)_j \text{ for }j>N \text{ and }\bx \in A \text{ where } (n,A,\rho)\in \mathfrak{F}_i(N,\varepsilon)\})>3 \varepsilon.$$
By~Lemma~\ref{lem:reduction close},  for each $i$ 
such that this claim holds, we have 
$$\sum_{\{(n,\beta,\rho)\in {\mathcal{F}}_i(N,\varepsilon)\}} \rho\beta>\varepsilon.$$ Further observe that if such a 
 triple $(n,\beta,\rho)$ has $\rho>\varepsilon$, then $\sigma_n>N$.  
Assuming the negation of~\eqref{eq:most not in B}, we check that this gives the opposite of the bound in~\eqref{eq:reduce}. 
Taking $N_0$ as in Proposition~\ref{prop:no friends} for $F$ and ${\varepsilon}$, then if $N> N_0$ 
we have that $F$ is given by 
the factor map to the one point system. Taking $N_0$ in the statement of Lemma~\ref{lem:reduction close 2} 
  to be the maximum of $N_1$ and the $N_0$ from Proposition~\ref{prop:no friends}, 
the statement follows. \end{proof}

\begin{lem} 
\label{lem:B}
Let $([0,1],R,\tau)$ be an ergodic measure preserving system and $(Z,R',\tau')$ be a factor with factor map $P$ that is not an isomorphism and let $F$ be the Markov operator defined by $P$.  Then 
there exists a set $B\subset[0,1)$ such that  $\tau(B)\geq \frac 1 3$ and $( F\one_B)(x)<\frac 1 2 $ for all $x\in B$. 
\end{lem}

\begin{proof}
By disintegration of measures, there exist probability measures $\tau_{P(x)}$ carried on $P^{-1}(P(x))$ such that 
$\tau=\int_Z \tau_zd\tau'(z)$.  
Let $m_x= {\rm ess}\,\inf\{y\colon\tau_{P(x)}([0,y])> \frac 1 2 \}$.  
We claim that 
\begin{equation}\label{eq:atom bound}
\frac 1 3 \leq \tau_{P(x)}([0,m_x))\leq \frac 1 2.
\end{equation} Indeed, if $\tau_{P(x)}$ is non-atomic, then  $\tau_{P(x)}([0,m_x))=\frac 1 2$. 
If  $\tau_{P(x)}$ is atomic, then by assumption there are at least two atoms and by ergodicity the atoms 
are of equal size.   Because the largest atom of $\tau_{P(x)}$ has measure at most $\frac 1 2$, and there are a sum of atoms with measure between $\frac 1 3$ and $\frac 1 2$ inclusive, we have~\eqref{eq:atom bound}.

We claim that $x\mapsto m_x$ is a Borel measurable function. First recall that by disintegration of measures, the map $x \mapsto \tau_{P(x)}$ is Borel. 
Next, for every interval $[0,a]$, we have that the map from Borel measures to real numbers, $\mathfrak{p} \mapsto \mathfrak{p}([0,a])$ is a Borel map. 
It follows that the function $h(x,y)=\tau_{P(x)}[0,y]$ is Borel 
and so $h^{-1}\big((\frac 1 2 ,1]\big)$ is Borel and for each $x$, $m_x=\text{ess inf}\{y:(x,y)\in h^{-1}\big((\frac 1 2 ,1]\big)\}$. Note that since $\tau_{P(x)}$ is a probability measure this is the same as 
$$\min\{\text{ess inf}\{y:(x,y)\in h^{-1}\big((\frac 1 2 ,1]\big)\},1\}.$$
Now if $A \subset [0,1]^2$ is Borel, then $x\mapsto \min\{{\rm ess}\,\inf\{y\colon (x,y) \in A\},1\}=m_x$ is Borel 
measurable. Indeed, the set of such $A$ is a monotone class containing the algebra generated by rectangles. 
To see this, for countable nested unions,
$$\max\{{\rm ess}\,\inf\{y\colon (x,y) \in \cup_{i=1}^\infty A_i\},1\}=\underset{n \to \infty}{ \lim}\, \min \{{\rm ess}\,\inf\ \{y\colon (x,y) \in \cup_{i=1}^n A_i\},1\}.$$ 
For countable nested intersections, we have that 
$\min\{{\rm ess}\,\inf\{y\colon(x,y) \in \cap_{i=1}^\infty A_i\},1\}$ can be defined piecewise as 
\begin{equation*}
\begin{cases}1 & \text{ if }\{y:(x,y) \in \cap_{i=1}^\infty A_i\} \text{ is a zero set}\\
\underset{n \to \infty}{\lim}\,  \min\{{\rm ess}\,\inf\{y\colon(x,y) \in \cap_{i=1}^n A_i\},1\} & \text{ otherwise.}
\end{cases}
\end{equation*}
 By the Monotone Class Theorem it is defined on the smallest $\sigma$-algebra containing rectangles, 
 which is the Borel $\sigma$-algebra. 
We set $B=\bigcup_{x\in[0,1]}[0,m_x)$.   Then for almost every $x$, $\tau_{P(x)}([0,m_x))\leq \frac 1 2 $ and so $F(\one_B)(x)\leq \frac 1 2 \one_B(x)$ for all $x\in B$. 
\end{proof}
Note that this lemma holds for any Lebesgue space and in particular to $(Y,\nu,T)$.

\begin{prop}\label{prop:norig}
Assume that $(Y, \nu,T)$  has a non-trivial factor with associated factor map $P\colon Y\to Z$ and let $F$ be the Markov operator defined a factor map $P$. Let 
$\underset{k \to \infty}{\lim}\, D(\sum_i \alpha_i^{(k)} U_{T^i} ,F)=0$ where $\alpha_i^{(k)}\geq0$ for all $i,k$ and $\sum_i \alpha_i^{(k)}=1$ for all $k$.  
  Then for all $\frac 1 {99}>\varepsilon>0$, there exists $N_0$ such that for any $N>N_0$ and all sufficiently large enough  $k$, we have 
\begin{equation}\label{eq:norig}
 \sum_{i }\alpha_i^{(k)} \sum_{\{(n,\beta,\rho)\in \mathcal{F}_{i}(N,\varepsilon)\colon  {n \neq} 0\}}\beta>{\frac 1 3}\cdot(\frac 1 2 -3\varepsilon).
\end{equation}
\end{prop}
\begin{proof}
Because $F \neq \Id$, 
 for almost every 
$\bx\in Y$
 we have $P^{-1}(P\bx)$ is at least two points, and so there exists a set $B$ as in Lemma~\ref{lem:B}.
By  assumption and Lemma~\ref{lem:reduction close 2}, there exist $k_1,N_1$ such that for all $k\geq k_1$
 and $N > N_1$ 
 we have 
\begin{equation}\label{eq:prop norig}
\nu(\{\bx:|(F\one_B)(\bx)-(\sum_i\alpha^{(k)}_i\sum_{(n,A,\rho)\in \mathfrak{F}_i(N,\varepsilon)}\one_A(\bx)(U_{T^n}\one_B)(\bx)|>\frac \varepsilon 9\})<\frac \varepsilon 9.
\end{equation} 

By the non-negativity of the $\alpha_i^{(k)}$, for almost every 
$\bx\in Y$
we have 
$$ (\sum_i \alpha_i^{(k)}\sum_{(n,A,\rho) \in \mathfrak{F}_i(N,\varepsilon)}\one_AU_{T^n}\one_B)(\bx)\geq \sum_i \alpha_i^{(k)}
\sum_{\{(n,A,\rho) \in \mathfrak{F}_i(N,\varepsilon)\colon n=0\text{ and }\bx \in A\}}
 \one_B(\bx).$$ 
If~\eqref{eq:norig} does not hold, 
then there is a set of $\bx$ of measure at least $\frac 2 3 +\varepsilon$
such that 
$$ \sum_{i }\alpha_i^{(k)} (\sum_{\{(n,A,\rho)\in \mathfrak{F}_{i}(N,\varepsilon)\colon  n= 0 \text{ and }\bx \in A\}}\beta)>\frac 1 2 +\varepsilon. $$
But then there is a set of $\bx\in B$ of measure at least $\varepsilon$ 
such that for each $\bx$ in this set,  $$\Big(\sum_i \alpha_i^{(k)}\sum_{(n,A,\rho) \in \mathfrak{F}_i(N,\varepsilon)}\one_A(x)U_{T^n}\one_B\Big)(\bx)>\frac 1 2 + \varepsilon \geq (F\one_B)(\bx)+\frac\varepsilon2, $$ a contradiction of~\eqref{eq:prop norig}.
\end{proof}

\section{The behavior of a projection}\label{sec:end}

\subsection{Overview of the proof that $(Y,\nu,T)$ is prime}
In this section, we show that our constructed system is prime: 
\begin{thm}
\label{thm:prime}
The system $(Y, \nu,T)$ is prime.
\end{thm}

We start with an overview of the proof and then proceed to study different cases. 
We assume that $(Y, \nu, T)$ has a non-trivial factor $Z$ with factor map $P\colon Y\to Z$ 
and assume that  $F\colon L^2(\nu)\to L^2(\nu)$ is the 
 Markov operator defined by   $P$.  
We further assume 
that $F$ is the limit, as $k\to\infty$, of $\sum \alpha_i^{(k)}U_{T^i}$ in the strong operator topology 
(Corollary~\ref{cor:CE factor}). 
Given $\varepsilon > 0$, by Proposition~\ref{prop:no friends} we can assume that there exists $N_0\in\N$ such that for all $N> N_0$ and sufficiently large $k$, we have that 
$$ \sum_i \alpha_i^{(k)}\sum_{(n,A,\rho)\in \mathfrak{F}_i}\one_AU_{T^n}$$
 gives a good approximation to $\sum \alpha_i^{(k)}U_{T^i}$ (which in turn leads to a good approximation for $F$).
 The general idea in the proof of Theorem~\ref{thm:prime} is that we rule out the 
possibility that $\sum \alpha_i^{(k)}U_{T^i}$ is close to a non-trivial projection. 
The key facts used are that the composition of projections is still a projection and by properties of the strong operator topology, we may assume that for any fixed $M$, for all large enough $k$  
\begin{equation}\label{eq:composition}
\underset{M \ {\rm times}}{\underbrace{(\sum \alpha_i^{(k)}U_{T^i}) \circ (\sum \alpha_i^{(k)}U_{T^i})\ldots \circ (\sum \alpha_i^{(k)}U_{T^i}}})
\end{equation}
 is close to $F^M=F$, and this is also close to   $\sum \alpha_i^{(k)}U_{T^i}$. 
We then use the fact that~\eqref{eq:composition} is
\begin{equation}\label{eq:composition coeff}\sum_{(i_1,\dots,i_M)}(\prod_{m=1}^M\
  \alpha_{i_m}^{(k)}) U_{T^{\sum i_m}}, 
 \end{equation}
  and apply Definition~\ref{def:reduce} to~\eqref{eq:composition coeff}.  That is, we study 
  $$ \sum_{(i_1,\dots,i_M)}(\prod_{m=1}^M\
  \alpha_{i_m}^{(k)}) \sum_{(n,A,\rho)\in \mathfrak{F}_{\sum_{j=1}^M i_j}(N,\frac \varepsilon 2)}\one_{A}U_{T^n}.$$
  Treating 3 different cases, this allows us to produce friends and obtain a contradiction via Lemma~\ref{lem:easy friend}. 
  We now make this precise.  

\subsection{Set up for the proof of Theorem~\ref{thm:prime}} \label{sec:setup}
We begin a proof by contradiction, assuming that there is a Markov operator $F$ coming from a non-trivial factor map. By Theorem~\ref{thm:CE factor},  there exists $\alpha_i^{(k)}\geq 0$ with $\sum_i \alpha_i^{(k)}=1$ for all $i$  such that $\sum_i \alpha_i^{(k)}U_{T^i}$ converges in the strong operator topology to $F$. 

We assume that  $\varepsilon>0$ is sufficiently small such that all of the Lemmas and Propositions in Section~\ref{sec:coding} hold.  That is, $\varepsilon<\min\{\frac 1 {8 \cdot 99 }, \frac 1 {9999}\}$  and small enough such that Lemma~\ref{lem:easy friend} holds. We also assume that 
\begin{equation}\label{eq:later epsilon bound} \varepsilon<\frac 1 {10^5 \cdot 9999},
\end{equation} which is to be used in 
Lemma~\ref{lem:still pretty bad}.  Furthermore, we choose $N_1>6$  (this choice is made to simplify the analysis in the third case we consider) to be sufficiently large such that Lemma~\ref{lem:easy friend} holds for $\frac {N_1} 2$ and $\varepsilon^8$ and such that 
\begin{equation} \label{eq:another bound} 2^{-\frac {N_1} 2}<\varepsilon^4.\end{equation} 

Setting 
$$
G_N=\{n\colon \sum_{\{(i ,c,\gamma)\in  \mathcal{F}_{n}(\frac N 2 ,\varepsilon^2)\colon \sigma_i> \frac N 2\}}c<\varepsilon^4\},$$
 our choices imply that for all sufficiently large $k$, 
\begin{equation}\label{eq:plenty good} 
\sum_{n  \in G_N}\alpha_{n}^{(k)}>\frac 3 4.  
\end{equation}
 Indeed, by Claim~\ref{claim:i-in-J} if $n\notin G_N$ then $n \in \mathcal{H}_{N,\frac{\varepsilon^2 \varepsilon^4}{9999}}$. Since $\frac{\varepsilon^6 }{9999}>\varepsilon^8$, by Lemma~\ref{lem:easy friend} this contradicts that our factor map is not to the 1 point system. 
We set $s =\min\{10^j\colon  10^{j-1}\geq N\}$, set $s'=\min\{10^\ell\colon 10^\ell\geq N\}$, and 
recall that $r_i$ is defined in~\eqref{def:ri}.  
Define 
\begin{equation}\label{eq:N and M}
M = M_N = \frac{r_{s- 2}}{2r_N}.
\end{equation}
Although $M = M_N$ depends on $N$, as $N$ is fixed at this point, we omit it from the notation, except at one step in the proof of Proposition~\ref{prop:iterate-a-lot}.  
The motivation behind this definition of $M$ is given by the following lemma (this plays a role in the proof of Proposition~\ref{prop:iterate-a-lot}): 
 \begin{lem} 
 \label{lem:M-is-large}
 For all sufficiently large $N$, we have that $M^{\frac 1 8} >r_{s'+1}$, where $M$ is defined as in~\eqref{eq:N and M}. Moreover, for any $\varepsilon>0$, for all sufficiently large $N$ we have $8^{(1-\varepsilon)10^k} < r_{10^k-1} < 8^{(1+\varepsilon)10^k}$.  
 \end{lem}
 \begin{proof} 
 We first claim that for all $\varepsilon>0$, there exists $k$ such 
 that
 \begin{equation}\label{eq: bound-on-r}8^{(1-\varepsilon)10^k} < r_{10^k-1} < 8^{(1+\varepsilon)10^k}.  
 \end{equation}
 For all $\ell \notin E$,  we have that $r_{\ell+1}=a_{\ell+1}r_\ell-1$.  Thus there exists $\ell_0$ with $r_{\ell+1}\geq 8^{1-\frac \varepsilon 2}r_\ell$ for all $\ell\geq \ell_0$ and the lower bound follows. 
For the upper bound, we have $r_k\leq q_k\leq 8^k \prod_{j=1}^{\lfloor \log_{10} k\rfloor}j$. It is straightforward that for all $\varepsilon>0$, there exists $k_0$ such that $8^k \prod_{j=1}^{\lfloor \log_{10} k\rfloor}j<8^{(1+\varepsilon)k}$ for all $k\geq k_0$.  

  For all large enough $N$, by~\eqref{eq: bound-on-r} 
 we have
 \begin{equation*}
 M=M_N = \frac{r_{s-1}}{2r_N}>\frac {\frac 1 8r_s}{2r_{s'}}>\frac{8^{(1-\varepsilon)10^k}}{8^{(1+\varepsilon)10^{k-1}}}>8\cdot8^{8.5(1+\varepsilon)10^{k-1}}>r_{s'+1}^{8},
 \end{equation*} 
 where the second to last inequality holds for all sufficiently large $N$. 
 \end{proof}

\subsection{The three cases:}
This leads us to consider three possibilities for the behavior of the projection $F$ on $L^2(\mu)$ (recall that $D$ is the metric defined in~\eqref{eq:def-D}): 

\noindent
\textbf{Case 1:} 
 \begin{equation}\label{eq:case 1} 
 D\Bigl(\sum_{(i_1,\dots,i_M)}\prod_{m=1}^M\
  \alpha_{i_m}^{(k)} U_{T^{\sum i_m}}, 
( \sum_{i} \alpha_i^{(k)}\sum_{\{j\colon 
 (j,A,\rho)  \in \mathfrak{F}_{i}(N,\varepsilon):\sigma_j<N\}} U_{T^j}\one_A\bigr)^M\Bigr)< \varepsilon^{\frac{1}{4}}.  
 \end{equation}
 
 \noindent
 \textbf{Case 2:} Case 1 does not hold and 
 \begin{equation}\label{eq:case 2}
 \sum_n\alpha_n^{(k)} \sum_{\{(j,\beta,\rho)\in \mathcal{F}_{n}(N,\varepsilon^2)\colon \sigma_j >N\}}\beta<\frac {\varepsilon^2} M.
 \end{equation}

\noindent
\textbf{Case 3:} Case 1 does not hold and 
\begin{equation}\label{eq:case 3}
 \sum_n\alpha_n^{(k)} \sum_{\{(j,\beta,\rho)\in \mathcal{F}_{n}(N,\varepsilon^2)\colon\sigma_j >N\}}\beta\geq \frac {\varepsilon^2} M.
 \end{equation}
 
 We analyze each of these cases separately.  

\subsection{Case 1:}
Fix $\varepsilon > 0$ and assume that~\eqref{eq:case 1} holds. 
 Roughly speaking,  the assumption means that when we iterate the approximation of 
 our transformation given by Definition~\ref{def:reduce} up to $M$ times,  we remain close to the original map.

 \begin{prop} 
 \label{prop:iterate-a-lot}
 There exists $N_{3}$ such that if $N>N_{3}$ and $\gamma$ is a probability measure  supported  on $\{-2r_N,-2r_N+1,\ldots,-1,0,1,\ldots,2r_N\}$ with $\gamma(\{0\})< \frac 1 7$, 
  and $M$ corresponds to $N$ as in~\eqref{eq:N and M}, 
  then  $$\gamma^M\bigl(\{(i_1,\ldots,i_M)\colon r_{s-2}> |\sum_{j=1}^Mi_j|>r_{s'+1}\}\bigr)>\frac 1 9.$$
 \end{prop}

 \begin{proof}
Let $(\Omega,\PP)$ be a probability space and let $F_1,\ldots,F_{M}\colon (\Omega,\PP)
\to \mathbb{Z}$ be a sequence of independent $\gamma$ distributed random variables, let 
$Z=\sum_{i=1}^{M}F_i-M\mathbb{E}_{\PP}(F_1)$, and let $\sigma$ be the variance of $F_1$.
Then the variance of $Z$ is $M\sigma$.  

By our choice of $M=M_N$, we have that $|\sum_{j=1}^Mi_j|{\leq 2r_NM}<r_{s-2}$ for all $i_j$ in the support of $\gamma$ and the remainder of the proof is devoted to showing the lower bound. 

By Lemma~\ref{lem:M-is-large}, for sufficiently large $N$, we have that $M^{\frac{1}{8}}$ is bounded from below. 
The proof splits into two cases.  In the first, $|\mathbb{E}_{\PP}(F_1)|$ is not too small and we 
make use of Chebyshev's inequality (as in the proof of the weak law of large numbers).  In the second case, $|\mathbb{E}_{\PP}(F_1)|$ is small, and using the central limit theorem we show that for many $\omega$, $|Z(\omega)|\sim \sqrt{M\sigma}$.

First assume that $\frac 1 {99}<|\mathbb{E}_{\PP}(F_1)|$. 
We compute $\mathbb{E}_{\PP}(Z^2)$ and apply Chebyshev's inequality.  By independence of the $F_i$, we have that 
$$\mathbb{E}_{\PP}(Z^2) = M\mathbb{E}_{\PP} \bigl((F_1^2-2F_1\mathbb{E}_{\PP}(F_1)+\mathbb{E}_{\PP}(F_1)^2)\bigr).$$
 Applying H\"older's Inequality,  we bound $\|F_i^2\|$ by $\|F_i\|_\infty \cdot \|F_i\|$. 
 By Lemma~\ref{lem:M-is-large}, we have that $\|F_i\|\leq 2r_N<M^{\frac 1 8}$ and similarly, $\|F_i\|_\infty<M^{\frac18}$.  Then  $\mathbb{E}_{\PP}(Z^2)\leq MM^{\frac 1 4 }$
 and so it follows from Chebyshev's Inequality that $\mathbb{P}(\{x\colon |f(x)|>\varepsilon \int f^2d\PP\})|<\frac 1{\varepsilon^{2} \int f^2 d\PP}.$ 
Therefore,  $$\mathbb E_{\PP}(\{\omega\colon |Z(\omega)|<4M^{\frac 3 4}\})>\frac 1 2 .$$ 
 Since $\sum F_i=M\mathbb{E}_{\PP}(F_1)+Z$ and $|M\mathbb{E}_{\PP}(F_1)|>\frac M {99}$, 
 we have that if $|Z(\omega)|<4M^{\frac 34}$,   then
  $$|\sum_{i=1}^M
   F_i(\omega)|>\frac{M}{99}-4M^{\frac 34}>M^{\frac 1 2}.$$

Now we consider the case that $\frac 1 {99}\geq |\mathbb{E}_{\PP}(F_1)|$. Under this assumption, 
because $\gamma(\{0\})<\frac 1 7$ and is supported on integers,  we have that 
$$|\mathbb{E}_{\PP}(F_1-\mathbb{E}_{\PP}(F_1)^2)|\leq 
\frac 67(1-\frac 1 {99})^2,$$ which implies that the variance $\sigma$ is at least $\frac 1 8$.

Let $\varepsilon > 0$.  
There exists $N_0$ such that for all $N\geq N_0$, and 
any probability measure $\mathfrak{p}$ on $[-r_N,r_N]$, we have 
$$\frac{1}{(M_{N})^{\frac k 2}}{\int |t|^k d\mathfrak{p}(t)}<\frac \varepsilon {2^k}$$ 
 for all $k\geq 3$.
It follows  that $\phi_t=\mathbb{E}e^{\frac{itZ}{\sqrt{M\sigma}}}$ is bounded by  
$-\frac{t^2}{2}+c(t)$, where $|c(t)|<\varepsilon$. 

 We use L\'evy's Continuity Theorem to complete the proof. 
 Namely, we  
 repeat this process for a sequence of $N_j$ tending to infinity 
 and obtain $Z_j, M_j$, and $ \sigma_j$ such 
 that $\phi^{(j)}_t=\mathbb{E}_{\PP}e^{\frac{itZ_j}{\sqrt{M_j\sigma_j}}} \to e^{-\frac{t^2}{2}}$ pointwise. Since $e^{-{t^2}/2}$ is continuous at 0, it follows that that $\frac{Z_j}{\sqrt{M_j\sigma_j}}$ converges in distribution. 
 Thus  for  sufficiently large $j$, 
 it follows that the probability that   $|Z_j|>\frac{\sqrt{M_j}}{16}$ is at 
 least $1/9$. If the expectation of $F_1$ is nonnegative, 
 this implies that  
 $$ \gamma^M \Bigl(\{(i_1,\ldots,i_M)\colon \sum_{j=1}^Mi_j>r_{s'+1}\}\Bigr)>\frac 1 9.$$ 
 Otherwise, 
 $$\gamma^M\Bigl(\{(i_1,\ldots,i_M)\colon \sum_{j=1}^Mi_j<-r_{s'+1}\}\Bigr)>\frac 1 9,$$
 and the result follows.  
 \end{proof}

 \noindent
 \textbf{Notation:}
Let $\tilde{\varepsilon}$  be as in Section~\ref{sec:setup}, that is less than $\frac 1 {10^5 \cdot 9999}$  and small enough such that Lemma~\ref{lem:easy friend} holds.  Let $N_2$ be chosen according to Corollary~\ref{cor:to quote} for $\tilde{\varepsilon}^{2}$. 
Let $\tilde{N}\geq \max\{N_1,N_2,N_3\}$ and  $M=M_{\tilde{N}}$. 
 Let $\tilde{k}$ be chosen large enough such that 
\begin{enumerate}[label=(\Alph*)]
\item\label{constants:delta} $D((\sum_i \alpha_i^{{(}\tilde{k}{)}}U_{T^i})^M,F)<\delta$ where $\delta$ is as in Corollary~\ref{cor:to quote} for  $\varepsilon=\tilde{\varepsilon}$ and $N_0=\tilde{N}$
\item~\eqref{eq:another bound} with $G_{\tilde{N}}$ and~\eqref{eq:plenty good} hold. 
\end{enumerate}

\subsubsection*{Concluding the proof of Theorem~\ref{thm:prime} in case 1:} 
Let   $$A=\{(i_1,\ldots,i_M):\sigma_{\sum i_j}>\tilde{N} \text{ and } \sigma_{\sum i_j} \notin E\}.$$
Let $\gamma$ be the probability measure on $\mathbb{Z}$ given by $\gamma(i)=\alpha_i^{(\tilde k)}$.  
By Proposition~\ref{prop:norig}, we have $\gamma(0)<\frac 17$.
 By Proposition~\ref{prop:iterate-a-lot}, applied to $\gamma^M$ we have that $\sum_{A}\prod_{j=1}^M \alpha_{i_j}^{(\tilde{k})}>\frac 1 2 $. 
Indeed if $\sum_{j=1}^M i_j \in (r_{s'}, r_{s-1})$, then $\sigma_{\sum i_j} \notin E$. 
 By~\eqref{eq:composition coeff} and our choice of $\tilde{k}$,  we obtain a contradiction of Corollary~\ref{cor:to quote}. Thus this case can not occur. 

\subsection{Case 2:} 
In the absence of the first case,  
we are left with showing that for at least $\tilde{\varepsilon}^2$ of the sums $\sum_{j=1}^Mi_j$,  for at least $\tilde{\varepsilon}^2$ points we have $\sum_{j=1}^Mi_j$-friends. 
Roughly speaking, the idea is that  $\prod_{j=1}^MU_T^{i_j}$ under iteration does not stay close to $\prod_{\mathfrak{F}_{i_j}(\tilde{N},\tilde{\varepsilon})} \sum_{(\ell,A,\rho) \in \mathfrak{F}_{i_j}(\tilde{N},\tilde{\varepsilon})} \one_A U_T^\ell$,  
 and 
 so the sum
 $$\sum_{(i_1,\ldots,i_M)\in \mathbb{Z}^M} (\prod_{j=1}^M\alpha_{i_j}^{(\tilde{k})} )
 \sum_{\{(j,\beta,\rho)\in \mathcal{F}_{\sum_{\ell=1}^M i_\ell}(\tilde{N},\tilde{\varepsilon} )\colon\sigma_j>\tilde{N}\}}\beta $$
  becomes significant.   
 To make this precise, we deal with two cases separately, depending on the sizes of the sums in~\eqref{eq:case 2} and~\eqref{eq:case 3}. 

We start with case 2. That is we are assuming:
$$\sum_n\alpha_n^{(\tilde{k})} \sum_{\{(j,\beta,\rho)\in \mathcal{F}_{n}(\tilde{N},\tilde{\varepsilon}^2)\colon \sigma_j  >\tilde{N}\}}\beta<\frac {\tilde{\varepsilon}^2} M.$$
Given $n\in\N$, set $\mathfrak{H}_0=\{(n,[0,1],\rho)\}$ and define the \emph{reduction} 
 $\red_k(n)[\bx]=(m,A,\rho)$,  where 
  \begin{itemize}
 \item $(m,A,\rho)\in  \mathfrak{H}_r(\tilde{N},\tilde{\varepsilon})$ for some (smallest) $r$, 
 \item $ \bx \in A,$ and
 \item $ \sigma_m\leq 10^k$ or $(m,A ,\rho) \in \mathfrak{F}(N,\tilde{\varepsilon})$ (that is, $m \notin E$ or $\rho>\tilde{\varepsilon})$. 
  \end{itemize}
 Let ${\tred_k(n)[\bx]}$ denote the first coordinate of $\red_k(n)[\bx]$. 
We say that the sum $\sum_j i_j$ is \emph{treatable} if 
$$ \sum_{\{(k,\beta,\rho)\in \mathcal{F}_{i_j}(N,\varepsilon^2)\colon \sigma_k >N\}}\beta<\tilde{\varepsilon}
$$
for all choices of  $i_j$ 
 and 
the sum $\sum_j i_j$ is \emph{$\bx$-treatable} if for all $i_j$, the elements 
$(n,A,\rho)\in \mathfrak{F}_{i_j}({N},
\tilde{\varepsilon})$ satisfy $\bx \in A$ has  $\sigma_n\leq {N}$. (Recall, $\mathfrak{F}_{i_j}(N,\tilde{\varepsilon})$ is not necessarily a singleton, but each $\bf x\in Y$ is in the second coordinate of exactly one triple in $\mathfrak{F}_{i_j}({N,\tilde{\varepsilon}})$.)

\begin{lem}\label{lem:carrying friends}Assume $\sum i_j$ is \emph{$\bx$-treatable}, $A \subset Y$ so that $(n,A,\rho)=\red_k( \sum_{j=1}^M
i_j)[\bx]$,  
\begin{equation}\label{eq:neq}
\tred_k(\sum_{j=1}^M i_j)[\bx] \neq \sum_{j=1}^M \tred_k(i_j)[\bx]
\end{equation}
 and $k$ is maximal with this property.  Then $\red_{k+1}(\sum_{j=1}^Mi_j)[\bx]$ or $\red_k(\sum_{j=1}^Mi_j)[\bx]$ lies in $\mathfrak{F}_{\sum_{j=1}^Mi_j}(\tilde{N},\tilde{\varepsilon})$.  
\end{lem} 

\begin{proof}
Set 
\begin{equation}\label{eq:k+1 same} m=\sum_{j=1}^M
 \tred_{k+1}(i_j)[\bx] =\tred_{k+1}(\sum_{j=1}^Mi_j)[\bx].
 \end{equation} 

 First,  because $i_j$ is $\bx$-treatable, we have
 \begin{multline*} 
 \tred_{k+1}(i_j)[\bx]-\tred_k(i_j)[\bx]\in
 \{ d_{10^{k+1}}(\tred_{k+1}(i_j)[\bx])r_{10^{k+1}},\\
 d_{10^{k+1}}(\tred_{k+1}(i_j)[\bx])r_{10^{k+1}}+d_{10^{k+1}}(\tred_{k+1}(i_j)[\bx])\}
 \end{multline*} 
  for all $j$. Let $n=\sum_j \big(\tred_{k+1}(i_j)[\bx]-\tred_k(i_j)[\bx]\big)$ and by our assumption that $k$ is maximal satisfying~\eqref{eq:neq},  $n\neq \tred_{k+1}(m)[\bx]-\tred_k(m)[\bx]$.  We now treat a series of cases, each of which is straightforward. If $\sigma_{\tred_{k+1}(m)[\bx]}> 10^{k+1}$, then because $|\tred_{k+1}(m)[\bx]|\leq Mr_{10^{k+1}+1}<r_{10^{k+2}-1}$, $\sigma_{|\tred_{k+1}(m)[\bx]}\notin E$. Thus the reduction algorithm given in Definition~\ref{def:reduce} halts and the lemma follows. If $\sigma_{\tred_{k+1}(m)[\bx]}= 10^{k+1}$, then by choice of $M$ and $N$ we have 
 \begin{equation}\label{eq:small sum}
  |\sum_j\tred_{k}(i_j)([\bx])|<r_{10^{k+1}-2}.
  \end{equation}
  Also, because $\tred_{k+1}(m)[\bx]-\tred_{k}(m)[\bx], \tred_{k+1}(i_j)[\bx]-\tred_{k}(i_j)[\bx]$ are all multiples of either $r_{10^{k+1}}$ or $r_{10^{k+1}}+1$ (depending on $x_{10^{k+1}}$),  we have that $\tred_{k+1}(m)[\bx]-\tred_{k}(m)[\bx]=n+p$ where $|p|\geq r_{10^{k+1}}$. 
 
By the algorithm for representing numbers in terms of $d_i$, 
we have that if $|p|\geq r_{\sigma_\ell+2}$ then $|p|>5|\ell|$ and so
  \begin{equation}\label{eq:where to}  \sigma_{\ell+p}\geq \sigma_p-1.
  \end{equation}
Thus  by Equation~\eqref{eq:small sum}, 
  $$\sigma_{\tred_k(m)[\bx]}=\sigma_{\sum_j \tred_k(i_j)[\bx]-p} \geq {10^{k+1}-1}>10^k.$$ 
 Therefore, by the definition of $\red_k(\cdot)[\bx]$ we must have that $\red_k(m)[\bx]\in \mathfrak(N,\tilde{\varepsilon})$, that is, the algorithm halts in this case as well. In the final case, $\sigma_{\tred_{k+1}(m)[\bx]}< 10^{k+1}$ we have that $n=0$ and either $\sigma_{\tred_{k+1}(m)[\bx]}\leq 10^k$ in which case $\red_{k}(m)[\bx]=\red_{k+1}(m)[\bx]$ and this is 
  $\sum \tred_{k}(i_j)[\bx]$ a contradiction, or $10^k<\sigma_{\tred_{k+1}(m)[\bx]}<10^{k+1}$ and so the algorithm stops. 
    \end{proof}

\begin{proof}[Concluding the proof of Theorem~\ref{thm:prime} in case 2]
Let 
$$I(\bx,M)=\{(i_1,\ldots,i_M)\colon  \sum i_j \text{ is }\bx\text{-treatable and }\sum_{j=1}^M\tred(i_j)[\bx]\neq \tred(\sum_{j=1}^M i_j)[\bx]\}.$$ 
We assume that we are  not in the first case and moreover~\eqref{eq:case 2} holds. 
Thus for a set of $\bx$ of measure at least $\tilde{\varepsilon}$,  we have
$$\sum_{\{(i_1,\ldots,i_M)\colon (i_1,\ldots,i_M)\in I(\bx,M)\}}\ \prod_{j=1}^M\alpha_{i_j}^{(\tilde{k})}\geq \tilde{\varepsilon}.$$
For each such $\bx$, $(i_1,\ldots,i_M)$  there exists $m$, $A \subset [0,1]$, $\rho\geq0$ such that $\bx \in A$ and $(m,A,\rho)\in \mathfrak{F}_{\sum_{j=1}^Mi_j}(\tilde N,\tilde \varepsilon)$ and 
$\sigma_m=k>\tilde N$.  
Thus we have 
$$\sum_{(i_1,\ldots,i_M)\in \mathbb{Z}^M} (\prod_{j=1}^M\alpha_{i_j}^{(\tilde{k})} )
 \sum_{\{(j,\beta,\rho)\in \mathcal{F}_{\sum_{\ell=1}^M i_\ell}(\tilde{N},\tilde{\varepsilon} )\colon\sigma_j>\tilde{N}\}}\beta \geq \tilde{\varepsilon}^2$$
 By Corollary~\ref{cor:to quote}, and  our choices of  $\tilde{N}, \, \tilde{k}$, (see~\ref{constants:delta})  this establishes case 2. 
\end{proof}

\subsection{Case  3 (we assume neither of the conditions in Case 1 or in Case 2 holds)}
We say that $n$ is \emph{good for reduction} if 
$$\sum_{\{(i,c,\gamma)\in  \mathcal{F}_{n}( \frac {\tilde{N}} 2,\tilde{\varepsilon}^{4})\colon \sigma_i>  \frac {\tilde{N}} 2\}}c<\tilde{\varepsilon}^4$$ 
and we say $n$ is \emph{bad for reduction} if 
$$\sum_{\{(i,c,\gamma)\in
\mathcal{F}_{n}(\tilde{N},\tilde{\varepsilon}^2)\colon \sigma_i> \tilde{N}\}}c>\tilde{\varepsilon}.$$

 By our assumption on $\tilde{k}$ and the estimate for sufficiently large $\tilde{k}$ given in~\eqref{eq:plenty good}, we have: 
\begin{lem}\label{lem:plenty nice} Let 
$G=\{\vec{i}\in \mathbb{Z}^M\colon 
i_j \text{ is good  for reduction for at least }\frac M 2 \text{ choices of }  j\}$, 
then $\sum_{\vec{i}\in G} \prod_{j=1}^M \alpha_{i_j}^{(\tilde{k})}\geq \frac 1 2$.
\end{lem}

\begin{lem}\label{lem:still pretty bad} If $j $  is bad for reduction and  $k$ and $m$ are good for reduction, then 
$$j-k+m \in \mathcal{H}_{\tilde{N},\tilde{\varepsilon}^4}.$$
Similarly, if $j$ and $m$ are good for reduction and $k$ is bad for reduction, then 
$$j-k+m \in \mathcal{H}_{\tilde{N},\tilde{\varepsilon}^4}.$$
\end{lem}
 Note that we separate the roles of the terms $k$ and $m$ to make it easier to apply the lemma (see Corollary~\ref{cor:good to bad}).  
\begin{proof}
We establish the first claim, as the second is similar.
First, if $j$ is bad for reduction, then by Claim~\ref{claim:i-in-J} we have that $j \in \mathcal{H}_{\tilde{N},\frac{\tilde{\varepsilon}^3}{9999}}$.
Taking $\mathcal{A}_j$ and $G_j$ as in Notation~\ref{notation:friends}, we have that when $\bx\in \mathcal{A}_j$ there exists $0< |\ell|\leq 3$ such that $(T^\ell T^j\bx)_i=(T^j(G_j\bx))_i$ for all $i\leq \tilde{N}$. 
Recall that the sets $D_a$ are defined in~\eqref{def:Dk}. 
Now if $(T^{n+\ell+j} \bx )_i \neq (T^{n+j} (G_j\bx))_i$ for some $i\leq \tilde{N}$, then there exists $a>\tilde{N}$ such that either $S^bT^{\ell+j} \bx\in D_a$ for $|b|\leq 3|n|\leq |\zeta_{T^{\ell+j}\bx}(b)|$ or $S^b T^j G_j\bx\in D_a$ for $|b|\leq 3|n|\leq |\zeta_{T^{j}G_j\bx}(b)|$. This uses Corollary~\ref{cor:friends to offset}.  If $\sigma_n\leq \frac {\tilde{N}} 2$, then  the measure of such points is at most 
$$4 \cdot 3n \sum_{a>\tilde{N}}\mu(D_a)<2^{-\frac {\tilde{N}} 2 }.$$ 

Let $d\in \mathbb{Z}$, $(n,A,\rho) \in \mathfrak{F}_d(\frac {\tilde{N}} 2,\tilde{\varepsilon}^4)$, and $\bx\in \mathcal{A}_j \cap A$ (which implies that $G_j(\bx) \in A$). If 
  $(T^{\ell+j} \bx )_i = (T^{j} (G_j\bx))_i$ for all $i\leq \tilde{N}$ but $(T^dT^\ell T^j \bx )_i \neq 
   (T^d T^{j} (G_j\bx))_i$ 
  for some $i\leq \tilde{N}$, then $\bx  \in \BR_r$ for some $r$. 
  Since $k$ and $m$ are good for reduction, by iterating Lemma~\ref{lem:cylinders containing bad} when $d=k$ or $m$,  we have that the measure  of the set of such points is at most $40\tilde{\varepsilon}^4$. 
Combining these two estimates and considering  $(n,A,\rho)\in \mathfrak{F}_{d}(\frac {\tilde{N}} 2 ,\tilde{\varepsilon}^4)$ with $\sigma_n>\frac {\tilde{N}} 2$, we obtain that 
$$j-k+m \in \mathcal{H}_{\tilde{N},\frac{\tilde{\varepsilon}^3}{9999}-2\cdot (\tilde{\varepsilon}^4+40\tilde{\varepsilon}^4+2^{-\frac {\tilde{N}} 2})}.$$ 
By the assumptions~\eqref{eq:later epsilon bound} and~\eqref{eq:another bound}  on $\tilde{N}$ and $\tilde{\varepsilon}$, 
the lemma follows. 
\end{proof}

\begin{cor}\label{cor:good to bad}
Assume $\sum_{\ell=1}^Mi_\ell$ is  good  for reduction and $j_\ell$ is such that $j_\ell=i_\ell$ except at one place where $i_\ell$ is good for reduction and $j_\ell$ is bad for reduction.  
Then
$\sum_{\ell=1}^Mj_\ell\in\mathcal{H}_{N,\varepsilon^4 }$.  Similarly if $j_\ell=i_\ell$ except at one place where $i_\ell$ is bad for reduction and $j_\ell$ is good for reduction, then  $\sum_{\ell=1}^Mj_\ell\in\mathcal{H}_{\tilde{N},\tilde{\varepsilon}^4}$.
\end{cor}
\begin{proof} We prove the first case and the second is similar. For concreteness we assume that $j_1\neq i_1$. 
So $\sum_{\ell=1}^M j_\ell=j_1-i_1+\sum_{\ell=1}^Mi_\ell$ satisfies the assumptions of Lemma~\ref{lem:still pretty bad},  completing the proof. 
\end{proof}

Set $A=\{i\colon i \text{ is good for reduction}\}$, set $B=\{i\colon  i \text{ is bad for reduction}\}$ 
and set $C=\mathbb{Z}\setminus (A\cup B)$.  
We define two closely related partitions $\mathcal{P}$ and $\mathcal{P'}$ of $\mathbb{Z}^M$. 
We index the partition  elements of both $\mathcal{P}$ and $\mathcal{P}'$ by elements of $A^a\times B^b\times C^c$, 
where $a,b,c\geq 0$ and $a+b+c=M-1$.
Given some triple $(\vec x,\vec y,\vec z)\in A^a\times B^b\times C^c$, let
$P_{(\vec{x},\vec{y},\vec{z})}$ be
the set of all 
$M$-tuples $(i_1,\ldots,i_M)$ such that there exists $e_1<\ldots<e_a$, $f_1<\ldots<f_b$, $g_1<\ldots<g_c$ 
satisfying that for each $\ell$ in the allowed ranges,  $i_{e_\ell}=x_\ell$, $i_{f_\ell}=y_\ell$ and $i_{g_\ell}=z_\ell$. Moreover if $j$ is the unique element of $\{1,\ldots,M\}\setminus \{e_1,\ldots,e_a,f_1,\ldots, f_b,g_1,\ldots, g_c\}$ then $j>e_a$ and $i_j \in A$.  Set
$$\mathcal{P}=\{P_{(\vec{x},\vec y,\vec z)}:(\vec x,\vec y,\vec z)\in A^a\times B^b\times C^c \text{ where }0\leq a,b,c \text{ and }a+b+c=M-1\}.$$

We give another way to describe this.  
The partition  elements  are 
 subsets of $\mathbb{Z}^M$  so that each $M$-tuple has $a+1$ terms in $A$, $b$ terms in $B$ and $c$ terms in $C$. 
  We fix all the terms that are in 
   $B$ and the order they come in relative to the other terms that are in $B$ (but not relative to the terms that are in $A$ and $C$), and similarly for $C$. For $A$, 
we fix all but the last term that is in $A$ that appears and their order relative to the other terms of $A$ 
(but not relative to the terms of $B$ and $C$). The last term that is in $A$ which appears is allowed to be any element of $A$. 

We now define $\mathcal{P}'$ by switching the roles of $A$ and $B$.
That is, we define 
$P'_{(\vec{x},\vec{y},\vec{z})}$, 
 to be 
the set of all $(i_1,\ldots,i_M)$ such that there exists $e_1<\ldots<e_a$, $f_1<\ldots<f_b$, $g_1<\ldots<g_c$ such that for each $\ell$ in the allowed ranges,  $i_{e_\ell}=x_\ell$, $i_{f_\ell}=y_\ell$, and $i_{g_\ell}=z_\ell$. 
Moreover if $j$ is the unique element of $\{1,\ldots,M\}\setminus \{e_1,\ldots,e_a,f_1,\ldots,f_b,g_1,\ldots,g_c\}$, then $j>f_b$ and $i_j \in B$.  Set
$$\mathcal{P}'=\{P'_{(\vec{x},\vec y,\vec z)}:(\vec x,\vec y,\vec z)\in A^a\times B^b\times C^c \text{ where }0\leq a,b,c \text{ and }a+b+c=M-1\}.$$

\begin{lem}\label{lem:one bad} 
For any $a+b+c=M-1$, $(\vec{x},\vec{y},\vec{z}) \in A^{a}\times B^b\times C^c$, 
if any element of $P_{(\vec x,\vec y, \vec z)}$ is  good   for reduction, then no element of $P'_{(\vec x,\vec y, \vec z)}$ is  good   for reduction, 
 and the analogous statement holds when the roles of $P_{(\vec x,\vec y, \vec z)}$ and $P_{(\vec x,\vec y, \vec z)}'$ are exchanged.  
\end{lem}

\begin{proof} Let $(i_1,\ldots,i_M), (i_1',\ldots,i_M') \in P_{(\vec x,\vec y, \vec z)}\cup P_{(\vec x,\vec y, \vec z)}'$. There exists a permutation $\pi$ such that $i_j=i_{\pi(j)}'$ except for at most one $j$.  Moreover, if $k\in P_{(\vec x,\vec y, \vec z)}$ and $\ell \in P'_{(\vec x,\vec y, \vec z)}$, then this is a change from a good for reduction element to a bad for reduction element. 
By Corollary~\ref{cor:good to bad} if $\sum_{j=1}^M i_j$ is good for reduction, then $\sum_{j=1}^Mi_j'$ is not   good   for reduction. Thus if there exists one element in $P_{(\vec x,\vec y, \vec z)}$ that was  good   for reduction, then this argument shows every $\ell \in P_{(\vec x,\vec y, \vec z)}'$ is not  good   for reduction and similarly vice versa.
\end{proof}

We define one more partition $\mathcal{N}$ of $\mathbb{Z}^M$. If $u,v,w\geq 0$ and $u+v+w=M$, let  
$$\mathcal{N}_{u,v,w}=\bigl\{(i_1,\ldots,i_M)\colon |\{j\colon i_j\in A\}|=u, \, |\{j\colon i_j\in B\}|=v \text{ and }|\{j\colon i_j \in C\}|=w\bigr\}.$$

\begin{prop}\label{prop:2b friends}
If the conditional probability that $\sum_{j=1}^M i_j$ 
 is good for reduction given that $(i_1,\ldots,i_M)\in \mathcal{N}_{(u,v,w)}$ is greater than $\frac 1 2$, then the conditional probability that 
$\sum_{j=1}^Mi_j$ is not  good   for reduction given that $(i_1,\ldots,i_M)\in  \mathcal{N}_{(u+1,v-1,w)}$ is at least $\frac 1 2$
\end{prop}

Observe that both $\mathcal{P}$ and $\mathcal{P}'$ are finer partitions than $\mathcal{N}$.

\begin{lem}\label{lem:partition} 
For $(\vec{x},\vec{y},\vec{z}) \in A^{a}\times B^b\times C^c$, 
the conditional probability of an element in $\mathcal{N}_{(a+1,b,c)}$ is in $\mathcal{P}_{(\vec x,\vec y, \vec z)}$ is the same as the conditional probability of an element in 
$\mathcal{N}_{(a,b+1,c)}$ being in $\mathcal{P}'_{(\vec x,\vec y, \vec z)}$.  
\end{lem}
\begin{proof} Let  $u_1=\sum_{i \in A}\alpha_i^{(\tilde{k})}$, $u_2=\sum_{i \in B}\alpha_i^{(\tilde{k})}$, and $u_3=1-(a+b)=\sum_{i \in C}\alpha_i^{(\tilde{k})}$.  
The conditional probability of being in a particular $P_{(\vec x,\vec y, \vec z)}$ with $(\vec{x},\vec y ,\vec z)\in A^{a}\times B^b\times C^c$ 
 given that one lies in $\mathcal{N}_{(a+1,b,c)}$ is
\begin{equation}\label{eq:conditional prob}
\prod_{i=1}^{a} \frac{v_i}{u_1}\prod_{i=1}^b \frac{w_i}{u_2} \prod_{i=1}^c \frac{x_i}{u_3}.
\end{equation}
This is also the conditional probability of being in $P_{(\vec x,\vec y, \vec z)}'$ given that one is in $\mathcal{N}_{(a,b+1,c)}$. 
\end{proof}

Combining Lemmas~\ref{lem:partition} and~\ref{lem:one bad},  Proposition~\ref{prop:2b friends} follows.

If $(\Omega,\mathbb{P})$ is a probability space and $H\colon \Omega\to \{0,1,\ldots\}$ is $\mathbb{P}$ measurable, we say $i$ is \emph{$(H,\delta)$-spread} if 
$$\max\{\mathbb{P}(H^{-1}(i+1)),\mathbb{P}(H^{-1}(i-1))\}>\delta \mathbb{P}(H^{-1}(i)).$$ 
We say $H$ is \emph{$\delta$-spread} if $\mathbb{P}(\bigcup_{i, \,(H,\delta)\text{-spread}}H^{-1}(i))>\delta$. 

\begin{lem}\label{lemma:adjacent comparable} There exists $C$ such that
 if $F_i\colon (\Omega,\mathbb{P})\to \{0,1\}$ are independent, identically $\mathbb{P}$ distributed  random variables 
 satisfying $\frac{\delta}K\leq \mathbb{P}(F_i^{-1}(0))\leq 1-\frac\delta K$, 
then  
 $H(\omega)=\sum_{i=1}^KF_i(\omega)$ is $\min\{\frac{\delta^2}C,\frac 1 C\}$-spread. 
 \end{lem}

 \begin{proof}
 If $C> {2^{18}}$ and $\delta \leq 9 $, $0$ is $(H, \min\{\frac 1 C, \frac {\delta^2} C\})$-spread and  has a definite probability of occurring, which proves the lemma. 
Thus we assume $\delta> 9$. 
 Let $p=\mathbb{P}(F_i^{-1}(0))$.  
  Due to the symmetry, we can assume that $p\leq \frac{1}{2}$, 
 and by the assumption on $\delta$, we can assume that $p > \frac{9}{K}$.  
  Thus \begin{align*}
  \frac{\mathbb{P}(\{\omega\colon\sum_{i=1}^KF_i(\omega)=n+1\})}{\mathbb{P}(\{\omega\colon\sum_{i=1}^KF_i(\omega)=n\})}  & = 
\frac{\binom{K}{n+1}(p)^{n+1}(1-p)^{K-n-1}}{\binom{K}{n}(p)^{n}(1-p)^{K-n}}
{\binom{K}{n}(p/K)^{n}(1-(p/K))^{K-n}} \\ &= \frac{K-n-1}{n-1}  p  (1- p )^{-1}.
\end{align*}
If $n\in [\frac 1 3 Kp,\frac 5 3Kp]$, then this is greater than 
$\min\{\frac 1 {99},\frac 1 {99}Kp\}$.  Since $\frac 1 {99}Kp\geq \frac 1 {99} \delta$ the result follows if at  least half of the $\omega$ lie in this range. 
To check this, note that we have 
$$\int_\Omega (\sum_{i=1}^K F_i(\omega)-p)^2\,d \mathbb{P} = \int _{\Omega}\sum F_i(\omega)^2 d\mathbb{P}=K\Big((1-p)^2(p)+(-p)^2(1-p)\Big)< \frac 3 2 p K.$$
Note that the first inequality uses that $F_i(\omega)-p$ and $F_j(\omega)-p$ are independent and have integral 0 for all $i\neq j$.
Thus by Chebyshev's inequality, $\mathbb{P}(\{\omega\colon |\sum_{i=1}^K F_i(\omega)-Kp|>2\sqrt{pK}\})<\frac 1 2.$ Since $pK\geq9$ we have that $2\sqrt{pK}\leq \frac 2 3 Kp$ establishing the necessary condition. 
 \end{proof}

  \begin{proof}[Concluding the proof of Theorem~\ref{thm:prime} in Case 3]
 In this proof only, we introduce some terminology for clarity: we say $i$ is \emph{decisive} if it is either good or bad for reduction. 
 Let $\gamma^{(\tilde{k})}$ be the probability measure on $\{0,1\}$ defined by 
 $$\gamma^{(\tilde{k})}(\{0\})=\frac{\sum_{i \text{ is bad for reduction}}\alpha_i^{(\tilde{k})}}{\sum_{i \text{ is decisive}}\alpha_i^{(\tilde{k})}}$$ 
 and $$\gamma^{(\tilde{k})}(\{1\})=\frac{\sum_{i \text{ is good for reduction}}\alpha_i^{(\tilde{k})}}{\sum_{i \text{ is \text{decisive}}}\alpha_i^{(\tilde{k})}}.$$ 
Note that $\gamma^{(\tilde{k})}(\{0\})$ is the conditional probability that 
$i$ is \emph{bad} for reduction given that it is decisive and 
 $\gamma^{(\tilde{k})}(\{1\})$ is the conditional probability that $i$ is \emph{good} for reduction given that it is decisive. 
As we are not in Case 2, it follows that 
$\gamma^{(\tilde{k})}(\{0\})>\frac{\tilde{\varepsilon}^2}M$. 
Thus by Lemma~\ref{lemma:adjacent comparable}, $\sum_{j=0}^{n}\gamma^{(\tilde{k})}$ is at least $\frac{\tilde{\varepsilon}^4}{64}$ spread. 
 We partition $\mathbb{Z}^M$ into sets $\mathcal{N}_{(a,b,c)} \cup \mathcal{N}_{(a+1,b-1,c)}$ where $a$ is even. 
 This gives rise to partitions of these elements into $P_{(\vec x,\vec y, \vec z)}$ and $P'_{(\vec x , \vec y , \vec z)}$. 
  By Lemma~\ref{lem:one bad}, for each  $(\vec x,\vec y, \vec z)$ 
  one of  $P_{(\vec x,\vec y, \vec z)}$ or $P'_{(\vec x , \vec y , \vec z)}$  
  contain no   good for reduction elements. 
  By Lemma~\ref{lem:partition} and the fact that $\sum_{i=0}^{M-c-1}\gamma^{(\tilde{k})}$ is $\tilde{\varepsilon}^4$ spread (so long as $c<\frac M 8$)
  it follows from Lemma~\ref{lem:plenty nice} that at least  $\frac 1 2 \tilde{\varepsilon}^8$ of the points in $\mathcal{N}_{(a,b,c)} \cup \mathcal{N}_{(a+1,b-1,c)}$ are not good  for reduction.  Once again this contradicts Corollary~\ref{cor:to quote} and our choices.  \end{proof}

\end{document}